\documentclass[12pt, leqno]{article}
\usepackage{amsmath,amsfonts,amsthm,amscd,amssymb,verbatim} 

\numberwithin{equation}{section}

\theoremstyle{plain}
\newtheorem{thm}{Theorem}[section]
\newtheorem{defn}[thm]{Definition}
\newtheorem{prop}[thm]{Proposition}
\newtheorem{lem}[thm]{Lemma}
\newtheorem{cor}[thm]{Corollary}

\theoremstyle{definition}
\newtheorem{rem}[thm]{Remark}

\renewcommand{\b}{\bullet}
\newcommand{\beast}{\begin{eqnarray*}}
\newcommand{\east}{\end{eqnarray*}}

\newcommand{\N}{{\Bbb N}}
\newcommand{\Z}{{\Bbb Z}}
\newcommand{\Q}{{\Bbb Q}} 
\newcommand{\R}{{\Bbb R}}
\newcommand{\C}{{\Bbb C}}

\newcommand{\Spf}{{\mathrm{Spf}}\,}
\newcommand{\Spm}{{\mathrm{Spm}}\,}

\newcommand{\lra}{\longrightarrow}

\newcommand{\hra}{\hookrightarrow}

\newcommand{\lla}{\longleftarrow}

\newcommand{\wt}[1]{\widetilde{#1}}

\newcommand{\ol}[1]{\overline{#1}}
\newcommand{\os}{\overset}

\newcommand{\conv}{{\mathrm{conv}}}

\newcommand{\Hom}{{\mathrm{Hom}}}

\newcommand{\Ker}{{\mathrm{Ker}}}
\newcommand{\im}{{\mathrm{Im}}}
\newcommand{\Coker}{{\mathrm{Coker}}}

\newcommand{\id}{{\mathrm{id}}}

\newcommand{\End}{{\mathrm{End}}}

\newcommand{\res}{{\mathrm{res}}}

\newcommand{\dlog}{{\mathrm{dlog}}\,}
\newcommand{\Ext}{{\mathrm{Ext}}}
\newcommand{\an}{{\mathrm{an}}}

\newcommand{\cE}{{\cal E}}
\newcommand{\cF}{{\cal F}}

\newcommand{\cO}{{\cal O}}
\newcommand{\cP}{{\cal P}}

\newcommand{\cU}{{\cal U}}

\newcommand{\fX}{{\mathfrak{X}}}
\newcommand{\fY}{{\mathfrak{Y}}}

\newcommand{\LNM}{{\mathrm{LNM}}}
\newcommand{\ULNM}{{\mathrm{ULNM}}}
\newcommand{\type}{{\mathrm{type}}}
\newcommand{\NID}{{\mathrm{(NID)}}}
\newcommand{\NLD}{{\mathrm{(NLD)}}}
\newcommand{\PTD}{{\mathrm{(PTD)}}}

\newcommand{\sing}{{\mathrm{sing}}}

\newcommand{\0}{{\bold 0}}
\newcommand{\1}{{\bold 1}}
\renewcommand{\k}{{\bold k}}
\newcommand{\vv}{{\bold v}}
\newcommand{\w}{{\bold w}}
\newcommand{\const}{{\mathrm{const}}}
\renewcommand{\sp}{{\mathrm{sp}}}

\renewcommand{\wt}{\widetilde}

\renewcommand{\res}{{\mathrm{res}}}
\renewcommand{\d}{\dagger}

\newcommand{\lam}{\lambda}

\newcommand{\pa}{\partial}
\newcommand{\Mat}{{\mathrm{Mat}}}
\newcommand{\e}{{\bold e}}
\newcommand{\f}{{\bold f}}
\newcommand{\Exp}{{\mathrm{Exp}}}

\renewcommand{\End}{{\mathrm{End}}}

\newcommand{\dsum}{\displaystyle\sum}
\newcommand{\I}{{\bold i}}

\begin{document}
\title{On logarithmic extension of overconvergent isocrystals}
\author{Atsushi Shiho
\footnote{
Graduate School of Mathematical Sciences, 
University of Tokyo, 3-8-1 Komaba, Meguro-ku, Tokyo 153-8914, JAPAN. 
E-mail address: shiho@ms.u-tokyo.ac.jp \,
Mathematics Subject Classification (2000): 12H25 (primary), 14F30 (secondary)}}
\date{}
\maketitle

\begin{abstract}
In this paper, we establish a criterion for an overconvergent isocrystal 
on a smooth variety over a field of characteristic $p>0$ to extend 
logarithmically to its smooth compactification 
whose complement is a strict normal crossing divisor. 
This is a generalization of a result of Kedlaya, who treated the case of 
unipotent monodromy. Our result is regarded as 
a $p$-adic analogue of the theory of canonical extension of 
regular singular integrable connections on smooth varieties of characteristic 
$0$. 
\end{abstract}

\tableofcontents

\section*{Introduction}

Let $K$ be a complete discrete valuation field of mixed characteristic 
$(0,p)$ with ring of integers $O_K$ and 
residue field $k$, and let $X \hra \ol{X}$ be an open 
immersion of smooth $k$-varieties such that $Z = \ol{X}-X$ is a 
simple normal crossing divisor. Denote the log structure on $\ol{X}$ 
associated to $Z$ by $M$. 
The purpose of this paper is to 
give a criterion for an overconvergent isocrystal on $(X,\ol{X})/K$ 
to extend to an isocrystal on log convergent site 
$((\ol{X},M)/O_K)_{\conv}$. This criterion is established by 
Kedlaya in \cite{kedlayaI} in the case of unipotent monodromy. 
In this paper, we generalize his result to the case of arbitrary 
monodromy (although we need certain `$p$-adic non-Liouvilleness 
assumption'.) \par 
Our result can be regarded as a narutal $p$-adic analogue of the theory 
of canonical extension of regular singular integrable connections on 
algebraic varieties developped in \cite{deligne}. So first we 
give a brief review of this theory. \par 
First let us consider the local situation. Let $K$ be a discrete valuation 
field of equal characterisic $0$ and denote the ring of integers by 
$O_K$. Let us fix a norm $|\cdot |$ corresponding to the valuation of 
$K$ and fix a uniformizer $t$ of $O_K$. Assume we are given 
a derivation $d:K \lra Kdt = K\dlog t$ on $K$. Then we can define 
the notion of a connection $\nabla: E \lra E\,\dlog t$ on 
a finite dimensional vector space $E$ over $K$. Let us assume given such
 $(E,\nabla)$ and let $\pa$ be the composite map 
$E \os{\nabla}{\lra} E\dlog t \os{=}{\lra} \allowbreak 
E$, where the second map sends 
$e\dlog t$ to $e$\,($e \in E$). 
Then $(E,\nabla)$ is called regular singular if 
the spectral norm $|\pa|_{\sp} := \lim_{n\to\infty}|\pa^n|^{1/n}$ 
(where $|\pa^n|$ denotes the operator norm of $\pa^n$) of $\pa$ is 
equal to $1$($= |d|_{\sp}$) (\cite[II 1.9, 1.11]{deligne}). 
It is known (\cite[II 1.12]{deligne}) 
that $(E,\nabla)$ is regular singular if and only if 
it comes from a log connection $\nabla: E_0 \lra E_0 \dlog t$ on 
a finitely generated free $O_K$-module $E_0$, that is, it 
extends to $O_K$ logarithmically. \par 
In global situation, we can define the notion of regular singularity 
in the following way: Let $j: X \hra \ol{X}$ be an open immersion of 
smooth $\C$-varieties such that $Z=\ol{X}-X$ is a simple normal crossing 
divisor. Let us denote the log structure on $\ol{X}$ 
associated to $Z$ by $M$. 
Then a locally free $\cO_X$-module $E$ endowed with an 
integrable connection $\nabla: E \lra E \otimes \Omega^1_{X/\C}$ 
is called regular singular if, for any generic point $\eta$ of $Z$, 
the restriction of $(E,\nabla)$ to 
$(j_*\cO_{X})_{\eta}$($= \C(X)$ endowed with the discrete valuation 
associated to $\eta$) is regular singular
(\cite[II 4.2, 4.4]{deligne}). Let us assume given such 
$(E,\nabla)$. Then, if $(E,\nabla)$ is unipotent 
(that is, the monodromy representation associated to $(E,\nabla)$ is 
unipotent along each irreducible component of $Z$), 
it extends uniquely to an integrable log connection on $(\ol{X},M)$ which 
has nilpotent residues (\cite[II 5.2]{deligne}). In general case, 
if we fix a section $\tau:\C/\Z \lra \C$ of the canonical projection, 
we can extend $(E,\nabla)$ uniquely to an integrable log connection on 
$(\ol{X},M)$ whose exponents along each irreducible component of $Z$ 
are in $\tau(\C/\Z)$ (\cite[II 5.4]{deligne}). \par 
Now let us return to the $p$-adic situation. 
First we consider the local situation. Let $K$ be a complete 
non-Archimedean valued field with ring of integers $O_K$ and 
residue field $k$ of characteristic $p>0$. For $\lam \in 
\{0\} \cup p^{\Q_{<0}}$, let $A^1_K[\lam,1)$ be the rigid analytic 
annulus with radius in $[\lam,1)$ over $K$. Consider a 
locally free module $E$ of finite rank on $A^1_K[\lam,1)$ endowed with 
a connection $\nabla: E \lra E\,dt$ (where $t$ is the coordinate of 
the annulus $A^1_K[\lam,1)$) and define $\pa$ by the composite map 
$E \os{\nabla}{\lra} Edt \os{=}{\lra} E$, where the second map 
sends $e\,dt$ to $e$ \,($e \in E$). 
For $\rho \in [\lam,1)\cap p^{\Q}$, Let us denote the spectral 
norm on $\Gamma(A^1_K[\rho,\rho],\cO)$ by $|\cdot|_{\rho}$. 
We call that $(E,\nabla)$ satisfies the Robba condition if 
the spectral norm $|\pa|_{\sp,\rho} = \lim_{n\to\infty}|\pa^n|^{1/n}_{\rho}$ 
of $\pa$ at $\rho$ is equal to the spectral norm 
$|d|_{\sp,\rho}$ of $d$(=derivation on $\cO_{A^1_K[\lam,1)}$) at $\rho$ 
for any $\rho\in [\lam,1)\cap p^{\Q}$. (See \cite[p.136, 11.1]{cmsurvey}.) 
For such $(E,\nabla)$ of rank $\mu$, one can define the notion of 
exponent in the sense of Christol-Mebkhout, which is an element of 
$\Z_p^{\mu}$ modulo certain equivalence relation 
(see \cite[11.5]{cmsurvey} for detail). 
The Robba condition is considered as a $p$-adic analogue of the notion 
of regular singularity. Indeed, it is known 
(\cite[6-2.6]{cmII}, \cite[12.1]{cmsurvey} plus some calculation) 
that if $(E,\nabla)$ satisfies the Robba condition and its exponent 
(in the sense of Christol-Mebkhout) has $p$-adically non-Liouville 
difference (see \cite[10.8]{cmsurvey} for definition), $(E,\nabla)$ extends to 
a locally free module with log connection on $A_K^1[0,1)$. \par 
Let us consider now the global situation: Let $K,O_K,k$ be as above and 
assume that $K$ is discretely valued. 
Let $X \hra \ol{X}$ be an open immersion of 
smooth $k$-varieties such that $Z=\ol{X}-X=\bigcup_{i=1}^rZ_i$ 
is a simple normal crossing 
divisor. Let us denote the log structure on $\ol{X}$ 
associated to $Z$ by $M$. We will define, for an overconvergent 
isocrystal $\cE$ on $(X,\ol{X})/K$, the notion of the Robba condition. 
(See Definition \ref{defrobba} 
for detail.) Roughly speaking, we call that $\cE$ satisfies 
the Robba condition if, for any generic point $\eta$ of $Z$, 
we can take certain annulus $A^1_L[\lam,1)$ over 
$\Spm L$ (which is a field defined by using certain lift around 
$\eta$) such that the connection on $A^1_L[\lam,1)$ induced by 
$\cE$ satisfies the Robba condition. We call $\cE$ is $\NLD$ when 
the exponents (in the sense of Christol-Mebkhout) 
of the connections on the annuli have $p$-adically 
non-Liouville difference for any $\eta$. 
For an overconvergent isocrystal $\cE$ on $(X,\ol{X})/K$ satisfying 
the Robba condition which is $\NLD$, we can define the set 
$\ol{\Sigma}_{\cE}$ of exponents of $\cE$, which is a subset of 
$(\Z_p/\Z)^r$. (See \S 3 for detail.) \par 
With this terminology, we can state our main result as follows: Let 
$\cE$ be 
an overconvergent isocrystal on $(X,\ol{X})/K$ satisfying 
the Robba condition which is $\NLD$. Then, if we fix a section 
$\tau: (\Z_p/\Z)^r \lra \Z_p^r$ of the canonical projection, 
we can extend $\cE$ uniquely to an isocrystal on log convergent 
site $((\ol{X},M)/O_K)_{\conv}$ whose exponents along $Z$ are contained 
in $\tau(\ol{\Sigma}_{\cE})$. This is a $p$-adic analogue of 
the logarithmic extension of integrable connections which we 
explained above. We would like to remark here that 
this result is proved by Kedlaya (\cite{kedlayaI}) in the case where 
$\cE$ has unipotent monodromy (the case $\ol{\Sigma}_{\cE} = \{0\}$ and 
$\tau(0) = 0$). \par 
The strategy of the proof is basically the same as that in 
\cite{kedlayaI}, which we briefly explain here. 
By the assumption of the Robba condition and the condition $\NLD$, 
$\cE$ extends logarithmically to the discs $A^1_L[0,1)$ at each generic 
point of $Z$. 
First we prove that this condition 
implies the log-extendability of $\cE$ to relative disc over  
$]Z^0[$, where $Z^0$ is the smooth locus of $Z$ and $]\cdot[$ denotes 
certain rigid analytic space defined by using some local lift of $Z^0$. 
We prove this by using the proposition called `generizarion' (see 
Proposition \ref{3.4.3}). Next we prove the log-extendability of 
$\cE$ to certain relative disc over a strict neighborhood of $]Z^0[$ in 
$]Z[$, by using the proposition called `overconvergent generization'
 (see Proposition \ref{3.5.3}). Then we extend $\cE$ logarithmically to 
the boundary step by step, along each irreducible component of $Z$. \par 
Let us explain the content of each section. 
In Section 1, we prove basic properties on 
log-$\nabla$-modules on rigid spaces. We introduce the notion of 
log-$\nabla$-modules with exponents in $\Sigma$ and 
the notion of $\Sigma$-unipotent log-$\nabla$-modules ($\Sigma \subseteq 
\ol{K}^r$ for some $r$), which are 
the generalization of log-$\nabla$-modules with nilpotent residues and 
unipotent log-$\nabla$-modules introduced in \cite[\S 3]{kedlayaI}. 
In Section 2, we prove the three key propositions which we need for 
the proof of our main result. The first one, called generization, asserts 
that the property of $\Sigma$-unipotence is `generic on the base'. 
The second one, called overconvergent generization, asserts that 
the property of $\Sigma$-unipotence can be extended `overconvergently 
on the base'. The third one asserts that, under the condition of 
log-convergence, the property of `having 
exponents in $\Sigma$' implies the 
$\Sigma$-unipotence. They are also generalization of the 
corresponding results of 
Kedlaya in \cite{kedlayaI}, but we would like to point out here that 
our results in this section are slightly different from those in 
Kedlaya partly because Kedlaya's argument in 
\cite[3.5.3]{kedlayaI} and \cite[3.6.2]{kedlayaI}
seems to contain an error\footnote{After putting the first version of 
this paper to ArXiv, Kedlaya informed me how to correct the errors in 
the paper \cite{kedlayaI}, which is different from the argument in 
this paper. His correction is written in \cite[Appendix A]{kedlayaIV}.} 
and partly because of the technical reason. 
In Section 3, we prove the main result of this paper, using 
the results in the previous section. We introduce the notion of 
an isocrystal on log convergent site with exponents in $\Sigma$ and 
the notion of an overconvergent isocrystal having $\Sigma$-unipotent 
monodromy. Also we introduce the notion of an overconvergent isocrystal 
satisfying the Robba condition and the property $\NLD$, and check that 
an overconvergent isocrystal 
satisfying the Robba condition and the property $\NLD$ has 
$\Sigma$-unipotent monodromy (for certain $\Sigma$), 
by using generization. Then we prove that, under some condition on 
$\Sigma$, an overconvergent isocrystal having $\Sigma$-unipotent 
monodromy can be extended uniquely to 
an isocrystal on log convergent site with exponents in $\Sigma$, 
by using the three key propositions proved in the previous section. \par 
A part of this work was done during the author's stay at Universit\'e 
de Rennes I. The author would like to thank to Pierre Berthelot 
for giving me an opportunity to stay there and to the members there 
for the hospitality. 
The author is partly supported by Grant-in-Aid for Young Scientists (B), 
the Ministry of Education, Culture, Sports, Science and Technology of 
Japan and JSPS Core-to-Core program 18005 whose representative is 
Makoto Matsumoto. \par

\section*{Convention}
(1) \,\, Throughout this paper, $k$ is a field of characteristic 
$p>0$ and a $k$-variety means a reduced separated scheme of finite type 
over $k$. \\
(2) \,\, Throughout this paper, $K$ is a field of characteristic $0$ 
complete with respect to the non-Archimedean absolute value 
$|\cdot|: K \lra \R_{\geq 0}$ with residue field $k$ of 
characteristic $p>0$. Let us put 
$\Gamma^* := \sqrt{|K^{\times}|} \cup \{0\}$ and let 
$O_K$ be the the ring of integers of $K$. In Section 3, 
we will assume moreover 
that the absolute value $|\cdot|$ is discretely valued. 
We denote the algebraic closure of $K$ by $\ol{K}$. 
For a $p$-adic formal scheme $P$ topologically of finite type over $O_K$, 
we denote the associated rigid space by $P_K$. \\
(3) \,\, We use freely the notion 
concerning isocrystals on log convergent site and overconvergent 
isocrystals. For the former, see \cite{shiho2}, \cite{shiho3} and 
\cite[\S 6]{kedlayaI}. For the latter, see \cite{berthelotrig} and 
\cite[\S 2]{kedlayaI}.

\section{Log-$\nabla$-modules and $\Sigma$-unipotence} 

In this section, first we recall the notion of (log-)$\nabla$-modules and 
the residues of log-$\nabla$-modules, 
following \cite[2.3]{kedlayaI}. 
Then we recall the notion of exponents of the residues, following 
\cite{bc} and introduce the notion of $\Sigma$-unipotence for 
log-$\nabla$-modules for a subset 
$\Sigma$ of $\ol{K}^r$ (for some $r$): It is a generalization 
of the notion of unipotence for log-$\nabla$-modules introduced in 
\cite[3.2.5]{kedlayaI}. After that, we prove some basic properties 
on log-$\nabla$-modules with exponents in $\Sigma$ and 
$\Sigma$-unipotent log-$\nabla$-modules, which are 
generalization of the results in \cite[3.2]{kedlayaI}. \par 
First let us recall the definition of (log-)$\nabla$-modules. 
For a rigid space $X$ over $K$, let $\Omega^1_{X/K}$ be the 
sheaf of continuous differentials on $X$ over $K$ and for a morphism 
$f:X \lra Y$ of rigid spaces over $K$, let us put 
$\Omega^1_{X/Y}:=\Omega^1_{X/K}/f^*\Omega^1_{Y/K}$. 

\begin{defn}[{\cite[2.3.4]{kedlayaI}}] 
Let $f:X \lra Y$ be a morphism of rigid spaces over $K$. A 
$\nabla$-module on $X$ relative to $Y$ is a coherent $\cO_X$-module 
$E$ endowed with an integrable $f^{-1}\cO_Y$-linear connection 
$\nabla: E \lra E \otimes_{\cO_X} \Omega^1_{X/Y}$. In the case 
$Y=\Spm K$, we omit the term `relative to $Y$'. 
\end{defn} 

It is known (\cite[2.2.3]{berthelotrig}) that when $X$ is a smooth 
rigid space, any $\nabla$-module on $X$ (relative to $\Spm K$) is 
automatically locally free. \par 
Next, when we are given a morphism $f:X \lra Y$ of rigid spaces over $K$ 
and let $x_1, ..., x_r$ be elements in $\Gamma(X,\cO_X)$. Then we define 
the module of logarithmic differentials $\omega^1_{X/Y}$ (denoted by 
$\Omega^{1,\log}_{X/Y}$ in \cite{kedlayaI}) by 
$$\omega^1_{X/Y} := (\Omega^1_{X/Y} \oplus \bigoplus_{i=1}^r \cO_X \cdot 
\dlog x_1)/N,$$
where $N$ is the sheaf locally generated by 
$(dx_i,0)-(0,x_i\dlog x_i)$ $(1 \leq i \leq r)$. 

\begin{defn}[{\cite[2.3.7]{kedlayaI}}]\label{lognabdef}
Let $f:X \lra Y$ be a morphism of rigid spaces over $K$ 
and let $x_1, ..., x_r$ be elements in $\Gamma(X,\cO_X)$. 
Then a 
log-$\nabla$-module on $X$ with respect to $x_1, ..., x_r$ 
relative to $Y$ is a locally free $\cO_X$-module 
$E$ endowed with an integrable $f^{-1}\cO_Y$-linear log connection 
$\nabla: E \lra E \otimes_{\cO_X} \omega^1_{X/Y}$. In the case 
$Y=\Spm K$, we omit the term `relative to $Y$'. 
\end{defn} 

Note that, in the case of log-$\nabla$-modules, the local freeness is 
not automatic even in the case $Y=\Spm K$. So it is built 
in the definition. \par 
Let $f:X \lra Y, x_1,...,x_r$ be as in Definition \ref{lognabdef} and 
let $(E,\nabla)$ be a log-$\nabla$-module on 
$X$ with respect to $x_1, ..., x_r$ relative to $Y$. 
For $1 \leq i \leq r$, let us put $D_i:=\{x_i=0\} \subseteq X$ and 
$M_i:={\rm Im}(\Omega^1_{X/Y} \oplus \bigoplus_{j\not=i}\cO_X \dlog x_j \lra 
\omega^1_{X/Y})$. Then the composite map 
$$ 
E \os{\nabla}{\lra} E \otimes_{\cO_X} \omega^1_{X/Y} \lra 
E \otimes_{\cO_X} (\omega^1_{X/Y}/M_i) \cong 
E \otimes_{\cO_X} \cO_{D_i} \dlog x_i \cong E \otimes_{\cO_X} \cO_{D_i}
$$ 
naturally induces an element $\res_i$ in 
$\End_{\cO_{D_i}}(E \otimes_{\cO_X} \cO_{D_i})$. We call it the 
residue of $(E,\nabla)$ along $D_i$. \par 
Keep the above notation and assume that $Y=\Spm K$, $X$ is a smooth 
affinoid rigid space and that the zero loci of $x_i$ are smooth and 
meet transversally. Then, by the argument of \cite[1.5.3]{bc}, 
we see that there exists a 
polynomial $P_i(x) \in K[x] \,(1 \leq i \leq r)$ satisfying 
$P_i(\res_i)=0$. Take $P_i(x)$ to be minimal and monic satisfying 
$P_i(\res_i)=0$. Then we call the roots of $P_i(x)$ (in $\ol{K}$) 
the exponents of $(E,\nabla)$ along $D_i$. When $X$ is not necessarily 
affinoid, the set of 
exponents of $(E,\nabla)$ is defined as the union of the sets of the 
exponents of $(E,\nabla)|_U$ for affinoid admissible opens 
$U \subseteq X$. \par 
Let $X$ be a smooth 
rigid space endowed with $x_1, ...,x_r \in \Gamma(X,\cO_X)$ 
such that the zero loci of $x_i$ are smooth and 
meet transversally. 
For a subset $\Sigma = \prod_{i=1}^r\Sigma_i$ in 
$\ol{K}^r$ ($\Sigma_i \subseteq \ol{K}$), we define the category 
$\LNM_{X,\Sigma}$ as the category of log-$\nabla$-modules 
on $X$ with respect to $x_1,...,x_r$ such that 
all the exponents along $D_i$ are contained in $\Sigma_i$ $(1 \leq i 
\leq r)$. \par 
As in \cite{kedlayaI}, we call a subinterval $I \subseteq [0,\infty)$ 
aligned if any endpoint of $I$ at which it is closed is contained in 
$\Gamma^*$. We call a subinterval $I \subseteq [0,\infty)$ quasi-open 
when it is open at non-zero endpoints. For an aligned interval $I$, we 
define the rigid space $A^n_K(I)$ by  
$A^n_K(I) := \{(t_1, ..., t_n) \in {\Bbb A}^{n,\an}_K 
\,\vert\, \forall i, |t_i| \in I \}.$ \par 
Following \cite[3.2.4]{kedlayaI}, we use the following convention: 
For a smooth 
affinoid rigid space $X$ 
endowed with $x_1, ...,x_r \in \Gamma(X,\cO_X)$ 
such that the zero loci of $x_i$ are smooth and 
meet transversally, we put 
$\omega^1_{X \times A^n_K[0,0]/K} := 
\omega^1_{X/K} \oplus \bigoplus_{i=1}^n\cO_X \dlog t_i$, where 
$\dlog t_i$ is the free generator `corresponding to 
the $i$-th coordinate on $A^n_K[0,0]$'. Using this, we can define the 
notion of a log-$\nabla$-module $(E,\nabla)$ 
on $X \times A_K^n[0,0]$ with respect to 
$x_1, ..., x_r, t_1, ..., t_n$ and the notion of the residue, the exponents 
of $(E,\nabla)$ along $\{t_i=0\}$ in natural way: 
To give a log-$\nabla$-module $(E,\nabla)$ 
on $X \times A_K^n[0,0]$ with respect to 
$x_1, ..., x_r, t_1, ..., t_n$ is equivalent to give a 
log-$\nabla$-module $(E,\nabla)$ 
on $X$ with respect to 
$x_1, ..., x_r$ and commuting endomorphisms $\pa_i := t_i 
\dfrac{\pa}{\pa t_i}$ of $(E,\nabla)$ 
$(1 \leq i \leq n)$. Also, we can define the 
category $\LNM_{X \times A^n_K[0,0],\Sigma}$ for $\Sigma = 
\prod_{i=1}^{r+n}\Sigma_i \subseteq \ol{K}^{r+n}$ as above: 
A log-$\nabla$-module $(E,\nabla)$ 
on $X \times A_K^n[0,0]$ with respect to 
$x_1, ..., x_r, t_1, ..., t_n$, regarded as a 
log-$\nabla$-module $(E,\nabla)$ 
on $X$ with respect to 
$x_1, ..., x_r$ endowed with commuting endomorphisms $\pa_i := t_i 
\dfrac{\pa}{\pa t_i}$ $(1 \leq i \leq n)$, is in the category 
$\LNM_{X \times A^n_K[0,0],\Sigma}$ if and only if $(E,\nabla)$ is 
in $\LNM_{X,\prod_{i=1}^r \Sigma_i}$ and 
all the eigenvalues of $\pa_i$ (which are a priori 
contained in $\ol{K}$) are 
in $\Sigma_{r+i}$ $(1 \leq i \leq n)$. \par 
For an aligned interval $I \subseteq [0,\infty)$ and 
$\xi := (\xi_1, ..., \xi_n) \in K^n$, we define the log-$\nabla$-module 
$(M_{\xi},\nabla_{M_{\xi}})$ on $A^n_K(I)$ with respect to 
$t_1, ..., t_n$ (which are the coordinates) 
as the log-$\nabla$-module 
$(\cO_{A_K^n(I)},d + \sum_{i=1}^n \xi_i \dlog t_i)$. 
Now we introduce the notion of (potential) $\Sigma$-constance and 
(potential) $\Sigma$-unipotence for 
log-$\nabla$-modules: 

\begin{defn}[cf. {\cite[3.2.5]{kedlayaI}}]\label{unipdef}
Let $X$ be a smooth rigid space endowed with 
$x_1, \allowbreak ..., x_r \in \Gamma(X,\cO_X)$ whose zero loci are smooth and 
meet transversally. Let $I \subseteq [0,\infty)$ 
be an aligned interval and fix $\Sigma := \prod_{i=1}^{r+n}\Sigma_i \subseteq 
\ol{K}^{r+n}$. \\
$(1)$ \,\,\, 
An object $(E,\nabla)$ in 
$\LNM_{X \times A_K^n(I),\Sigma}$ 
$($where $X \times A_K^n(I)$ is endowed with $x_1, ..., x_r, \allowbreak 
t_1, ..., t_n$, 
where $t_i$'s are the 
coordinates in $A_K^n(I))$ is called $\Sigma$-constant if 
$(E,\nabla)$ has the form 
$\pi_1^*(F,\nabla_F) \otimes \pi_2^*(M_{\xi},\nabla_{M_{\xi}})$ for some 
$(F,\nabla_F) \in \LNM_{X,\prod_{i=1}^r\Sigma}$ and $\xi \in 
\prod_{i=r+1}^{r+n}(\Sigma_i \cap K)$, where 
$\pi_1: X \times A_K^n(I) \lra X$, $\pi_2:X \times A_K^n(I) \lra A_K^n(I)$ 
denote the projections. \\
$(2)$ \,\,\, 
An object $(E,\nabla)$ in 
$\LNM_{X \times A_K^n(I),\Sigma}$ is called $\Sigma$-unipotent if 
$(E,\nabla)$ admits a filtration 
$$ 0 = E_0 \subset E_1 \subset \cdots \subset E_m=E $$ 
by sub log-$\nabla$-modules whose successive quotients are 
$\Sigma$-constant log-$\nabla$-modules. \\
$(3)$ \,\,\, An object $(E,\nabla)$ in $\LNM_{X \times A^n_K(I),\Sigma}$ 
is called potentially $\Sigma$-constant $($resp. potentially 
$\Sigma$-unipotent$)$ if it is $\Sigma$-constant $($resp. 
$\Sigma$-unipotent$)$ after some finite extension of the field $K$. \\

We denote the category of $\Sigma$-unipotent $($resp. potentially 
$\Sigma$-unipotent$)$ log-$\nabla$-modules on 
$X \times A_K^n(I)$ with respect to $x_1, ..., x_r, t_1, ..., t_n$ by 
$\ULNM_{X \times A^n_K(I),\Sigma}$ $($resp. 
$\ULNM'_{X \times A^n_K(I),\Sigma})$. 
\end{defn}

\begin{rem}\label{bar}
Let the notation be as above and assume that $r=0$ and that $I$ does not 
contain $0$. In this case, it is easy to see that the log-$\nabla$-modules 
$M_{\xi}$ and $M_{\xi'}$ $(\xi,\xi' \in \Sigma)$ are isomorphic if 
$\xi-\xi'$ is contained in $\Z^n$. From this fact, it is easy to see that 
the notion of (potential) $\Sigma$-unipotence only depends on the image 
$\ol{\Sigma}$ of $\Sigma$ in $(\ol{K}/\Z)^n$ in the following sense: 
An object $(E,\nabla)$ in 
$\LNM_{X \times A_K^n(I),\Sigma}$ is (potentially) 
$\Sigma$-unipotent if and only if 
it is $\tau(\ol{\Sigma})$-unipotent for some (or any) section 
$\tau: (\ol{K}/\Z)^n \lra \ol{K}^n$ of the canonical projection 
$\ol{K}^n \lra (\ol{K}/\Z)^n$. So, in this case, we will say also that 
$(E,\nabla)$ is (potentially) 
$\ol{\Sigma}$-unipotent, by abuse of terminology. 
\end{rem}

Following \cite[3.2.6]{kedlayaI}, we introduce an important functor 
$\cU_I$: 

\begin{defn} 
Let $X,I,\Sigma$ be as in Definition \ref{unipdef}. 
Then we define the functor 
$\cU_I: \allowbreak \LNM\allowbreak {}_{X\times A^n_K[0,0],\Sigma} \allowbreak \lra 
\LNM_{X\times A^n_K(I),\Sigma}$ as follows$:$ 
For an object in $\LNM_{X\times A^n_K[0,0],\Sigma}$ regarded as 
$((E,\nabla), \allowbreak \{\pa_i\}_{i=1}^n)$ $((E,\nabla) \in 
\LNM_{X,\prod_{i=1}^r\Sigma_i})$ is sent by $\cU_I$ to 
$\pi^*E$ $($where $\pi:X \times A^n_K(I) \lra X$ denotes the projection$)$ 
endowed with the connection 
$$ \vv \mapsto \pi^*\nabla(\vv) + \sum_{i=1}^n\pi^*(\pa_i)(\vv) \dlog t_i. $$
\end{defn}

We have the following remark, which is parallel to \cite[3.2.7]{kedlayaI}: 

\begin{rem}\label{3.2.7}
Let the notation be as above. 
\begin{enumerate}
\item 
In the case $X= \Spm K$ ($r=0$) and $I=[0,0]$, a log-$\nabla$-module on 
$X \times A^n_K(I)$ is nothing but a finite dimensional $K$-vector 
space $E$ endowed with commuting endomorphisms $\pa_i$ ($1 \leq i \leq n$). 
It is (potentially) $\Sigma$-constant 
if and only if each $\pa_i$ is equal to the 
multiplication by $\xi_i$ for some $\xi_i \in \Sigma_i \cap K$. 
It is $\Sigma$-unipotent (resp. potentially $\Sigma$-unipotent) 
if and only if each $\pa_i$ has eigenvalues in $\Sigma_i \cap K$ (resp. 
$\Sigma_i$). So, in this case, 
$\Sigma$-constance is equivalent to potential $\Sigma$-constance and 
$\Sigma$-unipotence (resp. 
potential $\Sigma$-unipotence) is equivalent to being contained 
in the category $\LNM_{X \times A^n_K(I),\Sigma \cap K^n}$ (resp. 
$\LNM_{X \times A^n_K(I),\Sigma}$). 
\item 
From the remark (1), we easily see that, in the case $X=\Spm K$ $(r=0)$ 
and $I$ arbitrary, an object in 
the image of $\cU_I$ is potentially $\Sigma$-unipotent and it is 
$\Sigma$-unipotent if $\Sigma \subseteq K^n$. 
\end{enumerate}
\end{rem}

Before investigating the properties of log-$\nabla$-modules with 
exponents in $\Sigma$ and (potentially) 
$\Sigma$-unipotent log-$\nabla$-modules, 
we consider several conditions on the set $\Sigma$. To do this, first 
we recall the notion of the type of a number in $\ol{K}$ and 
that of $p$-adic non-Liouvilleness: 

\begin{defn} $(1)$ \,\,\, 
For $a \in \ol{K}$, define the type $\type(a) \in [0,1]$ of $a$ by 
the radius of convergence of the series 
$\dsum_{s \in \N,\not=a} \dfrac{X^s}{a-s}.$ \\
$(2)$ \,\,\,  
$a \in \ol{K}$ is called $p$-adically non-Liouville if 
$\type(a)=\type(-a)=1$ holds. 
\end{defn} 

\begin{defn} 
A set $\Sigma \subseteq \ol{K}$ is called $\NID$ $($resp. $\PTD$, 
$\NLD)$ if, for any $\alpha, \beta \in \Sigma$, $\alpha - \beta$ is not 
a non-zero integer $($resp. the type of $\alpha - \beta$ is positive, 
$\alpha -\beta$ is $p$-adically non-Liouville$)$. A set 
$\Sigma := \prod_{i=1}^r\Sigma_i \subseteq \ol{K}^r$ is called 
$\NID$ $($resp. $\PTD$, $\NLD)$ if so is $\Sigma_i$ for all 
$1 \leq i \leq r$. \par 
We say a set $\ol{\Sigma} := \prod_{i=1}^r \ol{\Sigma}_i 
\subseteq (\ol{K}/\Z)^r$ is $\PTD$ 
$($resp. $\NLD)$ if, for some $($or any$)$ section $\tau: (\ol{K}/\Z)^r 
\lra \ol{K}^r$ of the canonical projection, $\tau(\ol{\Sigma})$ is $\PTD$ 
$($resp. $\NLD)$. Note that $\tau(\ol{\Sigma})$ is automatically 
$\NID$. 
\end{defn}

Now we prove the analogue of \cite[3.2.8]{kedlayaI}, which is a starting point 
of the study of log-$\nabla$-modules with exponents in $\Sigma$: 

\begin{lem}\label{3.2.8}
Let $X$ be a smooth rigid space endowed with  
$x_1, ..., x_r \allowbreak \in \Gamma(X,\cO_X)$ whose zero loci are smooth and 
meet transversally. Let $\Sigma := \prod_{i=1}^{r+s+1}\Sigma_i$ be a 
subset of $\ol{K}^{r+s}$ such that $\Sigma_{r+1}$ is $\NID$ and 
$\PTD$, let $a \in (0,\infty) \cap \Gamma^*$ and let $(E,\nabla)$ be an object 
in $\LNM_{X \times A^1_K[0,a] \times A^s_K[0,0],\Sigma}$ 
such that $E |_{X \times \{0\}}$ is free. 
Then there exists an element $b \in (0,a] \cap \Gamma^*$ such that 
the restriction of $E$ to 
$X \times A^1_K[0,b] \times A^s_K[0,0]$ is in the essential image 
of $\cU'_{[0,b]}: \LNM_{X\times A^1_K[0,0] \times A^s_K[0,0],\Sigma} \lra 
\LNM_{X \times A^1_K[0,b] \times A^s_K[0,0],\Sigma}$, where $\cU'_{[0,b]}$ 
is the functor naturally induced by 
$\cU_{[0,b]}: \LNM_{X\times A^1_K[0,0],\prod_{i=1}^{r+1} \Sigma_i} \lra 
\LNM_{X \times A^1_K[0,b],\prod_{i=1}^{r+1}\Sigma_i}$. 
\end{lem} 

\begin{proof} 
We give a proof, following \cite[3.2.8]{kedlayaI}. 
For a commutative ring $A$ and $n \in \N$, we denote the set of $n \times n$ 
matrices with entries in $A$ by $\Mat_n(A)$. Also, we denote the 
coodinate of $A^1_K[0,a]$ by $t$ and denote the differential operator 
$t\dfrac{\pa}{\pa t}$ simply by $\pa$. \par 
Let us take a basis $\e_1, ..., \e_n$ of $E|_{X \times \{0\}}$. 
Then, as in the proof of \cite[3.2.8]{kedlayaI}, we see that 
a lift of this basis (which we denote again by $\e_1, ..., \e_n$) 
forms a basis of $E$ on $X \times A^1_K[0,a]$ if we replace $a$ by a 
smaller value. Define the matrix $N = (N_{jl}) \in 
\Mat_n(\cO(X \times A^1_K[0,a]))$ by 
$\pa(\e_l) = \sum_{j=1}^nN_{jl}\e_j$ and write 
$N = \sum_{i=0}^{\infty}N_it^i$, $N_i \in \Mat_n(\cO(X))$. \par 
First we prove the lemma in the case $s=0$. 
The key claim for the proof is the following: \\
\quad \\
{\bf claim.} \\ 
(1) \,\, 
There exists a unique matrix $M \in \Mat_n(\cO(X)[[t]])$ congruent to 
identity modulo $t$ satisfying $NM + \pa M = MN_0$. \\
(2) \,\, The matrix $M$ in (1) converges on $X \times A^1_K[0,b]$ for some 
$b \in (0,a] \cap \Gamma^*$. \\
\quad \\
Let us prove the assertion (1) of the claim. If we write 
$M = \sum_{i=0}^{\infty}M_it^i$ with $M_0$ identity, we should have 
\begin{equation}\label{nmmn}
iM_i + N_0M_i - M_iN_0 = -\sum_{j=0}^{i-1}N_{i-j}M_j 
\end{equation} 
for $i>0$. 
Let $g: \Mat_n(\cO(X)) \lra \Mat_n(\cO(X))$ be the $\cO(X)$-linear map 
defined by $g(Y) = N_0Y-YN_0$. If we know that $g + i\cdot \id$ is invertible
 for $i>0$, we can uniquely determine $M_i$'s $(i\geq 0)$ inductively 
by the equation \eqref{nmmn}. So it suffice to prove that 
$g + i\cdot \id$ $(i>0)$ is invertible. To prove this, we may enlarge 
$K$. \par 
Since $N_0$ is the residue of $E$ along $\{t=0\}$, there exists a polynomial 
$P(x) \in K[x]$ of the form $P(x) = \prod_{i=1}^m(x-\xi_i)^{n_i}$ with 
$\xi_i \in \Sigma_{r+1}$ such that $P(N_0)=0$ holds. By enlarging $K$, we 
may assume that $\xi_i$'s belong to $K$. Then we have a decomposition 
$$\cO(X)^n = E|_{X \times \{0\}} = \bigoplus_{k=1}^mE_k$$ 
such that 
$N_0 - \xi_k \cdot \id$ acts on $E_k$ nilpotently. According to this 
decomposition, we have the decomposition 
$$ \Mat_n(\cO(X)) = \End_{\cO(X)}(E|_{X \times \{0\}}) = 
\bigoplus_{k,l=1}^m \Hom_{\cO(X)}(E_k,E_l), $$ 
and on each $\Hom_{\cO(X)}(E_k,E_l)$, $g$ is defined by 
$g(Y) = (N_0-\xi_l\cdot \id)(Y) - Y(N_0-\xi_k\cdot \id) + 
(\xi_l-\xi_k)Y$. Hence 
$g-(\xi_l-\xi_k)\cdot \id = (g+i\cdot\id)- (\xi_l-\xi_k-i)$ acts nilpotently 
on $\Hom_{\cO(X)}(E_k,E_l)$. Since $\Sigma_{r+1}$ is $\NID$, 
$\xi_l-\xi_k-i$ is non-zero for any $k,l$. It implies that $g+i\cdot\id$ is 
invertible. So we have proved the assertion (1). \par 
Let us prove the assertion (2) of the claim. Since $N$ is in 
$\Mat_n(\cO(X \times A^1_K[0,a]))$, there exists a constant $C>1$ satisfying 
$|N_i| \leq Ca^{-i}$. (Here $|\phantom{a}|$ denotes any fixed Banach norm on 
$X$ andd $|\phantom{a}|$ for a matrix is defined as the 
maximum of the norms of its entries.) Let us define $\cO(X)$-linear maps 
$h_{1,i}, h_2: \Mat_n(\cO(X)) \lra \Mat_n(\cO(X))$ $(i>0)$ by the composite 
maps 
$$ \Mat_n(\cO(X)) = \bigoplus_{k,l=1}^m \Hom_{\cO(X)}(E_k,E_l) 
\lra \bigoplus_{k,l=1}^m \Hom_{\cO(X)}(E_k,E_l) = \Mat_n(\cO(X)), $$ 
where the arrow in the middle is defined by 
$(Y_{kl})_{k,l} \mapsto ((\xi_l-\xi_k+i)Y_{kl})_{k,l}$ in the case of 
$h_{1,i}$ and 
$(Y_{kl})_{k,l} \mapsto ((g-(\xi_l-\xi_k)\id)(Y_{kl}))_{k,l}$ in the case of 
$h_2$. Then we have $g+i\cdot\id = h_{1,i}+h_2$ and $h_{1,i}$ (resp. $h_2$) is 
invertible (resp. nilpotent).
Let us take a natural number $e$ satisfying $h_2^e = 0$. Then the 
inverse of the map $g+i\cdot\id$ is given by 
$(g+i\cdot\id)^{-1} = \sum_{j=0}^{e-1}(-1)^j(h_{1,i}^{-1})^{j+1}h_2^j$. \par 
Let us replace the constant $C$ by a larger 
value so that it satisfies the inequalities 
$|h_{1,i}^{-1}| \leq C \max_{k,l}|(\xi_l-\xi_k+i)^{-1}|, 
|h_2| \leq C$ and let us put 
$$ A_i := \prod_{j=1}^i\max(\max_{k,l}|(\xi_l-\xi_k+j)^{-1}|, 1). $$
With these notation, we prove the inequality $|M_i| \leq A_i^eC^{2ei}a^{-i}$ 
by induction on $i$. It is obviously true for $i=0$. If the inequality 
is true for $i-1$, we have 
\begin{align*}
|(g+i\cdot\id)(M_i)| & = 
|\sum_{j=0}^{i-1}N_{i-j}M_j| \\ 
& \leq 
\max_{0\leq j \leq i-1} (Ca^{-i+j}\cdot A_j^eC^{2ej}a^{-j}) 
\leq A_{i-1}^eC^{2e(i-1)+1}a^{-i}. 
\end{align*}
Hence we have 
\begin{align*}
|M_i| & = 
|\sum_{j=0}^{e-1}(-1)^j(h_{1,i}^{-1})^{j+1}h_2^j((g+i\cdot\id)(M_i))| \\ 
& \leq 
\max_{0 \leq j \leq e-1}
(C^{j+1}(\max_{k,l}|(\xi_l-\xi_k+i)^{-1}|)^{j+1}C^{j}) \cdot 
A_{i-1}^eC^{2e(i-1)+1}a^{-i} \\ 
& \leq 
A_i^{e}C^{2ei}a^{-i}, 
\end{align*}
and so the desired inequality is proved. \par 
By \cite[VI, Lemma 1.2]{dgs}, we have an equality of the series 
$$ 
\sum_{s=0}^{\infty}\dfrac{X^s}{\prod_{i=1}^s(i-\alpha)} = 
\alpha e^{X} \sum_{s=0}^{\infty} \dfrac{(-X)^s}{s!}\cdot 
\dfrac{1}{\alpha-s}
$$ 
for $\alpha \in \ol{K}-\N$. From this, we easily see that the radius of 
convergence of the left hand side is positive if $\type(\alpha)$ is positive. 
Putting $\alpha = \xi_k-\xi_l$ (which is in $\ol{K}-\N$ with positive type 
since $\Sigma_{r+1}$ is $\NID$ and $\PTD$), we see that, for some $\rho>0$, 
the inequality $A_i < \rho^{-i}$ holds for all $i>0$. Then, if we take 
$b \in (0,a]\cap\Gamma^*$ to be a number satisfying $0<b<\rho^eC^{-2e}a$, 
we have $|M_i| < b^{-i}$ for all $i>0$. Hence $M$ converges on 
$X \times A^1_K[0,b]$ and so the assertion (2) of the claim is proved. \par 
By the claim, the matrix $M = (M_{jl})$ is an invertible matrix in 
$\Mat_n(\cO(X \times A^1_K[0,b]))$. Define the vectors 
$\vv_1, ..., \vv_n \in \Gamma(X \times A^1_K[0,b],E)$ by 
$\vv_l=\sum_j M_{jl}\e_j$. Then $\vv_1, ...,\vv_n$ forms a basis of $E$ on 
$X \times A^1_K[0,b]$ and the operator $\pa$ acts on this basis via the 
matrix $N_0$. Also, by the formal power series computation, we see that 
the $\cO_X$-span of this basis (which we denote by $F$) 
is characterized as the subspace of $E$ 
on which $P(\pa)$ acts as zero. Hence the restriction of $(E,\nabla)$ to 
$F$ naturally defines a log-$\nabla$-module on $X$ whose exponents along 
$\{x_i=0\}$ are in $\Sigma_i$ $(1\leq i \leq r)$, 
endowed with the operator $\pa=N_0$. That is, $F$ defines an object in 
$\LNM_{X \times A^1_K[0,0],\Sigma}$ and it is easy to see that 
$E = \cU_{[0,b]}(F)$ holds. 
Hence the proof of the lemma is finished in the case $s=0$. \par 
Next let us treat the general case. Let us rewrite $(E,\nabla)$ as 
$((E,\nabla'),\{\pa_i\}_{i=1}^s)$, where $(E,\nabla') \in 
\LNM_{X \times A^1_K[0,a],\prod_{i=1}^{r+1}\Sigma_i}$ and 
$\pa_i \,(1 \leq i \leq s)$ are commuting endomorphisms on $(E,\nabla')$ 
with eigenvalues in $\Sigma_{r+i+1}$. Then we can apply the argument 
up to the previous paragraph to $(E,\nabla')$: So let us take 
$\pa, b, N_0$, the basis $\vv_1,...,\vv_n \in 
\Gamma(X \times A^1_K[0,b],E)$ and 
$F$ as above, and let $L_i = \sum_{j=0}^{\infty}L_{ij}t^j \,(1 \leq i \leq s)$ 
be the matrix expression of $\pa_i$ with respect to the basis 
$\vv_1,...,\vv_n$. Since $\pa_i$'s are compatible with $\pa$, 
we have equalities $N_0L_i + \pa L_i = L_iN_0 \,(1 \leq i \leq s)$, that is, 
the equalities 
$$ jL_{ij} + N_0 L_{ij} - L_{ij} N_0 = O \,\,\,\, (1 \leq i \leq s, j \geq 0). 
$$ 
As we see above, the map $Y \mapsto jY + N_0Y - YN_0$ is invertible 
for any positive integer $j$. Hence $L_{ij} = O$ for $j>0$, that is, $\pa_i$'s 
 come from endomorphisms of $F$. So $((F,\nabla'|_F), \{\pa_i|_F\}_{i=1}^s)$ 
naturally defines an object in 
$\LNM_{X \times A^1_K[0,0] \times A^s_K[0,0],\Sigma}$ which is sent 
to $((E,\nabla'),\{\pa_i\}_{i=1}^s) = (E,\nabla)$ by $\cU'_{[0,b]}$. 
So the proof in general case is also finished. 
\end{proof}

Using Lemma \ref{3.2.8}, we obtain the following 
(cf. \cite[3.2.12]{kedlayaI}): 

\begin{lem}\label{3.2.12} 
Let $X$ be a smooth rigid space endowed with $x_1,...,x_r \in \Gamma(X,\cO_X)$ 
whose zero loci are smooth and meet transversally. 
Let $a \in (0,\infty)\cap\Gamma^*$, let 
$\Sigma= \prod_{i=1}^{r+n+m}\Sigma_i$ be a subset of $\ol{K}^{r+n+m}$ 
such that $\prod_{i=r+1}^{r+n+m}\Sigma_i$ is 
$\NID$ and $\PTD$ and let $(E,\nabla)$ be an object in 
$\LNM_{X \times A^n_K[0,a] \times A^m_K[0,0],\Sigma}$ 
such that $E|_{X \times \{0\}}$ is free. 
Then, there exists 
$b \in (0,a]\cap\Gamma^*$ such that the restriction of 
$(E,\nabla)$ to $X \times A^n_K[0,b] \times A^m_K[0,0]$ is in the essential 
image of $\cU'_{[0,b]}: \LNM_{X \times A_K^{n+m}[0,0],\Sigma} \lra 
\LNM_{X \times A_K^n[0,b] \times A^m_K[0,0],\Sigma}$, where $\cU'_{[0,b]}$ 
is the functor naturally induced by 
$\cU_{[0,b]}: \LNM_{X \times A_K^{n}[0,0],\prod_{i=1}^{r+n}\Sigma_i} \lra 
\LNM_{X \times A_K^n[0,b],\prod_{i=1}^{r+n}\Sigma_i}$. 
\end{lem}

\begin{proof}
We prove the lemma by induction on $n$. Let 
$i: X \times A_K^{n-1}[0,a] \times A^{m+1}_K[0,0] \hra 
X \times A^n_K[0,a] \times A^{m}_K[0,0]$ be the zero locus of the $n$-th 
coordinate function of $A^n_K[0,a]$. 
Then, by induction hypothesis, $i^*(E,\nabla)$ 
is in the essential image of 
$$ \cU'_{[0,b]}: \LNM_{X \times A_K^{n+m}[0,0],\Sigma} \lra 
\LNM_{X \times A_K^{n-1}[0,b] \times A^{m+1}_K[0,0],\Sigma} $$ 
for some $b$. Then, if we write 
$i^*(E,\nabla) = \cU'_{[0,b]}((E',\nabla'))$, we have the isomorphism 
$$ E|_{X \times \{0\}} = (i^*E) |_{X \times \{0\}} = 
(\pi^*E')|_{X \times \{0\}} = E',$$
where $\pi: X \times A_K^{n-1}[0,b] \lra X$ denotes the projection. 
Hence $E'$ is free and so $i^*E = \pi^*E'$ is also free. 
Hence we may apply Lemma \ref{3.2.8} and we see that 
the restriction of $(E,\nabla)$ to $X \times A^n_K[0,b] \times A^m_K[0,0]$ 
is in the essential image of 
$$ \cU'_{[0,b]}: \LNM_{X \times 
A_K^{n-1}[0,b] \times A^{m+1}_K[0,0],\Sigma} \lra 
\LNM_{X \times A_K^{n}[0,b] \times A^{m}_K[0,0],\Sigma} $$ 
for some smaller $b$. Then, by using the induction hypothesis again, 
we see that 
the restriction of $(E,\nabla)$ to $X \times A^n_K[0,b] \times A^m_K[0,0]$ 
is in the essential image of 
$$ \cU'_{[0,b]}: \LNM_{X \times A_K^{n+m}[0,0],\Sigma} \lra 
\LNM_{X \times A_K^n[0,b] \times A^m_K[0,0],\Sigma}$$
for some even smaller $b$, as desired. 
\end{proof}

Using Lemma \ref{3.2.12}, we can prove the following proposition, which is 
the analogue of \cite[3.2.14]{kedlayaI}: 

\begin{prop}\label{3.2.14}
Let $X$ be a smooth rigid space endowed with $x_1, ..., x_r \in 
\Gamma(X,\cO_X)$ whose zero loci are smooth and meet transversally. 
Let $\Sigma = \prod_{i=1}^r\Sigma_i$ be a subset of $\ol{K}^r$ which is 
$\NID$ and $\PTD$. Then, for a morphism 
$f:(E,\nabla_E) \lra (F,\nabla_F)$ 
in $\LNM_{X,\Sigma}$, $\Ker(f)$ and $\Coker(f)$ endowed with 
canonical connections are also in the category 
$\LNM_{X,\Sigma}$. In particular, $\LNM_{X,\Sigma}$ is an abelian 
category. 
\end{prop} 

\begin{proof} 
It suffices to prove that $\Ker(f)$ and $\Coker(f)$ are locally free. 
So we may replace $K$ by a finite extension and it suffices to check the 
local freeness on a neighborhood of each point of $X$ 
(see \cite[3.2.13]{kedlayaI}). Hence, by \cite[2.3.3]{kedlayaI} 
(see also the proof of \cite[3.2.14]{kedlayaI}, \cite[1.18]{kiehl}), 
we may assume $X=A^n_K[0,a]$ for some $n \geq r$ and $a \in (0,\infty)\cap
\Gamma^*$ such that $x_i$ $(1 \leq i \leq r)$ are a part of the 
coordinate of $A^n_K[0,a]$. Then, by replacing $\Sigma$ by 
$\Sigma \times \{0\}^{n-r}$, we may assume that the exponents of $E$ and 
$F$ are 
contained in $\Sigma$ and by Lemma \ref{3.2.12} (with 
$X = \Spm K$ in the notation there), we may assume that 
$E$ and $F$ are in the image of 
$\cU_{[0,a]}: \LNM_{A^n_K[0,0],\Sigma} \lra \LNM_{A^n_K[0,a],\Sigma}$. 
Moreover, by enlarging $K$, we may assume that $\Sigma \subseteq K^n$. \par 
Let $\e_1, ..., \e_l$ (resp. $\f_1, ..., \f_m$) be a basis of 
$E$ (resp. $F$) on which $\pa_i:=t_i\dfrac{\pa}{\pa t_i}$ acts via 
matrices over $K$. We can assume moreover that there exists an element 
$\xi := (\xi_1,...,\xi_r) \in \Sigma$ such that $\pa_i(\e_1)=\xi_i\e_1$. 
Let us consider the following claim:\\
\quad \\
{\bf claim.} \,\, 
$f(\e_1)$ is contained in the $K$-span of $\f_1, ..., \f_m$. \\
\quad \\
If this is true, we can prove the local freeness of 
$\Ker(f)$ and $\Coker(f)$ in the same way as 
\cite[3.2.14]{kedlayaI}, by induction on the rank of $E\oplus F$. 
Hence it suffices to prove the above claim. \par 
Let us put $f(\e_1) = \sum_{i=1}^m a_i(t)\f_i$. Then, to prove the claim, 
it suffices to prove the following statement $(*)_j$ for each 
$1 \leq j \leq n$: \\
\quad \\
$(*)_j$: For $1 \leq i \leq m$, $a_i(t)$ are constant with respect to 
the variable $t_j$. \\
\quad \\
By symmetry with respect to the variables, it suffices to prove $(*)_1$. 
To prove this, we may change the basis $\f_1, ..., \f_m$ by another basis 
of the $K$-span of $\f_1, ..., \f_m$. So we may assume that there exist
elements $\eta_1, ...,\eta_k \in \Sigma_1$ and the sequence 
$0 = m_0 <  m_1 < m_2 < \cdots < m_k=m$ satisfying 
$$ \pa_1(\f_{m_i}) = \eta_i\f_{m_i} \,(1 \leq i \leq k), 
\,\,\, \pa_1(\f_{j}) = \eta_i\f_j+\f_{j+1} \,(m_i<j<m_{i+1}). $$
Let us define $\ol{\nabla}_E, \ol{\nabla}_F$ as the composite map 
\begin{align*}
E & \os{\ol{\nabla}_E}{\lra} E \otimes \omega^1_{A_K^n[0,a]/K} \lra 
E \otimes \omega^1_{A_K^n[0,a]/A_K^{n-1}[0,a]} = E\,\ol{\dlog t_1}, \\
F & \os{\ol{\nabla}_E}{\lra} F \otimes \omega^1_{A_K^n[0,a]/K} \lra 
F \otimes \omega^1_{A_K^n[0,a]/A_K^{n-1}[0,a]} = F\,\ol{\dlog t_1}, 
\end{align*}
respectively. Then we have 
\begin{equation}\label{3.2.14eq1}
f(\ol{\nabla}_E(\e_1)) = f(\xi_1\e_1\,\ol{\dlog t_1}) = 
\left(\sum_{i=1}^m\xi_1a_i(t)\f_i\right)\,\ol{\dlog t_1}. 
\end{equation} 
On the other hand, we have 
{\allowdisplaybreaks{
\begin{align}
& \phantom{=} \ol{\nabla}_F(f(\e_1)) = 
\ol{\nabla}_F\left(\sum_{i=1}^ma_i(t)\f_i\right) = 
\left(\sum_{i=1}^m \dfrac{\pa a_i}{\pa t_1}(t)t_1\f_i + 
\sum_{i=1}^m a_i(t)\pa_1(\f_i)\right)
\label{3.2.14eq2} \\ 
& = 
\sum_{i=0}^{k-1}\left\{\left( 
\dfrac{\pa a_{m_i+1}}{\pa t_1}(t)t_1 + \eta_{i+1}a_{m_i+1}(t)\right)
\f_{m_i+1} 
\right. 
\nonumber \\ & \hspace{2cm} \left. + 
\sum_{j=m_i+2}^{m_{i+1}} \left(
\dfrac{\pa a_j}{\pa t_1}(t)t_1 + \eta_{i+1}a_j(t) + 
a_{j-1}(t) \right) \f_j \right\} \ol{\dlog t_1}. \nonumber 
\end{align} }}
By comparing \eqref{3.2.14eq1} and \eqref{3.2.14eq2}, we obtain the 
equalities 
\begin{align} 
& (\xi_1-\eta_{i+1}) a_{m_i+1}(t) = \dfrac{\pa a_{m_i+1}}{\pa t_1}(t)t_1 
\,\,\, (0 \leq i \leq k-1), \label{3.2.14eq3} \\ 
& (\xi_1-\eta_{i+1}) a_{j}(t) = \dfrac{\pa a_{j}}{\pa t_1}(t)t_1 + 
a_{j-1}(t), 
\,\,\, (0 \leq i \leq k-1, m_i+2 \leq j \leq m_{i+1}). 
\label{3.2.14eq4} 
\end{align} 
Since $\xi_1 - \eta_{i+1}$ is not a non-zero integer, we see from 
\eqref{3.2.14eq3} that $a_{m_i+1}(t)$ $(0 \leq i \leq k-1)$ are 
constant with respect to $t_1$. From this, \eqref{3.2.14eq4} and 
the induction on $j$, we see that $a_j(t)$ 
$(0 \leq i \leq k-1, m_i+2 \leq j \leq m_{i+1})$ are also 
constant with respect to $t_1$. So we have proved the assertion 
$(*)_1$ and so we are done. 
\end{proof} 

\begin{cor}
Let $X$ be a smooth rigid space endowed with $x_1, ..., x_r \in 
\Gamma(X,\cO_X)$ whose zero loci are smooth and meet transversally. 
Let $\Sigma = \prod_{i=1}^{r+n}\Sigma_i$ be a subset of 
$\ol{K}^{r+n}$ such that $\Sigma = \prod_{i=1}^{r}\Sigma_i$ is 
$\NID$ and $\PTD$. Then
$\LNM_{X\times A^n_K[0,0],\Sigma}$ is an abelian 
category. 
\end{cor} 

\begin{proof} 
Immediate from Proposition \ref{3.2.14}. 
\end{proof} 

We have the following remark generalizing 
Remark \ref{3.2.7}, 
which is the analogue of \cite[3.2.16]{kedlayaI}. 

\begin{rem}\label{3.2.16} 
Let $X$ be a connected 
smooth rigid space endowed with $x_1, ...,x_r \allowbreak \in \allowbreak 
\Gamma(X,\cO_X)$ whose zero loci are smooth and meet transversally. 
Let $\Sigma:=\prod_{i=1}^{r+n}\Sigma_i$ be a subset of $\ol{K}^{r+n}$ 
such that 
$\Sigma_i$ $(1 \leq i \leq r)$ are 
$\NID$ and $\PTD$ and let $I$ be an aligned subinterval of $[0,\infty)$. 
\begin{enumerate} 
\item 
In the case $I=[0,0]$, a log-$\nabla$-module on 
$X \times A^n_K(I)$ is potentially $\Sigma$-unipotent if and only if 
it is in the category $\LNM_{X \times A^n_K[0,0],\Sigma}$. \par 
We give a proof of this assertion by induction on the rank. 
(It suffices to prove that an object in $\LNM_{X \times A^n_K[0,0],\Sigma}$ 
is potentially $\Sigma$-unipotent.) Let us take an object in 
$\LNM_{X \times A^n_K[0,0],\Sigma}$, regarded as an object 
$(E,\nabla)$ in $\LNM_{X,\prod_{i=1}^r\Sigma_i}$ 
endowed with commuting endomorphisms $\pa_i$ $(1 \leq i \leq n)$ 
such that all the eigenvalues of 
$\pa_i$ are in $\Sigma_i$ for any $1 \leq i \leq n$. 
We may enlarge $K$ in order that all the eigenvalues are in $K$. 
Since $\pa_i$ are commuting endomorphisms of $(E,\nabla)$, 
$E_1 := \bigcap_{i=1}^n\Ker(\pa_i-\xi_i)$ 
is an subobject of $(E,\nabla)$ in $\LNM_{X,\prod_{i=1}^r\Sigma_i}$ and 
it is non-zero for some $(\xi_1,...,\xi_n) \in \prod_{i=1}^n\Sigma_{r+i}$. 
It is easy to see that $E_1$ is $\Sigma$-constant, and the quotient 
$E/E_1$ is potentially $\Sigma$-unipotent by induction hypothesis. Hence 
$E$ is potentially $\Sigma$-unipotent, as desired. \par 
The same argument shows that, in the case $I=[0,0]$ and 
$\prod_{i=r+1}^{r+n} \Sigma_i \subseteq K^n$, a log-$\nabla$-module on 
$X \times A^n_K(I)$ is $\Sigma$-unipotent if and only if 
it is in the category $\LNM_{X \times A^n_K[0,0],\Sigma}$. In particular, 
$\Sigma$-unipotence is equivalant to potential $\Sigma$-unipotence 
in this case. 
\item 
From (1), we easily see that, under the assumption here 
(with $I$ arbitrary), an object in 
the essential image of the functor $\cU_I:
\ULNM'_{X \times A^n_K[0,0],\Sigma} = 
\LNM_{X \times A^n_K[0,0],\Sigma} \lra 
\LNM_{X \times A^n_K(I),\Sigma}$ is potentially $\Sigma$-unipotent, 
and it is $\Sigma$-unipotent if $\prod_{i=r+1}^{r+n}\Sigma_i \subseteq 
K^n$. 
\end{enumerate}
\end{rem}

Next we prove an important property of the functor $\cU_I$, which is 
an analogue of \cite[3.3.2]{kedlayaI}: 

\begin{prop}\label{3.3.2}
Let $X$ be a 
smooth rigid space endowed with $x_1, ...,x_r \in 
\Gamma(X,\cO_X)$ whose zero loci are smooth and meet transversally and 
let $I$ be a quasi-open subinterval of positive length in 
$[0,\infty)$. Let 
$\Sigma = \prod_{i=1}^{r+n}\Sigma_i$ be a subset of $\ol{K}^{r+n}$ which 
is $\NID$ and $\NLD$. Then, for 
any objects $E,E'$ in 
$\ULNM'_{X \times A^n_K[0,0],\Sigma}
(=\LNM_{X \times A^n_K[0,0],\Sigma}$ by Remark \ref{3.2.16}$)$, 
the canonical morphism 
$$ \Ext^i(E,E') \lra \Ext^i(\cU_I(E),\cU_I(E')) $$ 
is an isomorphism for $i=0,1$. $($Here $\Ext^i$ denotes the 
extension group in the category $\LNM_{X \times A^n_K[0,0],\Sigma}$ and 
$\LNM_{X \times A^n_K(I),\Sigma}.)$
\end{prop} 

\begin{proof}
Let us put $F = E^{\vee}\otimes E'$ (as log-$\nabla$-module on 
$X \times A^n_K[0,0]$ with respect to $x_i,...,x_r,t_1,...,t_n$). 
Then, by \cite[3.3.1]{kedlayaI}, we have 
\begin{align*}
\Ext^i(E,E') & \cong H^i(X,F \otimes \omega^{\b}_{X \times A^n_K[0,0]/K}), \\ 
\Ext^i(\cU_I(E),\cU_I(E')) & \cong 
H^i(X,\cU_I(F) \otimes \omega^{\b}_{X \times A^n_K(I)/K}). 
\end{align*}
Hence it suffices to prove that the map 
$$ 
H^i(X,F \otimes \omega^{\b}_{X \times A^n_K[0,0]/K}) \lra 
H^i(X,\cU_I(F) \otimes \omega^{\b}_{X \times A^n_K(I)/K}) $$
is an isomorphism for $i \geq 0$. 
To prove this, we can enlarge $K$. So, 
by five lemma and the exactness of 
$\cU_I$, we can reduce to the case where $E,E'$ are $\Sigma$-constant. 
Also, we can reduce to the case where $X$ is affinoid (by considering the 
\v{C}ech spectral sequence), and we can assume that $F$ is a free 
$\cO_X$-module. Hence $F$ has the form 
$F_0 \otimes \pi^*(M_{-\xi} \otimes M_{\xi'})$ for some 
$\xi,\xi' \in \prod_{i=1}^n\Sigma_{r+i}$ and for some log-$\nabla$-module 
$F_0$ on $X$ with respect to $x_1,...,x_r$ with $F_0$ free as $\cO_X$-module. 
(Here $\pi$ denotes the projection $X \times A^n_K[0,0] \lra A^n_K[0,0]$.) 
In particular, if we regard $F$ as a log-$\nabla$-module on 
$X \times A^n_K[0,0]$ with respect to $t_1,...,t_n$ relative to $X$, 
it is a finite direct sum of $\pi^*(M_{-\xi} \otimes M_{\xi'})$. \par 
Now let us consider the Katz-Oda type spectral sequence for the diagram 
$X \times A^n_K[0,0] \lra X \lra \Spm K$ and 
the diagram 
$X \times A^n_K(I) \lra X \lra \Spm K$. From the first diagram, we obtain the 
spectral sequence 
$$ 
E_2^{p,q}= H^p(\Gamma(X,\omega^{\b}_{X/K}) \otimes 
H^q(X,F\otimes\omega^{\b}_{X \times A^n_K[0,0]/X})) \,\Longrightarrow\, 
H^{p+q}(X,F \otimes\omega^{\b}_{X\times A^n_K[0,0]/K}) $$
and from the second diagram, we obtain the spectral sequence 
\begin{align*}
E_2^{p,q}& = H^p(\Gamma(X,\omega^{\b}_{X/K}) \otimes 
H^q(X\times A^n_K(I),\cU_I(F)\otimes\omega^{\b}_{X \times A^n_K(I)/X})) \\ 
& \,\Longrightarrow\, 
H^{p+q}(X\times A^n_K(I),\cU_I(F) \otimes\omega^{\b}_{X\times A^n_K(I)/K}). 
\end{align*}
Hence it suffices to prove that the map 
\begin{equation}\label{3.3.2eq1}
H^q(X,F\otimes\omega^{\b}_{X \times A^n_K[0,0]/X}) \lra 
H^q(X\times A^n_K(I),\cU_I(F)\otimes\omega^{\b}_{X \times A^n_K(I)/X})
\end{equation}  
is an isomorphism. Since the map \eqref{3.3.2eq1} depends only on the 
structure of $F$ as a log-$\nabla$-module on 
$X \times A^n_K[0,0]$ with respect to $t_1,...,t_n$ relative to $X$, 
we may suppose that $F$ has the form 
$\pi^*(M_{-\xi} \otimes M_{\xi'}) = \pi^*(M_{\alpha})$, 
where we put $\alpha_i:=\xi'_i-\xi_i, \alpha:=(\alpha_1,...,\alpha_n)$. 
In this case, the map 
\eqref{3.3.2eq1} is the map between cohmologies induced by the following 
map of complexes 
$$ g: \Gamma(X,\omega^{\b}_{X\times A^n_K[0,0]/X}) \lra 
\Gamma(X\times A^n_K(I),\omega^{\b}_{X\times A^n_K(I)/X}), $$
where the complexes $\omega^{\b}_{X\times A^n_K[0,0]/X}, 
\omega^{\b}_{X\times A^n_K(I)/X}$ are the relative 
de Rham complex associated to 
the map $d+\sum_{i=1}^n\alpha_i\dlog t_i$. Let 
$$ h: \Gamma(X\times A^n_K(I),\omega^{\b}_{X\times A^n_K(I)/X}) \lra 
\Gamma(X,\omega^{\b}_{X\times A^n_K[0,0]/X}) $$ 
be the map induced by the map $\cO(X \times A^n_K(I)) \lra 
\cO(X)$ of `taking the constant coefficient'. Then the composite 
$h \circ g$ is the identity. 
Hence it suffices to prove that the composite map 
$g \circ h$ is homotopic to the identity. We construct the 
homotopy $\varphi$ in a similar way to \cite[3.3.2]{kedlayaI}. For a monomial 
of the form 
$$ t_1^{i_1}\cdots t_n^{i_n} \dlog t_{j_1} \wedge \cdots \wedge 
\dlog t_{j_k}, \,\,\,\, 
(i_1,...,i_n \in \Z, 1 \leq j_1 < \cdots j_k \leq n), $$ 
let $l:=l(i_1,...,i_n)$ 
be the least integer with $i_l\not=0$ and send it by $\varphi$ to 
$0$ if $l \notin \{j_1,...,j_k\}$ and to 
$$ \dfrac{1}{i_l+\alpha_l}t_1^{i_1}\cdots t_n^{i_n} \dlog t_{j_1} \wedge 
\cdots \dlog t_{j_{s-1}} \wedge \dlog t_{j_{s+1}} \wedge \cdots 
\dlog t_{j_k} $$
when $l=j_s$. (It is well-defined because $\alpha_l$ is not 
a non-zero integer.) 
We claim that we can extend $\varphi$ in natural way 
to define a map 
$$ \varphi: \Gamma(X\times A^n_K(I),\omega^{k}_{X\times A^n_K(I)/X}) \lra 
\Gamma(X\times A^n_K(I),\omega^{k-1}_{X\times A^n_K(I)/X}). $$
To prove this, it suffices to prove that, for 
a formal power series 
$\sum_{\I}a_{\I}t_1^{i_1}\cdots t_n^{i_n}$ (where 
$\I=(i_1,...,i_n)$ and $a_{\I} \in \Gamma(X,\cO_X)$) 
convergent on $X \times A^n_K(I)$, so is 
the series $\sum_{\I}\dfrac{a_{\I}}{i_{l(\I)}+\alpha_{l(\I)}}t_1^{i_1}\cdots 
t_n^{i_n}$. We can prove this claim easily by using the fact that 
$\alpha_1,...,\alpha_n$ are $p$-adically non-Liouville (which follows from 
the assumption that $\Sigma$ is $\NLD$). Then we can easily check by 
formal power series computation that the map $\varphi$ gives a homotopy 
between $g \circ h$ and the identity. So we are done. 
\end{proof} 

\begin{cor}[{cf. \cite[3.3.4]{kedlayaI}}]\label{3.3.4}
Let $X, x_1, ..., x_r, I,\Sigma$ 
be as in Proposition \ref{3.3.2}. Then $\cU_I$ induces 
equivalences of categories 
$$
\ULNM_{X\times A^n_K[0,0],\Sigma} \os{=}{\lra} 
\ULNM_{X\times A^n_K(I),\Sigma}, \,\,\,\, 
\ULNM'_{X\times A^n_K[0,0],\Sigma} \os{=}{\lra} 
\ULNM'_{X\times A^n_K(I),\Sigma}. 
$$
In particular, we have the canonical equivalence 
$
\ULNM_{X\times A^n_K(I),\Sigma} = 
\ULNM'_{X\times A^n_K(I),\Sigma}$ 
if $\prod_{i=r+1}^{r+n} \Sigma_i \subseteq K^n$. 
\end{cor} 

\begin{proof} 
It is immediate from Proposition \ref{3.3.2} that the two functors 
are fully faithful. Moreover, since any $\Sigma$-unipotent object 
can be written as a successive extension by $\Sigma$-constant objects 
and since any $\Sigma$-constant object is in the essential image of 
$\cU_I$, we see that the former functor is essentially surjective, 
by using Proposition \ref{3.3.2} again. \par 
Let us prove that the latter functor is also essentially surjective. 
Take an object $(E,\nabla)$ in $\ULNM'_{X\times A^n_K(I),\Sigma}$. 
Then there exists a finite Galois extension $K \subseteq K'$ such that 
the restriction $(E',\nabla')$ of $(E,\nabla)$ to 
$X_{K'} \times A^n_{K'}(I)$ (where $X_{K'}$ is the rigid analytic space 
over $K'$ naturally induced by $X$) is $\Sigma$-unipotent. 
Then, since the functor 
$\cU_{I,K'}: 
\ULNM_{X_{K'}\times A^n_{K'}[0,0],\Sigma} \lra 
\ULNM_{X_{K'}\times A^n_{K'}(I),\Sigma}$ is an equivalence of categories, 
$(E',\nabla')$ comes from an object $(F',\nabla')$ in 
$\ULNM_{X_{K'}\times A^n_{K'}[0,0],\Sigma}$. Moreover, since 
$(E',\nabla')$ is endowed with descent data $\iota_{\sigma}: 
\sigma^*(E',\nabla') \os{=}{\lra} (E',\nabla') \,(\sigma \in 
{\rm Gal}(K'/K))$, $(F',\nabla')$ is also endowed with the descent data 
which is sent to $\{\iota_{\sigma}\}_{\sigma}$ by $\cU_{I,K'}$. 
Then $(F',\nabla')$ descends uniquely to an object $(F,\nabla)$ in 
$\ULNM'_{X \times A^n_K[0,0],\Sigma}$ which is sent to 
$(E,\nabla)$ by $\cU_I$. Hence the latter functor is also essentially 
surjective. \par 
The last assertion follows from the two equivalences of categories and 
Remark \ref{3.2.16} (1). 
\end{proof} 

\begin{cor}[{cf. \cite[3.3.6]{kedlayaI}}]\label{3.3.6}
Let $X, x_1, ..., x_r, \Sigma$ 
be as in Proposition \ref{3.3.2} and let $I$ be an aligned interval 
of positive length. Then $\cU_I$ induces fully-faithful functors 
$\ULNM^{(')}_{X\times A^n_K[0,0],\Sigma} \lra 
\ULNM^{(')}_{X\times A^n_K(I),\Sigma}.$ 
\end{cor}

\begin{proof} 
It is easy to see that the functor $\cU_I$ is faithful. On the other 
hand, if we take a quasi-open 
subinterval $J \subset I$ of positive length, 
we have the faithful restriction functor 
$\ULNM^{(')}_{X\times A^n_K(I),\Sigma} \lra 
\ULNM^{(')}_{X\times A^n_K(J),\Sigma}$ 
such that the composite functor 
$$ \ULNM^{(')}_{X\times A^n_K[0,0],\Sigma} \os{\cU_I}{\lra} 
\ULNM^{(')}_{X\times A^n_K(I),\Sigma} \lra 
\ULNM^{(')}_{X\times A^n_K(J),\Sigma}, $$ 
which is nothing but $\cU_J$, is an equivalence of categories by 
Corollary \ref{3.3.4}. From this, we see that the functor $\cU_I$ is 
full. 
\end{proof} 

The following proposition is the analogue of \cite[3.2.20]{kedlayaI}. 

\begin{prop}\label{addprop}
Let $X, x_1,...,x_r, \Sigma$ is as in Proposition \ref{3.3.2} and let $I$ be 
a quasi-open subinterval or a closed aligned subinterval of positive 
length in $[0,\infty)$. Then $\ULNM^{(')}_{X \times A^n_K(I),\Sigma}$ is 
an abelian subcategory of $\LNM_{X \times A^n_K(I),\Sigma}$. Moreover, 
for any $E \in \ULNM^{(')}_{X \times A^n_K(I),\Sigma}$, any subquotient of 
$E$ in the category $\LNM_{X \times A^n_K(I),\Sigma}$ 
belongs to the category $\ULNM^{(')}_{X \times A^n_K(I),\Sigma}$. 
\end{prop} 

\begin{proof} 
It suffices to prove the latter assertion and it suffices to work 
only in the case of $\ULNM_{X \times A^n_K(I),\Sigma}$. Take 
$E \in \ULNM_{X \times A^n_K(I),\Sigma}$ and let 
$F$ be a subquotient of 
$E$ in the category $\LNM_{X \times A^n_K(I),\Sigma}$. 
First we prove the proposition in the case $I$ is quasi-open, assuming 
that the proposition is true in the closed aligned case. Write $I$ as 
the union $I=\bigcup_{m=1}^{\infty}[a_m,b_m]$ of closed aligned intervals. 
Then $F|_{X \times A^n_K[a_m,b_m]}$ are $\Sigma$-unipotent by 
assumption. Then we can write 
$F|_{X \times A^n_K(a_m,b_m)} = \cU_{(a_m,b_m)}(G_m)$ for some 
$G_m \in \LNM_{X \times A^n_K[0,0],\Sigma}$, and the compatibility of 
$F|_{X \times A^n_K(a_m,b_m)}$'s with respect to $m$ implies that 
$G_m$ $(m \in \N)$ does not depend on $m$, which we denote by $G$. 
Then we have 
$F|_{X \times A^n_K(a_m,b_m)} = \cU_{(a_m,b_m)}(G)$ and hence 
$F= \cU_{I}(G)$. Hence $F$ is $\Sigma$-unipotent, as desired. 
So it suffices to prove the proposition in the case where $I$ is a 
closed aligned interval of positive length. \par 
To prove the proposition, it suffices to see that the subquotients of 
a $\Sigma$-constant object $E$ is again $\Sigma$-constant. 
Then it suffices to check one of the subobjects or the quotients are 
$\Sigma$-constant, by considering the dual and noting that 
$-\Sigma := 
\{(\xi_i) \,\vert\, \forall i, -\xi_i \in \Sigma_i\}$ is again 
$\NID$ and $\NLD$. \par 
First we prove that, for $n=1$ and $X=\Spm K$ $(r=0)$, 
being $\Sigma$-constant is stable by taking quotients. 
Let $E$ be a $\Sigma$-constant object and let 
$f: E \lra E'$ be a surjection in $\LNM_{A^1_K(I),\Sigma}$. 
Then, for some $\xi = (\xi_1,...,\xi_n) \in \Sigma \cap K^n$, 
$E \otimes M_{-\xi}$ is $\{0\}^n$-constant, that is, 
constant in the sense of \cite{kedlayaI} 
and $f$ induces the surjection $E \otimes M_{-\xi} \lra 
E' \otimes M_{-\xi}$ in 
$\LNM_{A^1_K(I),\{0\}^n}$. Let $F'$ be the image of 
$\{x \in E \otimes M_{-\xi} \,\vert\, \forall i, 
t_i\frac{\pa}{\pa t_i}(x)=0\}$ 
in $E' \otimes M_{-\xi}$. Then the map 
$F' \otimes_K \cO_{A^n_K(I)} \lra E' \otimes \allowbreak M_{-\xi}$ is 
suejective by construction and injective by \cite[3.2.19]{kedlayaI}. 
Hence it is an isomorphism, that is, $E' \otimes M_{-\xi}$ is 
$\{0\}^n$-constant. Hence $E'$ is $\Sigma$-constant (more precisely, 
$\prod_{i=1}^n\{\xi_i\}$-constant). \par 
Next we check that, for $n$ and $X$ arbitrary, being $\Sigma$-constant is 
stable by taking subobject. Let $E$ 
be a $\Sigma$-constant object and let us take a subobject $F \hra E$ in 
$\LNM_{X \times A^n_K(I),\Sigma}$. First we reduce to the case $n=1$ 
and $X$ arbitrary. To do this, we put $Y:=X \times A^{n-1}_K(I)$. 
Then $E$ is $\Sigma$-constant on $Y \times A^1_K(I)$. 
By the proposition in the case $n=1$, we see that $F$ is also 
$\Sigma$-constant on $Y \times A^1_K(I)$. By definition of 
$\Sigma$-constance, we have $E=\cU_I(E'), F=\cU_I(F')$ for some 
$E',F' \in \LNM_{Y \times A^1_K[0,0],\Sigma}$. 
Then, by Corollary \ref{3.3.6}, the inclusion $F \hra E$ 
comes from an inclusion $F' \hra E'$. 
Let us express 
$E', F'$ as $(E'',\pa_E), (F'',\pa_F)$, where $E'', F''$ are 
objects in 
$\LNM_{Y, \prod_{i=1}^{n-1}\Sigma_i} = \LNM_{X \times A^{n-1}_K(I), 
\prod_{i=1}^{n-1}\Sigma_i}$ and $\pa_E, \pa_F$ are endomorphisms on 
$E'', F''$, respectively. Then $E''$ is 
$\prod_{i=1}^{r+n-1} \Sigma_i$-consant and $\pa_E$ is the 
multiplication by some $\xi \in \Sigma_{r+n} \cap K$. 
Then, 
$F''$ is also $\prod_{i=1}^{r+n-1}\Sigma_i$-constant by induction hypothesis 
and 
$\pa_F = \pa_E|_F$ is also the multiplication by $\xi$. 
Hence $F' = (F'',\pa_F)$ is $\Sigma$-constant and so $F = \cU_I(F')$ is 
$\Sigma$-constant on $X \times A^n_K(I)$. \par 
Finally we prove the claim in the previous paragraph in the case 
$n=1$. Let $F \hra E$ be as above. Note first that, 
to prove the $\Sigma$-constance of $F$, we may replace $K$ by 
its finite Galois extension $K'$: Indeed, if we denote 
the restriction of $F$ to $X_{K'} \times A^1_{K'}(I)$ by $F_{K'}$, it carries 
the descent data $\iota_{\sigma}: 
\sigma^*F_{K'} \os{=}{\lra} F_{K'} \,(\sigma 
\in {\rm Gal}(K'/K))$. Then, if we have $F_{K'} = \cU_I(F_{0,K'})$ 
for some $\Sigma$-constant object $F_{0,K'}$ on 
$X_{K'} \times A^1_{K'}[0,0]$, $\iota_{\sigma}$'s descend uniquely 
to the descent data on $F_{0,K'}$, thanks to the full-faithfulness of 
$\cU_I$ (Corollary \ref{3.3.6}). Hence $F_{0,K'}$ descends to 
a $\Sigma$-constant object $F_0$ in $X_{K} \times A^1_{K}[0,0]$ satisfying 
$\cU_I(F_0) = F$, as desired. So we may assume that 
$A^1_K(I)$ admits a $K$-rational point $x$. Now let us take 
$\xi \in \Sigma_{r+1}$ such that 
$E \otimes \pi^*M_{-\xi}$ is $\prod_{i=1}^r \Sigma_i \times 
\{0\}$-constant, where 
$\pi$ denotes the projection $X \times A^1_K(I) \lra X$. 
Denote the section $X \hra X \times A^1_K(I)$ associated to $x$ 
also by $x$ and put $F' := \im (x^*(F \otimes \pi^*M_{-\xi}) \lra 
x^*(E \otimes \pi^*M_{-\xi}))$. Then we have 
$$ 
\pi^*F' \otimes M_{\xi} = 
\im (\pi^*x^*(F \otimes \pi^*M_{-\xi}) \lra 
\pi^*x^*(E \otimes \pi^*M_{-\xi})) \otimes M_{\xi}  \hra E. $$
To prove the claim, 
it suffices to prove the equality $F= \pi^*F' \otimes M_{\xi}$. 
Hence it suffices to prove that the composition maps 
$F \hra E \lra E/(\pi^*F' \otimes M_{\xi})$ and 
$\pi^*F' \otimes M_{\xi} \hra E \lra E/F$ are zero. 
To check this, we may replace $X$ by $X':=$the complement of the zero loci 
of $x_i$'s. (Then one can ignore the part `$\prod_{i=1}^r \Sigma_i$'.) 
Moreover, it suffices to check on a neighborhood of each 
point of $X'$. Hence we may assume that $X'$ is equal to a closed 
polydisc of radius $p^{-m}$ for some $m$ (after possibly enlarging 
$K$). Let $L$ be the 
completion of the fraction field of $\cO(X')$ with respect to the 
spectrum norm on $\cO(X')$. 
(This is possible because $\cO(X')$ is a reduced affinoid algebra 
whose reduction is an integral domain.) Then it suffices to check the 
vanishing of the above two maps after we pull them back to 
$A^1_L(I)$. Then it is true because we already know that $F$ is 
$\{\xi\}$-constant on $A^1_L(I)$. So we proved the 
claim and hence the proof of the proposition is finished. 
\end{proof} 

Using Corollary \ref{3.3.4} and Proposition \ref{addprop}, 
we can prove the following, which is an analogue 
of \cite[3.3.8]{kedlayaI}: 

\begin{prop}\label{3.3.8} 
Let $X$ be a 
smooth rigid space endowed with $x_1, ...,x_r \in 
\Gamma(X,\cO_X)$ whose zero loci are smooth and meet transversally and 
Let $U$ be the complement of these zero loci. 
Let 
$\Sigma = \prod_{i=1}^{r}\Sigma_i$ be a subset of $\ol{K}^{r}$ which 
is $\NID$ and $\NLD$. Let $E$ be an object in $\LNM_{X,\Sigma}$ 
and let $F$ be a sub $\nabla$-module of the restriction of $E$ to 
$U$. Then $F$ uniquely extends to a subobject $\wt{F}$ of $E$ in 
the category $\LNM_{X,\Sigma}$. 
\end{prop}

\begin{proof} 
The proof is the same as \cite[3.3.8]{kedlayaI}. 
By induction, it suffices to check the following claim: 
Let $Z$ be the zero locus of $t_r$ and suppose that we are given a 
sub log-$\nabla$-module $F$ of the restriction of $E$ to $X-Z$ 
in the category $\LNM_{X-Z,\Sigma}$. Then $F$ uniquely 
extends to a subobject $\wt{F}$ of $E$ in 
the category $\LNM_{X,\Sigma}$. \par 
Since the claim is local, we may assume that $Z, X$ are affinoid. 
Then, by \cite[1.18]{kiehl}, we may assume moreover that 
$X$ has the form $Z \times A^1_K[0,a]$ for some 
$a \in (0,\infty) \cap \Gamma^*$ and that the 
restriction of $E$ to $Z$ is free. Also, we may shrink 
$X$ in order that $X = Z \times A^1_K[0,a)$ with $a$ replaced by a 
smaller positive value. 
So, by Lemma \ref{3.2.8} and Remark \ref{3.2.16},  
we may assume that 
$E$ is potentially $\Sigma$-unipotent and $E$ 
has the form $\cU_{[0,a)}(G)$ for some $G \in 
\LNM_{Z \times A^1_K[0,0],\Sigma}$. 
Since $F$ is a subobject of 
$E|_{Z \times A^1_K(0,a)}$ in the category 
$\LNM_{Z \times A^1_K(0,a),\Sigma}$ and $E|_{Z \times A^1_K(0,a)}$ is 
potentially 
$\Sigma$-unipotent, we see by Proposition \ref{addprop} that $F$ is also 
potentially $\Sigma$-unipotent, that is, $F$ has the form $\cU_{(0,a)}(H)$ for 
some $H \in \LNM_{Z \times A^1_K[0,0],\Sigma}$. By Corollary \ref{3.3.4}, 
we see that the inclusion $F \hra E|_{Z \times A^1_K(0,a)}$ is induced by 
some injection $H \hra G$ in $\LNM_{Z \times A^1_K[0,0],\Sigma}$. 
Then, if we put $\wt{F} := \cU_{[0,a)}(H)$, it defines a 
sub log-$\nabla$-module of $E$ which extends $F$. To see the uniqueness 
of the extension, we may work locally as above. Hence we may assume that 
the extensions should be potentially 
$\Sigma$-unipotent and then the uniqueness 
follows from Corollary \ref{3.3.4}. 
\end{proof}

\section{$\Sigma$-unipotence, generization and log-convergence} 

In this section, we prove three key propositions in this paper. 
The first one, called generization, 
asserts that the property of $\Sigma$-unipotence is 
`generic on the base'. It is an analogue of \cite[3.4.1, 3.4.3]{kedlayaI}. 
The second one, called overconvergent generization, 
asserts that one can extend the the property of $\Sigma$-unipotence 
`overconvergently on the base'. It is an analogue of 
\cite[3.5.3]{kedlayaI}. The third one asserts that, under the condition 
of log-convergence, the property of 
`having exponents in $\Sigma$' implies the 
$\Sigma$-constance. It is an analogue of \cite[3.6.2, 3.6.9]{kedlayaI} 
and a variant of \cite[6.5.2]{bc} (see also \cite{bc0}, \cite{chr}). \par 
Here we explain the slight difference of the results in this section 
from the corresponding ones in 
the paper \cite{kedlayaI}. It seems to us that there is an error 
in the proof of \cite[3.5.3]{kedlayaI} and \cite[3.6.2]{kedlayaI}. 
(See Remark \ref{e} for detail.) By this reason and some technical reason, 
our analogues corresponding to them are slightly weaker than the 
results in \cite{kedlayaI} when restricted to the unipotent 
case. On the other hand, our analogue concerning 
generization is slightly more generalized than in \cite{kedlayaI}, since 
we need a slightly generalized version to prove the main result 
in the next section. \par 
First we recall the notion concerning the convergence of multisequences and 
introduce a notation: 

\begin{defn}[{\cite[2.4.1]{kedlayaI}}]
For an affinoid space $X$, a coherent $\cO_X$-module $E$ and 
$\eta \in [0,\infty)$, a multisequence $\{\vv_I\}_I$ of elements in 
$\Gamma(X,E)$ indexed by $I=(i_1,...,i_n) \in \N^n$ is called $\eta$-null 
if, for any multisequence $\{c_I\}_I$ in $K$ with $|c_I| \leq \eta^{|I|}$ 
$($where $|I| := \sum_{j=1}^ni_j)$, the multisequence $\{c_I\vv_I\}_I$ 
converges to zero in $\Gamma(X,E)$ with respect to the topology induced by 
the affinoid topology on $X$. 
\end{defn} 

\begin{defn} 
Let $X$ be an affinoid rigid space endowed with sections 
$x_1,...,x_r \in \Gamma(X,\cO_X)$, let $I \subseteq [0,\infty)$ 
be a closed aligned subinterval and let $E$ be a log-$\nabla$-module on 
$X \times A^n_K(I)$. Then, for $\xi:=(\xi_1,...,\xi_n) \in K^n$, 
we define $H^0_{X,\xi}(X \times A^n_K(I),E)$ by 
$$ H^0_{X,\xi}(X \times A^n_K(I),E) := 
\left\{\e \in \Gamma(X \times A^n_K(I),E) \,\left\vert\, 
\forall i, \left(t_i\dfrac{\pa}{\pa t_i}-\xi_i\right)(\e)=0\right.\right\}. $$
\end{defn} 

Then the key lemma for the generization is the following 
(this is the analogue of \cite[3.4.1]{kedlayaI}): 

\begin{lem}\label{3.4.1}
Let $A$ be an affinoid algebra such that $X :=\Spm A$ 
is smooth and endowed with $x_1,...,x_r \in \Gamma(X,\cO_X)$ whose zero loci 
are smooth and meet transversally and 
let $A \subseteq L$ be one of the following$:$ \\
$(1)$ \,\, $L$ is an affinoid algebra over $K$ satisfying 
the following conditions$:$ 
$\Spm L$ is smooth, the zero loci of $x_1, ..., x_r$ 
are smooth and meet transversally on $\Spm L$ 
and the supremum norm on $L$ restricts to 
the supremum norm on $A$. \\
$(2)$ \,\, $L$ is a field containing $A$ which is complete for a 
norm restricting to the supremum norm on $A$. \\
Let $I$ be a quasi-open subinterval of positive length in 
$[0,1)$ and let us write 
$\Spm L \times A^n_K(I)$ simply by $A^n_L(I)$ in the case $(1)$. 
$($In the case $(2)$, $A^n_L(I)$ is already defined as polyannulus 
over $L.)$ Let $\Sigma := \Sigma' \times \prod_{i=1}^n\Sigma_i$ be a subset of 
$\ol{K}^r \times \Z_p^n$ which is $\NID$ and $\NLD$ and 
let $E$ be an object in $\LNM_{X \times A^n_K(I),\Sigma}$ such that 
the induced object $F \in \LNM_{A^n_L(I),\Sigma}$ is 
$\Sigma$-unipotent. Then, for any closed aligned subinterval 
$[b,c]$ of $I$, there exists an element $\xi \in \prod_{i=1}^n
\Sigma_i$ such that 
$H^0_{X,\xi}(X \times A^n_K(I),E) \not= 0$. 
\end{lem} 

\begin{proof} 
Since $F$ is $\Sigma$-unipotent, we have $F = \cU_I(W)$ for some 
$W \in \LNM_{A^n_L[0,0],\Sigma}$ and $W$ is regarded as a log-$\nabla$-module 
on $\Spm L$ with respect to $x_1, ..., x_r$ 
(resp. a finite dimensional $L$-vector space) endowed with 
commuting endomorphisms $N_i := t_i\dfrac{\pa}{\pa t_i}$ with eigenvalues in 
$\Sigma_i$ $(1 \leq i \leq n)$ in the case (1) (resp. (2)). 
Let $\xi_{ik} \in \Sigma_i$ $(1 \leq k \leq n_i)$ be the eigenvalues of 
$N_i$ and let $P_i(x) = \prod_{k=1}^{n_i}(x-\xi_{ik})^{m_{ik}} \in K[x]$ 
be the minimal monic polynomial satisfying $P_i(N_i)=0$. 
Let us put $m := \max_{i,k}(m_{ik})$. 
For $\k :=(k_1,...,k_n)$ $(1 \leq k_i \leq n_i)$, let us define 
\begin{align*}
W_{\k} & := \{x \in W \,\vert\, \forall i, (N_i-\xi_{ik_i})^m(x) = 0\}, \\
W'_{\k} & := \{x \in W \,\vert\, \forall i, (N_i-\xi_{ik_i})(x) = 0\}. 
\end{align*}
Then, for some $\k$, $W'_{\k}$ is non-zero. Hence we may assume that 
$W'_{\1}$ (where $\1 := (1,...,1)$) is non-zero. Then there exist 
polynomials $Q_i(x)$ dividing $P_i(x)$ $(1 \leq i \leq n)$ such that 
the map $\prod_{i=1}^nQ(N_i): W \lra W$ is a non-zero map whose image 
is contained in $W'_{\1}$. Also, there exist $L$-linear maps 
$\varphi_{\k}:W \lra W$ ($\k = (k_1,...,k_n)$, $1 \leq k_i \leq n_i$) 
whose image is contained in $W_{\k}$ 
such that $\sum_{\k} \varphi_{\k}$ is equal to the identity map on $W$. \par 
Now define the sequence of operators $D_l$ $(l \in \N)$ on $E$ by 
$$ D_l := 
\prod_{i=1}^{n}\left\{ Q_i(t_i\dfrac{\pa}{\pa t_i}) 
\prod_{k=1}^{n_i}\prod_{j=1}^l 
\left( 
\dfrac{j-(t_i\frac{\pa}{\pa t_i}-\xi_{ik})}
{j-(\xi_{i1}-\xi_{ik})} \cdot 
\dfrac{j+(t_i\frac{\pa}{\pa t_i}-\xi_{ik})}
{j+(\xi_{i1}-\xi_{ik})}\right)^m \right\}. $$
Take an aligned closed subinterval $[d,e]$ in $I$ with 
$d<b, c<e$ if $b>0$ and $d=0,c<e$ if $b=0$. We prove the following claim: \\
\quad \\
{\bf claim 1.} \,\, 
For any $\vv \in \Gamma(X \times A^n_K[d,e],E)$, the sequence 
$\{D_l(\vv)\}_l$ converges to an element in 
$H^0_{X,\xi}(X \times A^n_K[b,c],E)$, where $\xi:=(\xi_{11},...,\xi_{n1})$. \\
\quad \\
To prove the convergence, it suffice to check in 
$\Gamma(A^n_L[b,c],F)$. So we can write $\vv$ as 
$\vv:=\sum_{J}t_1^{j_1}\cdots t_n^{j_n}\vv_J$ with $\vv_J \in W$. 
Let us put $\vv_{J,\k} := \varphi_{\k}(\vv_J).$ Then we have 
\begin{align*}
D_l(\vv) & = 
\sum_{\k} \left\{ 
\sum_J t_1^{j_1} \cdots t_n^{j_n} \cdot \right. \\
& \left. \prod_{i=1}^n \left\{ 
Q_i(N_i+j_i) 
\prod_{k=1}^{n_i}\prod_{j=1}^l 
\left( 
\dfrac{j-(N_i+j_i-\xi_{ik})}{j-(\xi_{i1}-\xi_{ik})} \cdot 
\dfrac{j+(N_i+j_i-\xi_{ik})}{j+(\xi_{i1}-\xi_{ik})}\right)^m \right\}
(\vv_{J,\k}) \right\}. 
\end{align*}

We compute the term 
\begin{equation}\label{3.4.1eq1}
\prod_{i=1}^n \left\{ 
Q_i(N_i+j_i) 
\prod_{k=1}^{n_i}\prod_{j=1}^l 
\left( 
\dfrac{j-(N_i+j_i-\xi_{ik})}
{j-(\xi_{i1}-\xi_{ik})} \cdot 
\dfrac{j+(N_i+j_i-\xi_{ik})}
{j+(\xi_{i1}-\xi_{ik})}\right)^m \right\}(\vv_{J,\k})
\end{equation} 
for each $J,\k$. First, in the case $J=\0 := (0,...,0)$, 
we have $\prod_{i=1}^nQ(N_i)(\vv_{\0,\k}) \in W'_{\1}$. Hence 
the term \eqref{3.4.1eq1} is equal to 
$$ 
\prod_{i=1}^n 
\prod_{k=1}^{n_i}\prod_{j=1}^l 
\left( 
\dfrac{j-(\xi_{i1}-\xi_{ik})}
{j-(\xi_{i1}-\xi_{ik})} \cdot 
\dfrac{j+(\xi_{i1}-\xi_{ik})}
{j+(\xi_{i1}-\xi_{ik})}\right)^m 
\left(\prod_{i=1}^nQ(N_i)(\vv_{\0,\k})\right) = 
\prod_{i=1}^nQ(N_i)(\vv_{\0,\k}). $$
Second, let us consider the case where $J := (j_1,...,j_n)$ contains some 
factor $j_i$ with $|j_i| \leq l$. If we put $\k:=(k_1,...,k_n)$, 
then the term 
\eqref{3.4.1eq1} contains a factor 
$$ 
\left(\dfrac{j_i-(N_i+j_i-\xi_{ik_i})}{j_i-(\xi_{i1}-\xi_{jk_i})}\right)^m 
= (\const) \cdot (N_i-\xi_{ik_i})^m $$ 
or a factor 
$$ 
\left(\dfrac{-j_i+(N_i+j_i-\xi_{ik_i})}{-j_i+(\xi_{i1}-\xi_{jk_i})}\right)^m 
= (\const) \cdot (N_i-\xi_{ik_i})^m. $$ 
Hence we can conclude that the term \eqref{3.4.1eq1} is equal to zero in 
this case. Finally, let us consider the case where 
any factor $j_i$ of $J := (j_1,...,j_n)$ satisfies $|j_i|>l$. 
In this case, let us rewrite the term \eqref{3.4.1eq1} as 

\begin{align}
\label{3.4.1eq2}
& \prod_{i=1}^n \left\{ 
Q_i((N_i-\xi_{ik_i})+(\xi_{ik_i}+j_i)) 
\prod_{k=1}^{n_i}\prod_{j=1}^l 
\left( 
\dfrac{(j-j_i+\xi_{ik}+\xi_{ik_i})-(N_i-\xi_{ik_i})}
{j-(\xi_{i1}-\xi_{ik})} \right. \right. \\ & 
\left. \left. \hspace{6cm} \cdot 
\dfrac{(j+j_i-\xi_{ik}+\xi_{ik_i})+(N_i-\xi_{ik_i})}
{j+(\xi_{i1}-\xi_{ik})}\right)^m \right\}(\vv_{J,\k}). \nonumber
\end{align} 
Then we see that the operator in 
\eqref{3.4.1eq2} between parentheses $\{ \phantom{aaa}\}$ is a polynomial on 
$N_i-\xi_{ik_i}$ and hence it can be written as the form 
$\sum_{\alpha=0}^{m-1} c_{i,\alpha}^{j_i,l}(N_i-\xi_{ik_i})^{\alpha}$ for 
some $c_{i,\alpha}^{j_i,l} \in K$ modulo $(N_i-\xi_{ik_i})^m$. 
Then the term \eqref{3.4.1eq1} can be expressed as 
$$ 
\sum_{\alpha_1,...,\alpha_n=0}^{m-1} 
\prod_{i=1}^n\{ c_{i,\alpha_i}^{j_i,l}(N_i-\xi_{ik_i})^{\alpha_i}\}
(\vv_{J,\k}). $$ 

Therefore, we can calculate $D_l(\vv)- \prod_{i=1}^nQ_i(N_i)(\vv_{\0})$ 
as follows: 
\begin{align}
& D_l(\vv)- \prod_{i=1}^nQ_i(N_i)(\vv_{\0}) \label{3.4.1eq3} \\ = & 
\sum_{\k}\sum_{\alpha_1,...,\alpha_n=0}^{m-1} 
\sum_{\scriptstyle J \atop \scriptstyle 
\forall i, \, |j_i|>l} \left(\prod_{i=1}^n 
c_{i,\alpha_i}^{j_i,l}\right) 
t_1^{j_1}\cdots t_n^{j_n} \cdot \prod_{i=1}^n(N_i - \xi_{ik_i})^{\alpha_i} 
(\vv_{J,\k}). \nonumber 
\end{align}
Since $\prod_{i=1}^nQ_i(N_i)(\vv_{\0})$ is killed by $N_i - \xi_{i1}$ for any 
$i$, the proof of claim 1 is reduced to the following claim: \\
\quad \\
{\bf claim 2.} \,\, The sequence 
$\{D_l(\vv)- \prod_{i=0}^nQ_i(N_i)(\vv_{\0})\}$ is 
$\eta$-null for some $\eta>1$. \\
\quad \\
By \eqref{3.4.1eq3} and \cite[3.1.11]{kedlayaI} (see also the proof of 
\cite[3.4.1]{kedlayaI}), the claim 2 is reduced to the following: \\
\quad \\
{\bf claim 3.} \,\, 
There exists a constant $A$ such that, for any $\delta>1$, we have the 
inequality $|c_{i,\alpha}^{j',l}| \leq (\const) \, \delta^{A|j'|}$ for 
all $i,\alpha$ and $j'$ with $|j'|>l$. \\
\quad \\
We prove claim 3. Let us define $p^{j'}_{i,\alpha}, q^{j',l}_{i,k,\alpha}, 
r^{j',l}_{i,k,\alpha} \in K$ by 
\begin{align*}
Q_i(x-(\xi_{ik_i}+j')) & = 
\sum_{\alpha=0}^{m-1} p^{j'}_{i,\alpha}x^{\alpha} \,\,\,\, 
\,\,{\rm mod}\,x^m, \\
\prod_{j=1}^l \dfrac{(j-j'+\xi_{ik}-\xi_{ik_i})-x}{j-(\xi_{i1}-\xi_{ik})} 
& = 
\sum_{\alpha=0}^{m-1} q^{j',l}_{i,k,\alpha}x^{\alpha} \,\,\,\, 
\,\,{\rm mod}\,x^m, \\
\prod_{j=1}^l \dfrac{(j+j'-\xi_{ik}+\xi_{ik_i})-x}{j+(\xi_{i1}-\xi_{ik})} 
& = 
\sum_{\alpha=0}^{m-1} r^{j',l}_{i,k,\alpha}x^{\alpha} \,\,\,\, 
\,\,{\rm mod}\,x^m. 
\end{align*}
Then, to prove the claim 3, it suffices to prove the same assertion for 
$p_{i,\alpha}^{j'}$, $q_{i,k,\alpha}^{j',l}$ and $r_{i,k,\alpha}^{j',l}$ 
$(1 \leq k \leq n_i)$. 
Since $\xi_{ik_i}+j'$ is in $\Z_p$, $|p_{i,\alpha}^{j'}|$ is equal or less 
than the maximum of the absolute values of the coefficients of $Q_i$'s, 
which is independent of $i,\alpha$ and $j'$. 
Hence the claim 3 is true for $|p_{i,\alpha}^{j'}|$. Next, let us note 
that $\pm(\xi_{i1}-\xi_{ik}), \pm(\xi_{ik}-\xi_{ik_i})$ are in 
$\Z_p-(\Z-\{0\})$ and $p$-adically non-Liouville. Hence the claim 3 for 
$q_{i,k,\alpha}^{j',l}$ and $r_{i,k,\alpha}^{j',l}$ are reduced to the 
following claim: \\
\quad \\
{\bf claim 4.} \,\, 
Let $\sigma, \tau \in \Z_p-(\Z-\{0\})$ be $p$-adically non-Liouville numbers 
and define $s_{\alpha}^{j',l}$ by 
$$ \prod_{j=1}^l\dfrac{(j+j'+\tau)+x}{j+\sigma} = 
\sum_{\alpha=0}^{m-1} s_{\alpha}^{j',l} x^{\alpha} \,\,\,\, 
\,\,{\rm mod}\,x^m. $$ 
Then, for any $\delta>1$, we have 
$|s_{\alpha}^{j',l}| \leq (\const)\, \delta^{(2m-1)|j'|}$ 
for any $\alpha$ and $j'$ with $|j'|>l$. \\
\quad \\
We prove this claim. Let us fix $\delta>1$. Then by the 
$p$-adic non-Liouvilleness of $\tau$, we have, for some constant $C>1$, 
the inequalities  
$$ |j+j'+\tau|^{-1} < C\, \delta^{|j+j'|} < C\, 
\delta^{2|j'|} $$ 
for any $j \leq l < |j'|$. Hence we have
$$ 
|s_{\alpha}^{j',l}| \leq C^{\alpha} \delta^{2|j'|\alpha} |s_0^{j',l}|
\leq C^{m-1}\, \delta^{2(m-1)|j'|} |s_0^{j',l}| $$
for any $\alpha$ and $j'$ with $|j'|>l$. So it suffices to prove the 
inequality $|s_0^{j',l}| \leq (\const)\, \delta^l$ for any $j'$ with 
$|j'|>l$. \par 
Note that we have $s_0^{j',l} = \prod_{j=1}^l\dfrac{j+j'+\tau}{j+\sigma} 
= \prod_{j=1}^l\dfrac{j+j'+\tau}{j} \cdot \prod_{j=1}^l\dfrac{j}{j+\sigma}$, 
and it is easy to see the inequality 
$\left| \prod_{j=1}^l\dfrac{j+j'+\tau}{j} \right| \leq 1$. 
Moreover, the radius of convergence of the series 
$\sum_{l=0}^{\infty}\dfrac{X^l}{\prod_{j=1}^l(j+\sigma)}$ is 
equal to $p^{-1/(p-1)}$ by \cite[VI Lemma 1.2]{dgs} and 
for any $\delta >1$, we have 
$|\prod_{j=1}^l j| \leq (\const) \, (\delta p^{-1/(p-1)})^l$. 
Therefore, for any $\delta>1$, we have
$|\prod_{j=1}^l\dfrac{j}{j+\sigma}| < (\const)\,\delta^l$. So we obtain 
the inequality $|s_0^{j',l}| \leq (\const) \, \delta^l$ 
and hence we have proved the
 claim 4. So the proof of the claims 1,2 and 3 are also finished. \par 
By the claim, we can define, for any $\vv \in \Gamma(X \times A^n_K[d,e],E)$ 
the element $f(\vv) := \lim_l D_i(\vv) \in H^0_{X,\xi}(X \times 
A^n_K[b,c],E)$ and it is equal to $\prod_{i=1}^nQ_i(N_i)(\vv_{\0})$. 
If we take a generator $\{\w_1,...,\w_s\}$ of 
$\Gamma(X \times A^n_K[d,e],E)$, we see that the $L$-span of the set 
$\{t^J\w_i\,\vert\,J \in \Z^n, 1 \leq i \leq s\}$ 
(resp. $\{t^J\w_i\,\vert\,J \in \N^n, 1 \leq i \leq s\}$) is dense in 
$\Gamma(A^n_L[b,c],F)$ when $b>0$ (resp. $b=0$). 
Hence the $L$-span on the 
image of this set by $f$ is dense 
in $\prod_{i=1}^nQ_i(N_i)(W)$, which is non-zero. Hence there exists an 
element $\vv \in \Gamma(X \times A^n_K[d,e],E)$ with $f(\vv) \not= 0$. 
Hence we have $H^0_{X,\xi}(X \times A^n_K[b,c],E) \not= 0$, as desired. 
\end{proof}

By using Lemma \ref{3.4.1}, we can prove the following proposition, 
which is the analogue of \cite[3.4.3]{kedlayaI}: 

\begin{prop}[generization]\label{3.4.3}
Let $A$ be an affinoid algebra such that $X :=\Spm A$ 
is smooth and endowed with $x_1,...,x_r \in \Gamma(X,\cO_X)$ whose zero loci 
are smooth and meet transversally and 
let $A \subseteq L$ be one of the following$:$ \\
$(1)$ \,\, $L$ is an affinoid algebra over $K$ such that 
$\Spm L$ is smooth, $x_i$'s are invertible in $L$ 
and that the supremum norm on $L$ restricts to 
the supremum norm on $A$. \\
$(2)$ \,\, $L$ is a field containing $A$ which is complete for a 
norm restricting to the supremum norm on $A$. \\
Let $I$ be a quasi-open subinterval of positive length in 
$[0,1)$ and let us define $A^n_L(I)$ 
as in Lemma \ref{3.4.1}. 
Let $\Sigma := \Sigma' \times \prod_{i=1}^{n} \Sigma_i$ be a subset of 
$\ol{K}^r \times \Z_p^n$ which is $\NID$ and $\NLD$, and 
let $E$ be an object in $\LNM_{X \times A^n_K(I),\Sigma}$ such that 
the induced object $F \in \LNM_{A^n_L(I),\Sigma}$ is 
$\Sigma$-unipotent. Then $E$ is also $\Sigma$-unipotent. 
\end{prop} 

\begin{proof} 
The proof is similar to that in \cite[3.4.3]{kedlayaI}, 
although we provide some supplementary argument which was missing there. \par 
Fix any closed aligned subinterval $[b,c] \subseteq I$ of positive length. 
Then, by Lemma \ref{3.4.1}, there exists some 
$\xi:=(\xi_1,...,\xi_n) \in \prod_{i=1}^{n} \Sigma_i$ such that 
$H_E := H^0_{X,\xi}(X \times A^n_K[b,c],E) \allowbreak 
\not=0$. First we prove the 
following claim: \\
\quad \\
{\bf claim 1.} \,\, $H_E$ is a finitely generated $A$-module and the 
canonical map $\pi_X^*H_E \lra E$ (where $\pi_X: X \times A^n_K[b,c] \lra X$ 
is the projection) is injective. \\
\quad \\
Let us take any finitely 
generated sub $A$-module $M \subseteq H_E$ and consider the following 
diagram: 
\begin{equation}\label{3.4.3eq1}
\begin{CD}
\pi^*_{X}M @>>> E \\ 
@VVV @VVV \\ 
\pi_L^*H^0_{L,\xi}(A^n_L[b,c],F) @>>> 
F|_{A^n_L[b,c]}. 
\end{CD}
\end{equation}
Then the vertical arrows are injective. 
We prove that the lower horizontal arrow is also injective: 
Indeed, if we write $F|_{A^n_L[b,c]}$ by 
$F|_{A^n_L[b,c]} = \cU_{[b,c]}(W)$ as in the proof of Lemma \ref{3.4.1}, 
a formal power series computation shows the equality 
$H^0_{L,\xi}(A^n_L[b,c],F) = \{x \in W \,\vert\, \forall i, N_i(x)=\xi_ix\}$. 
(Here $N_i$ is as in the proof of Lemma \ref{3.4.1}.) Hence we have the 
injection $H^0_{L,\xi}(A^n_L[b,c],F) \hra W$ which induces the desired 
injection $\pi_L^*H^0_{L,\xi}(A^n_L[b,c],F) \hra \pi_L^*W = F|_{A^n_L[b,c]}$. 
So, by the diagram \eqref{3.4.3eq1}, we see that the map 
$\pi_{X}^*M \lra E$ is injective for all finitely generated sub $A$-module 
$M$ of $H_E$. By the Noetherianness of $E$, we can choose $M$ so that the 
image of $\pi_{X}^*M$ in $E$ is maximal among the modules of this form. 
Then, if $M$ is strictly smaller than $H_E$, 
we can choose a finitely generated 
sub $A$-module $M'$ of $H_E$ which strictly contains $M$ and then we have a 
strict injection $\pi_{X}^*M \subsetneq \pi_{X}^*M'$, which contradicts the 
definition of $M$. Hence we have $M=H_E$. So $H_E$ is a finitely generated 
$A$-module and the map $\pi_{X}^*H_E \lra E$ is injective, as desired. \par 
By claim 1, we can naturally regard $H_E$ as a coherent $\cO_X$-module. 
Let $X'$ be the complement of the zero loci of $x_1,...,x_r$ in $X$ 
and let us put $G':=\pi_{X'}^*(H_E|_{X'})$ (where $\pi_{X'}$ denotes 
the restriction of $\pi_X$ to $X' \times A^n_K[b,c]$), which is 
a coherent submodule of $E' := E|_{X' \times A^n_K[b,c]}$. Then we prove the 
following claim: \\ 
\quad \\
{\bf claim 2.} \,\,\, 
$G'$ has a natural structure of a $\Sigma$-constant log-$\nabla$-module 
on $X' \times A^n_K[b,c]$ which is a subobject of $E'$ 
in $\LNM_{X' \times A^n_K[b,c],\Sigma}$. \\
\quad \\
Indeed, 
we see easily that $H_E|_{X'} 
\subseteq E'$ is stable by continuous derivations 
of $X'$ over $K$ and it admits commuting endomorphisms $t_i(\pa/\pa t_i)$ 
($1 \leq i \leq n$), which is nothing but the multiplication by $\xi_i$. 
If we forget the endomorphisms $t_i(\pa/\pa t_i)$, $H_E|_{X'}$ is regarded as 
a $\nabla$-module on $X'$. Hence it is a locally free $\cO_{X'}$-module. 
Then, adding the endomorphisms $t_i(\pa/\pa t_i)$ 
($1 \leq i \leq n$), $H_E|_{X'}$ is regarded as an object in 
$\LNM_{X' \times A^n_K[0,0],\Sigma}$. Then $G' = \pi_{X'}^*(H_E|_{X'}) 
\hra E'$ 
is naturally regarded as $\cU_{[b,c]}(H_E|_{X'})$. Hence $G'$ is 
$\Sigma$-constant (in fact $\Sigma' \times 
\prod_{i=1}^n\{\xi_i\}$-constant), 
as desired. \par 
Now let $G \subseteq E$ be the subobject in the category 
$\LNM_{X \times A^n_K[b,c],\Sigma}$ extending $G' \subseteq E'$, 
which uniquely exists by Proposition \ref{3.3.8}. 
Also, we put $H_G := H^0_{X,\xi}(X \times A^n_K[b,c], G) \subseteq H_E$. 
(Then, by the same argument as the proof of claim 1, we see that 
$H_G$ is also coherent.) Next we prove the following claim: \\
\quad \\
{\bf claim 3.} \,\, 
The canonical map $\pi_X^*H_G \lra G$ is surjective. \\
\quad \\
Let us note that the restriction $G_L$ of $G$ to $A^n_L[b,c]$, which is equal 
to the restriction of $G'=\cU_{[b,c]}(H_E|_{X'})$ to $A^n_L[b,c]$, is 
$\Sigma' \times \prod_{i=1}^n\{\xi_i\}$-constant. 
Hence, if we take $[d,e] \supseteq [b,c]$ as in 
Lemma \ref{3.4.1}, we can define the map 
$$ f = \lim_lD_l: 
\Gamma(X \times A^n_K[d,e],G) \lra 
H^0_{X,\xi}(X \times A^n_K[b,c],G) =H_G $$
as in Lemma \ref{3.4.1} with $Q_i(x)=1$ ($1 \leq i \leq n$) in the definition 
of $D_l$. Then the calculation in Lemma \ref{3.4.1} says that, if we 
write $G_L = \cU_{[b,c]}(V)$ $(V \in \LNM_{A^n_L[0,0],\Sigma})$ and if 
we express 
an element $\vv \in \Gamma(X \times A^n_K[d,e],G)$ by 
$\vv = \sum_{J}\vv_Jt^J$ $(\vv_J \in V)$, we have $f(\vv)=\vv_{\0}$, 
where $\0=(0,...,0)$. \par 
Now we prove the claim 3 in the case $b>0$. Note that, for 
$\vv \in \Gamma(X \times A^n_K[d,e],G)$, the series 
$\vv = \sum_{J\in \Z^n}\vv_Jt^J$ converges in $G_L$, hence in $G$. 
Hence we obtain the equality $\vv = \sum_{J \in \Z^n} f(t^{-J}\vv)t^J$ 
in $G$. From this equality, we see that $\vv$ is contained in 
$\pi_X^*H_G$ for any $\vv \in \Gamma(X \times A^n_K[d,e],G)$. Since 
$\Gamma(X \times A^n_K[b,c],G)$ is dense in 
$\Gamma(X \times A^n_K[d,e],G)$, 
the canonical map $\pi_X^*H_G \lra G$ is surjective in the case $b>0$. \par 
The claim 3 is true also in the case $b=0$ which we can prove in 
exactly the same way as 
\cite[3.4.3]{kedlayaI}. (The proof is more complicated, but 
the idea is similar. See \cite[3.4.3]{kedlayaI} for detail.) \par 
By the claims 1 and 3, we see that the canonical map 
$\pi_X^*H_G \lra G$ is in fact an isomorphism. Now we prove that $G$ is 
$\Sigma$-constant. If we enlarge $K$ in order that 
$A^n_K[b,c]$ contains a $K$-rational point $x$ and if 
$i_x:X \lra X \times A^n_K[b,c]$ denotes 
the section of $\pi_X$ induced by $x$, 
then we have $H_G = i_x^*\pi_X^*H_G \os{\cong}{\lra} i_x^*G$. Hence $H_G$ 
is a locally free $\cO_X$-module. (Note that this conclusion is true 
without enlarging $K$.) So $H_G=H^0_{X,\xi}(X \times A^n_K[b,c],G)$ 
is naturally regarded as an object in $\LNM_{X \times A^n_K[0,0],\Sigma}$ 
and we have $G = \cU_{[b,c]}(H_G)$. Hence $G$ is $\Sigma$-constant. \par 
Then, by induction on the rank of $E$, we can prove that $E/G$ is 
$\Sigma$-unipotent on $X \times A^n_K(b,c)$ when $b\not=0$ and 
on $X \times A^n_K[0,c)$ when $b=0$. Hence so is $E$. 
Since $[b,c] \subseteq I$ is 
arbitrary, we see that $E$ is $\Sigma$-unipotent 
on $X \times A^n_K(I)$. So the proof is finished. 
\end{proof}

We have the following variant of Proposition \ref{3.4.3}, 
which is also important. 

\begin{cor}\label{3.4.3cor}
Let $A$ be an affinoid algebra such that $X :=\Spm A$ 
is smooth and 
let $L$ be an affinoid algebra containing $A$ such that 
$\Spm L$ is smooth. 
Let $I$ be a quasi-open interval of positive length in $[0,1)$, 
let $[b,c] \subseteq I$ be a closed aligned subinterval of positive 
length, let 
$\Sigma := \prod_{i=1}^{n} \Sigma_i$ be a subset of 
$\Z_p^n$ which is $\NID$ and $\NLD$, and 
let $E$ be an object in $\LNM_{X \times A^n_K(I),\Sigma}$ such that 
the induced object $F \in \LNM_{\Spm L \times A^n_K(I),\Sigma}$ is 
$\Sigma$-unipotent and that 
$H^0_{X,\xi}(X \times A^n_K[b,c],E)$ is non-zero for some $\xi \in \Sigma$. 
Then $E |_{X \times A^n_K[b,c]}$ 
contains a non-zero $\Sigma$-constant subobject $G \in 
\LNM_{X \times A^n_K[b,c],\Sigma}$. 
\end{cor} 

\begin{proof} 
By the proof of claims 1 and 2 in Proposition \ref{3.4.3}, 
we see that the $\cO_{X \times A^n_K[b,c]}$-span of 
$H^0_{X,\xi}(X \times A^n_K[b,c],E)$ naturally defines 
a non-zero $\Sigma$-constant subobject $G \in  
\LNM_{X \times A^n_K[b,c],\Sigma}$. 
(Note that, since we do not have `the log structure $x_1,...,x_r$' on $X$ 
in the situation here, we have $X=X'$ in the notation of the proof of 
Proposition \ref{3.4.3}. Note also that, although we do not assume that 
the spectral norm of $L$ is compatible with that of $A$, this does not 
cause any problem because we did not use this assumption in the proof of 
the claims 1 and 2 in Proposition \ref{3.4.3}.) So we are done. 
\end{proof}

\begin{rem}\label{e0}
Lemma \ref{3.4.1} and Proposition \ref{3.4.3} are slightly 
more generalized than the corresponding results in 
\cite[3.4.1, 3.4.3]{kedlayaI} in the sense that we treat the case (1). 
(In \cite{kedlayaI}, only the case (2) is treated.) 
Also, Corollary \ref{3.4.3cor} (which allows us to prove a 
kind of generization in the case where the spectrum norm on $L$ is not 
compatible with that on $A$) seems to be important. 
(See Remark \ref{e} below.) 
\end{rem}

Next we prove the proposition called overconvergent generization, which 
is an analogue of (a slightly weaker form of) \cite[3.5.3]{kedlayaI}: 

\begin{prop}[overconvergent generization]\label{3.5.3}
Let $P$ be a $p$-adic affine formal scheme topologically of finite type 
over $O_K$ and let $X \subseteq P_k$ be an open dense subscheme 
such that $P$ is formally smooth over $O_K$ on a neighborhood of $X$. 
Let $I \subseteq [0,1)$ be a quasi-open subinterval of positive length, 
let $V$ be a strict neighborhood of $]X[_P$ in $P_K$ and let 
$\Sigma = \prod_{i=1}^n\Sigma_i$ be a subset of $\Z_p^n$ which is 
$\NID$ and $\NLD$. Let $E$ be an object in 
$\LNM_{V \times A_K^n(I),\Sigma}$ whose restriction to 
$\LNM_{]X[_P \times A_K^n(I),\Sigma}$ is $\Sigma$-unipotent. 
Then, for any closed aligned subinterval $[b,c] \subseteq I$ of 
positive length, there exists a strict neighborhood $V'$ of $]X[_P$ 
in $P_K$ contained in $V$ such that the restriction of $E$ to 
$V' \times A^n_K[b,c]$ is also $\Sigma$-unipotent. 
\end{prop}

\begin{proof} 
We may assume that there exists $g \in \Gamma(P,\cO_P)$ such that 
$X$ is defined as the locus $\{g\not=0\}$ in $P_k$. For 
$\lam \in [0,1) \cap \Gamma^*$, let us put 
$V_{\lam} := \{x \in P_K \,\vert\, |g(x)| \geq \lam\}.$ 
Then we may assume that $V=V_{\lam_0}$ for some 
$\lam_0 \in [0,1) \cap \Gamma^*$ and that $V$ is smooth. \par 
Take a sequence of aligned closed subintervals 
$[b,c] \subseteq [b',c'] \subseteq [d,e]$ of $I$ with 
$d<b'<b, c<c'<e$ if $b>0$ and $d=b'=0, c<c'<e$ if $b=0$. 
Then, by the method of Lemma \ref{3.4.1}, we see that there exists 
an element $\vv \in \Gamma(V \times A^n_K[d,e],E)$ such that 
$\{D_l(\vv)\}$ converges to a non-zero element in 
$H^0_{]X[_P,\xi}(]X[ \times A^n_K[b',c'],E)$ for some 
$\xi \in \Sigma$. \par 
Let us take an affinoid open $W \subseteq V \times A^n_K[d,e]$ 
such that $E$ is free on $W$. We prove the following claim: \\
\quad \\
{\bf claim 1.} \,\, 
On $W$, the sequence $\{D_l(\vv)-D_{l-1}(\vv)\}_l$ is 
$\rho$-null for some $\rho>0$. \\
\quad \\
First, we have
{\allowdisplaybreaks{
\begin{align}
& |D_l(\vv)-D_{l-1}(\vv)| \label{3.5.3eq1} \\
= &  
\left| \left( \prod_{i=1}^n \prod_{k=1}^{n_i} 
\left( 
\dfrac{l-(t_i\frac{\pa}{\pa t_i}-\xi_{ik})}{l-(\xi_{i1}-\xi_{ik})} \cdot 
\dfrac{l+(t_i\frac{\pa}{\pa t_i}-\xi_{ik})}{l+(\xi_{i1}-\xi_{ik})}
\right)^m -1 \right) D_{l-1}(\vv) \right| \nonumber \\ 
= &  
\left| \left( \prod_{i=1}^n \prod_{k=1}^{n_i} 
\left( 
\dfrac{1}{l-(\xi_{i1}-\xi_{ik})} \cdot 
\dfrac{1}{l+(\xi_{i1}-\xi_{ik})} \cdot \right.\right.\right. 
\nonumber \\ 
& \phantom{===} \left.\left.\left.
\left(l-\left(t_i\dfrac{\pa}{\pa t_i}-\xi_{ik}\right)\right) \cdot 
\left(l+\left(t_i\dfrac{\pa}{\pa t_i}-\xi_{ik}\right)\right) 
\right)^m - 1 \right) 
D_{l-1}(\vv) \right| \nonumber \\ 
\leq & 
C \cdot \left| \prod_{i=1}^n \prod_{k=1}^{n_i} 
\left( \dfrac{1}{l-(\xi_{i1}-\xi_{ik})} \cdot 
\dfrac{1}{l+(\xi_{i1}-\xi_{ik})} \right) \right|^m \cdot 
|D_{l-1}(\vv)|. \nonumber 
\end{align}}}
(Here $C>1$ is a constant which is independent of $l$, which is written 
by using the action of the operators $t_i\dfrac{\pa}{\pa t_i}$ 
on the basis of $E|_W$.) 
If we put 
$A_l := C \cdot 
\left| \prod_{i=1}^n \prod_{k=1}^{n_i} 
\left( \dfrac{1}{l-(\xi_{i1}-\xi_{ik})} \cdot 
\dfrac{1}{l+(\xi_{i1}-\xi_{ik})} \right) \right|^m$, 
we see from \eqref{3.5.3eq1} the inequality 
$|D_l(\vv)-D_{l-1}(\vv)| \leq A_l|D_{l-1}(\vv)|$. 
Then we can prove the inequality 
\begin{equation}\label{3.5.3eq2}
|D_l(\vv)-D_{l-1}(\vv)| \leq (\prod_{j=i}^lA_j) |D_{i-1}(\vv)|
\,\,\,\, (1 \leq i \leq l)
\end{equation}
 by descending induction: Indeed, \eqref{3.5.3eq2} for 
 $i=l$ is already shown, and if \eqref{3.5.3eq2} is true for $i$, 
 we have 
\begin{align*}
|D_l(\vv)-D_{l-1}(\vv)| & \leq (\prod_{j=i}^lA_j) |D_{i-1}(\vv)| \\ 
& \leq (\prod_{j=i}^lA_j) \max (|D_{i-1}(\vv)-D_{i-2}(\vv)|,|D_{i-2}(\vv)|) \\ 
& \leq (\prod_{j=i}^lA_j) \max (A_{i-1}|D_{i-2}(\vv)|,|D_{i-2}(\vv)|) 
= (\prod_{j=i-1}^lA_j) |D_{i-2}(\vv)|. 
\end{align*}
In particular, we have the inequality 
$|D_l(\vv)-D_{l-1}(\vv)| \leq (\prod_{j=1}^lA_j) |\vv|$. 
Let us calculate $\prod_{j=1}^lA_j$. By definition, we have
$$ 
\prod_{j=1}^lA_j = 
C^l\left| \prod_{i=1}^n \prod_{k=1}^{n_i} \left\{
\left(\prod_{j=1}^l \dfrac{1}{j-(\xi_{i1}-\xi_{ik})}\right)^m 
\left(\prod_{j=1}^l \dfrac{1}{j+(\xi_{i1}-\xi_{ik})}\right)^m \right\} 
\right|. $$
Since $\pm(\xi_{i1}-\xi_{ik})$ are $p$-adically non-Liouville, 
we see from \cite[VI Lemma 1.2]{dgs} 
that, for any $\zeta > p^{1/(p-1)}$, we have the inequalities 
$$ \left| \prod_{j=1}^l \dfrac{1}{j\pm (\xi_{i1}-\xi_{ik})} \right| \leq 
(\const)\,\cdot\,\zeta^l. $$
Hence we have 
$\prod_{j=1}^lA_j \leq (\const)\,\cdot(C \zeta^{2m})^l$. 
Therefore, if we put $\rho := (C \zeta^{2m})^{-2}$, 
we see that the sequence $\{D_l(\vv)-D_{l-1}(\vv)\}_l$ is 
$\rho$-null. So we have proved the claim 1. \par 
When $W \cap (]X[ \times A^n_K[d,e])$ is not empty, we see from 
claim 2 in Lemma \ref{3.4.1} that the sequence 
$\{D_l(\vv)-D_{l-1}(\vv)\}_l$ is $\eta$-null for some $\eta>1$ on 
$]X[ \times A^n_K[b',c']$. On the other hand, by claim 1, 
it is $\rho$-null for some $\rho>0$ on $W \cap (V \times A^n_K[b',c'])$. 
Hence, by \cite[3.5.2]{kedlayaI}, we see that it is 
$1$-null on $W \cap (V_{\lam} \times A^n_K[b',c'])$ for some 
$\lam \in (\lam_0,1) \cap \Gamma^*$. On the other hand, 
when $W \cap (]X[ \times A^n_K[d,e])$ is empty, then we see that 
$W \cap (V_{\lam} \times A^n_K[b',c'])$ is also empty for some 
$\lam$. (See the proof of \cite[3.5.3]{kedlayaI}.) \par 
By taking a covering of $V \times A^n_K[d,e]$ by affinoids $W$ on which 
$E$ is free and argueing as above, we see that there exists 
some $\lam \in (\lam_0,1) \cap \Gamma^*$ such that 
the limit $\lim_lD_l(\vv)$ exists on $V_{\lam} \times A^n_K[b',c']$, 
which is non-zero for some $\vv$. Hence we have 
$H^0_{V_{\lam},\xi}(V_{\lam} \allowbreak 
\times A^n_K[b',c'],E) \not= 0$. 
Then, by Corollary \ref{3.4.3cor} (with $A=\Gamma(V_{\lam},\cO), 
L=\Gamma(]X[,\allowbreak \cO)$), we see that 
$E$ contains a $\Sigma$-constant subobject $G$ on 
$V_{\lam} \times A^n_K[b',c']$. Then, $E/G$ satisfies the assumption of
 the proposition required for $E$ if we replace $I$ by $(b',c')$. 
Hence, by induction on the rank of $E$, we can assume that 
$E/G$ is $\Sigma$-unipotent on $V_{\lam} \times A^n_K[b,c]$ for some 
$\lam$. Hence $E$ is $\Sigma$-unipotent on 
$V_{\lam} \times A^n_K[b,c]$ for some $\lam$ and so 
we are done. 
\end{proof}

\begin{rem}\label{e}
Proposition \ref{3.5.3} is slightly weaker than the corresponding 
result in \cite[3.5.3]{kedlayaI}, because we do not put `the log 
structure' on $P_K$ here. However, it seems to us that there is 
an error in the proof of \cite[3.5.3]{kedlayaI}: 
At the final step of the proof there, it is written that 
`the role of $L$ in the proof of Proposition 3.4.3 is played by 
a complete field containing $\cO(]X[)$ whose norm is compatible 
with the norm on $\cO(V_{\lam}$)'. But such a field $L$ does not 
exist in general because the reduction of the affinoid algebra 
$\cO(V_{\lam})$ is not an integral domain in general. 
(See also the statement of \cite[3.4.1]{kedlayaI}.) \par 
Note that, in the proof of \cite[3.4.3]{kedlayaI}, he uses the fact that 
the norm of $L$ is compatible with that of $\cO(V)$: It is used in 
the proof of the assertion corresponding to the claim 3 in 
Proposition \ref{3.4.3} here. On the other hand, we saw that, 
in the proof of Corollary \ref{3.4.3cor}, we do not need the 
compatibility of the norms of $L$ and that of $A$. Moreover, 
in the proof of Proposition \ref{3.5.3}, we used only 
Corollary \ref{3.4.3cor}, not Proposition \ref{3.4.3}. 
This is the reason that our Proposition \ref{3.5.3} is slightly 
weaker. However, as we will see in the next section, this does not 
cause any problem in proving our main result. \par 
We would like to note that there is another way to fix the above-mentioned 
error, which is explained in \cite[Appendix A]{kedlayaIV}. 
\end{rem} 

Now we proceed to prove the third key proposition. 
To do this, we introduce the notion of log-convergence 
of a log-$\nabla$-module. 

\begin{defn}\label{log-conv}
Let $X$ be a smooth affinoid rigid space endowed with 
$x_1,...,x_r \in \Gamma(X,\cO_X)$ whose zero loci are smooth and 
meet transversally, let $a \in (0,1]\cap\Gamma^*$ and 
let $E$ be a log-$\nabla$-module on 
$X \times A^n_K[0,a)$ with respect to $x_1,...,x_r,t_1,...,t_n$. 
Then $E$ is called log-convergent if, for any $a' \in (0,a)\cap\Gamma^*, 
\eta \in (0,1)$ and $\vv \in \Gamma(X \times A^n_K[0,a'],E)$, the 
multisequence 
\begin{equation}\label{log-conveq1}
\left\{ \dfrac{1}{i_1!\cdots i_n!}\left( 
\prod_{j=1}^n\prod_{l=0}^{i_j-1}\left( 
t_j\dfrac{\pa}{\pa t_j} - l \right) \right) (\vv) \right\}_{i_1,...,i_n} 
\end{equation}
is $\eta$-null on $X \times A^n_K[0,a']$. 
\end{defn} 

For a multi-index 
$I:=(i_1,...,i_n) \in \N^n$ and a multi-variable 
$x := (x_1,...,x_n)$, we put 
$P_I(x) := \dfrac{1}{i_1!\cdots i_n!}\left( 
\prod_{j=1}^n\prod_{l=0}^{i_j-1}\left( 
x_j - l \right) \right).$ Then the multi-sequence in Definition 
\ref{log-conv} can be written as $\{P_I(t\frac{\pa}{\pa t})(\vv)\}_I$ 
with this notation. The operator $P_I(t\frac{\pa}{\pa t})$ is the same 
as the operator $t^I (\frac{\pa}{\pa t})^I$ in the notation of 
\cite[3.7.2]{bc}. 

\begin{rem} 
The definition of log-convergence given above is closely related 
to the notion of `overconvergence' defined in \cite[5.1]{bc} 
in the case $a=1$: Indeed, 
they are equivalent if $r=0$ and if 
$E$ has the form $\pi^*F$ for some 
coherent sheaf $F$ on $X$, where $\pi$ denotes the projection 
$X \times A^n_K[0,1) \lra X$. 
\end{rem} 

It is known (\cite[3.7.3]{bc}) that, for 
$f \in \Gamma(X \times A^n_K[0,a'],\cO)$ and $\vv \in 
\Gamma(X \times A^n_K[0,a'],E)$, we have 
$$ P_I\left(t\frac{\pa}{\pa t}\right)(f\vv) = 
\sum_{0 \leq I' \leq I} P_{I'}\left(t\frac{\pa}{\pa t}\right)(f) \cdot 
P_{I-I'}\left(t\frac{\pa}{\pa t}\right)(\vv). $$
Hence we have 
\begin{align*}
\left| P_I\left(t\frac{\pa}{\pa t}\right)(f\vv) \right| & \leq 
\sup_{0\leq I'\leq I} 
\left| P_{I'}\left(t\frac{\pa}{\pa t}\right)(f) \right| 
\left| P_{I-I'}\left(t\frac{\pa}{\pa t}\right)(\vv) \right| \\ 
& \leq 
|f| \sup_{0\leq I'\leq I} 
\left| P_{I-I'}\left(t\frac{\pa}{\pa t}\right)(\vv) \right|, 
\end{align*}
where the second inequality follows from \cite[5.2]{bc}. 
Hence, for $\vv = \sum_j f_j\vv_j$, we have 
$$ 
\left| P_I\left(t\frac{\pa}{\pa t}\right)(\vv) \right| \leq 
\sup_j \left( |f_j| \sup_{0\leq I'\leq I} 
\left| P_{I-I'}\left(t\frac{\pa}{\pa t}\right)(\vv_j) \right| \right). 
$$ 
So we have 
\begin{equation}\label{log-conveq2}
\left| P_I\left(t\frac{\pa}{\pa t}\right)(\vv) \right| \eta^{|I|} 
\leq 
\sup_j \left( |f_j| \sup_{0\leq I'\leq I} 
\left| P_{I-I'}\left(t\frac{\pa}{\pa t}\right)(\vv_j) \right| 
\eta^{|I-I'|/2} \right) \cdot \eta^{|I|/2} 
\end{equation} 
for $\eta<1$. 
From \eqref{log-conveq2}, we see that, to check the $\eta$-nullity 
of the multisequence \eqref{log-conveq1} for all 
$\vv \in \Gamma(X \times A^n_K[0,a'],E)$, it suffices to check it 
only for a set of generators of $\Gamma(X \times A^n_K[0,a'],E)$. 
Also, \eqref{log-conveq2} shows that $E$ is log-convergent if and 
only if  
the multi-sequence 
$\left| P_I(t\frac{\pa}{\pa t}) \right| \eta^{|I|}$ 
converges to zero for any $a' \in (0,a) \cap \Gamma^*$, 
where $\left| P_I(t\frac{\pa}{\pa t}) \right|$
denotes the operator norm of $P_I(t\frac{\pa}{\pa t})$ on 
$\Gamma(X \times A^n_K[0,a'],E)$. \par 
Note that the trivial log-$\nabla$-module $(\cO,d)$ 
on $X \times A^n_K[0,1)$ is 
log-convergent, since we have $P_I(t\frac{\pa}{\pa t})(1)=0$ 
for $I$ with $|I|>0$. Also, we remark here that, if there exists 
a exact sequence 
$$ 0 \lra E' \lra E \lra E'' \lra 0 $$
of log-$\nabla$-modules on $X \times A^n_K[0,a)$ with respect to 
$x_1,...,x_r,t_1,...,t_n$ with $E',E''$ log-convergent, then 
$E$ is also log-convergent: Indeed, if we fix an isomorphism 
$E\cong E' \oplus E''$ as $\cO_{X \times A^n_K[0,a']}$-modules and 
if we write $t_j\dfrac{\pa}{\pa t_j}$ on $E$ as 
$\begin{pmatrix} \left(t_j\dfrac{\pa}{\pa t_j}\right)' & \Gamma_j \\ 0 & 
\left(t_j\dfrac{\pa}{\pa t_j}\right)'' \end{pmatrix}$ (where 
$\left(t_j\dfrac{\pa}{\pa t_j}\right)', 
\left(t_j\dfrac{\pa}{\pa t_j}\right)''$ are 
the operator $t_j\dfrac{\pa}{\pa t_j}$ for $E', E''$, respectively), 
we obtain, by a tedious calculation, an inequality 
$$ 
\left| P_I\left(t\dfrac{\pa}{\pa t}\right) \right| \leq 
\left( \sup_{0 \leq \alpha \leq |I|} \left| 1/\alpha \right| \right) 
\cdot 
\sup_{I'+I''<I} \left( \left| P_{I'}\left(\left(t\dfrac{\pa}{\pa t}
\right)'\right) \right| 
\left| P_{I''}\left(\left(t\dfrac{\pa}{\pa t}\right)''\right) 
\right| \right) \cdot 
\sup_j |\Gamma_j|. $$
Hence we have, for $\eta<1$, the inequality 
\begin{align*}
& \left| P_I\left(t\dfrac{\pa}{\pa t}\right) \right| \eta^{|I|} 
\\ & \leq 
(\const) \cdot \log |I| \cdot 
\sup_{I'+I''<I} \left( \left| P_{I'}\left(\left(t\dfrac{\pa}{\pa t}
\right)'\right) \right| 
\left| P_{I''}\left(\left(t\dfrac{\pa}{\pa t}\right)''\right) 
\right| \eta^{(|I'|+|I''|)/2}\right) 
\cdot \eta^{|I|/2}. 
\end{align*}
From this we see that the log convergence of $E'$ and $E''$ implies that of 
$E$. \par 
Now we prove some relations between log-convergence and 
$\Sigma$-unipotence. The following (certainly well-known) 
lemma, which can be regarded as a 
variant of \cite[3.6.1]{kedlayaI}, shows that 
$\Sigma$-unipotent log-$\nabla$-module is log-convergent under certain 
condition: 

\begin{lem} 
Let $X$ be a smooth affinoid rigid space endowed with 
$x_1,...,x_r \in \Gamma(X,\cO_X)$ whose zero loci are smooth and 
meet transversally. Let $\Sigma=\Sigma' \times \prod_{i=1}^{n}\Sigma_i$ 
be a subset of $\ol{K}^r \times \Z_p^n$ and let $E$ be a $\Sigma$-unipotent 
log-$\nabla$-module on $X \times A^n_K[0,1)$ 
with respect to $x_1,...,x_r,t_1,...,t_n$. 
Then $E$ is log-convergent. 
\end{lem} 

\begin{proof} 
Since the log-convergence is closed under extension, we may assume that $E$ is 
$\Sigma$-constant. Hence we may assume that $E$ has the form 
$\pi_1^*F \otimes \pi_2^*M_{\xi}$ for some $F \in \LNM_{X,\Sigma'}$ and 
$\xi = (\xi_1,...,\xi_n) 
\in \Z_p^n$, where $\pi_1: X \times A^n_K[0,1) \lra X$, 
$\pi_2: X \times A^n_K[0,1) \lra A^n_K[0,1)$ are projections. 
If we take a generator $\vv_1,...,\vv_d$ of $\Gamma(X,F)$, it forms a 
generator of $\Gamma(X \times A^n_K[0,a],E)$ for any $a \in (0,1)\cap\Gamma^*$ 
and we have the inequality 
\begin{align*}
\left| P_I\left(t\dfrac{\pa}{\pa t}\right)(\vv_s) \right| 
& = 
\left| \dfrac{1}{i_1!\cdots i_n!} \prod_{j=1}^n \prod_{l=0}^{i_j-1} 
(\xi_j-l)(\vv_s) \right| \\ 
& = 
\left| \prod_{j=1}^n \dfrac{\xi_j(\xi_j-1)\cdots (\xi_j-i_j+1)}{i_j!} (\vv_s)
\right| \leq |\vv_s|. 
\end{align*}
From this we see that $E$ is log-convergent. 
\end{proof}

The following proposition, which is the third key proposition in this section, 
shows that under some condition, log-convergence implies 
$\Sigma$-unipotence. This is a variant of the transfer theorems 
in \cite{chr}, \cite{bc0}, \cite[6.5.2]{bc} and \cite[3.6.2]{kedlayaI}. 

\begin{prop}\label{3.6.2}
Let $A$ be an integral affinoid algebra such that $X:=\Spm A$ is smooth and 
endowed with $x_1,...,x_r \in A$ whose zero loci are smooth and meet 
transversally. Assume moreover that there exists $A \subseteq L$ 
satisfying one of the conditions $(1)$, $(2)$ in Proposition \ref{3.4.3}. 
Let $\Sigma = \Sigma' \times \prod_{i=1}^n\Sigma_i$ be a subset 
of $\ol{K}^r \times \Z_p^n$ which is $\NID$ and $\NLD$. Then, if 
$E$ is an object in $\LNM_{X \times A^n_K[0,1),\Sigma}$ which is 
log-convergent, it is $\Sigma$-unipotent. 
\end{prop} 

\begin{proof} 
Let $F$ be the object in $\LNM_{A^n_L[0,1),\Sigma}$ induced by 
$E$. Then $F$ is also log-convergent. By Proposition \ref{3.4.3}, 
it suffices to prove that $F$ is $\Sigma$-unipotent, and it is 
enough to prove that $F$ is $\Sigma$-unipotent on $A^n_L[0,a)$ 
for any $a \in (0,1)\cap\Gamma^*$. By induction on the rank of $F$, 
we see that it suffices to prove the following claim: 
Let $a'<a$ be elements in $(0,1)\cap\Gamma^*$. Then, if 
$F$ is an object in $\LNM_{A^n_L[0,a),\Sigma}$ which is log-convergent, 
$F$ contains a $\Sigma$-constant subobject on $A^n_L[0,a')$. We will 
prove this claim. \par 
Let $F_0$ be the pull-back of $F$ by the map 
$\Spm L \hra A^n_L[0,1)$ of `inclusion into the origin'. Then 
$F_0$ admits commuting endomorphisms $N_i$ ($1 \leq i \leq n$) 
coming from the residues along $t_i=0$. Let $\xi_{ik} \in \Sigma_i$ 
$(1 \leq k \leq n_i)$ be the eigenvalues of $N_i$ and let 
$P_i(x) = \prod_{k=1}^{n_i}(x-\xi_{ik})^{m_{ik}}$ be the minimal 
polynomial of $N_i$. Let us put $m=\max_{i,k}(m_{ik})$. 
For $\k := (k_1,...,k_n)$ $(1 \leq k_i \leq n_i)$, let us define 
$F'_{0,\k} := \{x \in F_0 \,\vert\, \forall i, (N_i-\xi_{ik_i})(x)=0\}$. 
We may assume that $F'_{0,\1} \not= \{0\}$, where we put 
$\1 := (1,...,1)$. Let us take polynomials $Q_i(x)$ $(1 \leq i \leq n)$ 
dividing $P_i(x)$ such that the map $\prod_{i=1}^nQ_i(N_i): F_0 \lra 
F_0$ is a non-zero map whose image is contained in $F'_{0,\1}$. \par 
Now let us define a polynomial $D_l(x)$ with variable $x:=(x_1,...,x_n)$ 
by 
$$ 
D_l(x) := \prod_{i=1}^nD_{l,i}(x_i) := 
\prod_{i=1}^n \left\{ 
Q_i(x_i) \prod_{k=1}^{n_i}\prod_{j=1}^l \left( 
\dfrac{j-(x_i-\xi_{ik})}{j-(\xi_{i1}-\xi_{ik})}\right)^m\right\}. 
$$ 
On the other hand, let us put $P_l(y):=\prod_{j=1}^l \dfrac{(j-1)-y}{j}$. 
Then we have 
\begin{align*}
D_{l,i}(x) & = 
Q_i(x) \prod_{k=1}^{n_i} \left\{ 
\left( \prod_{j=1}^l \dfrac{j}{j-(\xi_{i1}-\xi_{ik})} \right)^m 
\cdot 
P_l(x-(\xi_{ik}+1))^m \right\} \\ 
& = 
Q_i(x) \prod_{k=1}^{n_i} \left\{ 
\left( \prod_{j=1}^l \dfrac{j}{j-(\xi_{i1}-\xi_{ik})} \right)^m 
\cdot 
\left( \sum_{q=0}^l P_{l-q}(-(\xi_{ik}+1))P_q(x)\right)^m \right\}. 
\end{align*}
If we put $c_{il} := \prod_{k=1}^{n_i} 
\left( \prod_{j=1}^l \dfrac{j}{j-(\xi_{i1}-\xi_{ik})} \right)^m$, 
we see that $D_{l,i}(x)$ has the form 
$$ c_{il}\,Q_i(x) \cdot 
(\text{$\Z_p$-linear combination of 
$P_{q_1}(x)\cdots P_{q_m}(x)$  $(0 \leq q_i \leq l)$}). $$
Let us fix some $a'' \in (a',a)\cap\Gamma^*$. Then, 
for any $\eta \in (0,1)$, 
we have the following inequality concerning the norms of operators on 
$A^n_L[0,a'']$: 
{\allowdisplaybreaks{
\begin{align}
& \left| D_l\left(t\dfrac{\pa}{\pa t}\right) \right| \eta^l \label{1eq}\\ 
\leq &
\left| \prod_{i=1}^n Q_i\left(t_i\dfrac{\pa}{\pa t_i} \right) \right| \cdot 
\left| \prod_{i=1}^n c_{il} \right| \cdot 
\prod_{i=1}^n \sup_{0 \leq q_1,...,q_m \leq l} 
\left( \left|P_{q_1}\left(t_i\dfrac{\pa}{\pa t_i}\right)\right| \cdots 
\left|P_{q_m}\left(t_i\dfrac{\pa}{\pa t_i}\right)\right| \right) 
\eta^l \nonumber \\ 
\leq &
(\const)\left| \prod_{i=1}^n c_{il} \right|\eta^{l/2} \, 
\prod_{i=1}^n \sup_{0 \leq q_1,...,q_m \leq l} 
\left( \left|P_{q_1}\left(t_i\dfrac{\pa}{\pa t_i}\right)\right|
\eta^{q_1/2nm} \cdots 
\left|P_{q_m}\left(t_i\dfrac{\pa}{\pa t_i}\right)\right|\eta^{q_m/2nm} 
\right). \nonumber 
\end{align}}}
On the right hand side on \eqref{1eq}, 
$\left| \prod_{i=1}^n c_{il} \right|\eta^{l/2}$ converges to zero 
when $l\to\infty$ and the term 
$\sup_{0 \leq q_1,...,q_m \leq l} 
\left( \left|P_{q_1}\left(t_i\dfrac{\pa}{\pa t_i}\right)\right|
\eta^{q_1/2nm} \cdots 
\left|P_{q_m}\left(t_i\dfrac{\pa}{\pa t_i}\right)\right|\eta^{q_m/2nm} 
\right)$ is known to be bounded above, by the log-convergence of $F$. 
Hence we have shown that 
$\left| D_l\left(t\dfrac{\pa}{\pa t}\right) \right| \eta^l$ 
converges to zero as $l\to\infty$. \par 
If $b \in (0,a'']\cap\Gamma^*$ is small 
enough, $F$ is $\Sigma$-unipotent on $A^n_L[0,b]$. 
(In fact, if $L$ is as in (2) in Proposition \ref{3.4.3}, 
it follows from Lemma \ref{3.2.12}. If $L$ is as in (1) in 
Proposition \ref{3.4.3}, $\Spm L$ admits a finite 
admissible open affinoid covering 
$\Spm L = \bigcup_i \Spm L_i$ such that $F|_{\Spm L_i \times \{0\}}$ 
is free. Then, by Lemma \ref{3.2.12} and Corollaries \ref{3.3.4}, 
\ref{3.3.6}, 
$F|_{A^n_{L_i}[0,b]}$ has the form $\cU_{[0,b]}(F'_i)$ for 
a unique $F'_i \in \LNM_{A^n_{L_i}[0,0],\Sigma}$. Using 
Corollary \ref{3.3.6}, we can glue $F'_i$'s to define an object 
$F'$ in $\LNM_{A^n_L[0,0],\Sigma}$ with $\cU_{[0,b]}(F') = F$.) 
Then, the direct 
calculation similar to the proof of Lemma \ref{3.4.1} shows that 
the image of $D_{l+1}\left(t\dfrac{\pa}{\pa t}\right) - 
D_l\left(t\dfrac{\pa}{\pa t}\right)$ 
on $\Gamma(A^n_L[0,b],F)$ 
is contained in $t_1^l\cdots t_n^l\Gamma(A^n_L[0,b],F)$. Hence the 
image of  $D_{l+1}\left(t\dfrac{\pa}{\pa t}\right) - 
D_l\left(t\dfrac{\pa}{\pa t}\right)$ 
on $\Gamma(A^n_L[0,a''],F)$ is contained in 
$$ \Gamma(A^n_L[0,a''],F) \cap 
t_1^l\cdots t_n^l\Gamma(A^n_L[0,b],F) = 
t_1^l\cdots t_n^l\Gamma(A^n_L[0,a''],F). $$ Hence the operator 
$\dfrac{1}{t_1^l\cdots t_n^l}
\left(D_{l+1}\left(t\dfrac{\pa}{\pa t}\right) - 
D_l\left(t\dfrac{\pa}{\pa t}\right)\right)$ is well-defined 
as an operator on $\Gamma(A^n_L[0,a''],F)$. \par 
Now take any $c \in (0,1)\cap\Gamma^*$ and $\eta \in (0,1)$ with 
$\eta<c<a''$. Then, 
since $\left| D_l\left(t\dfrac{\pa}{\pa t}\right) \right| (\eta/c)^{nl}$ 
converges to zero as $l\to\infty$, 
$\left| D_{l+1}\left(t\dfrac{\pa}{\pa t}\right) - 
D_l\left(t\dfrac{\pa}{\pa t}\right) \right| (\eta/c)^{nl}$ 
converges to zero as $l\to\infty$ on $A^n_L[0,a'']$. 
Hence, for any $\vv \in \Gamma(A^n_L[0,a''],F)$, 
$$ \left| \dfrac{1}{t_1^l\cdots t_n^l}\left(
D_{l+1}\left(t\dfrac{\pa}{\pa t}\right) - 
D_l\left(t\dfrac{\pa}{\pa t}\right)\right)(\vv) \right|\eta^{nl} $$ 
converges to zero on $A^n_L[c,a'']$ as $l\to\infty$. 
(Here we use \cite[3.1.8]{kedlayaI}.) Then, by \cite[3.1.6]{kedlayaI}, 
it converges to zero on $A^n_L[0,a'']$ as $l\to\infty$. 
Since this is true for any $\eta<a''$ as above, we see that, 
for any $\vv \in \Gamma(A^n_L[0,a''],F)$, 
$\left| \left(D_{l+1}\left(t\dfrac{\pa}{\pa t}\right) - 
D_l\left(t\dfrac{\pa}{\pa t}\right)\right)(\vv) \right|$ converges to zero on 
 $A^n_L[0,a']$ as $l\to\infty$. 
Hence the limit $f(\vv) = \lim_lD_l(\vv)$ exists on 
$A^n_L[0,a']$. We see, by the similar argument to the proof of Lemma 
\ref{3.4.1}, that $f(\vv)$ defines an element in 
$H^0_{L,\xi}(A^n_L[0,a'],F)$ (where $\xi:=(\xi_{11},\cdots,\xi_{n1})$) 
and it is non-zero for some $\vv \in \Gamma(A^n_L[0,a''],F)$. \par 
In the case where $L$ is as in (2) in Proposition \ref{3.4.3}, 
we can define a morphism $M_{\xi} \lra F$ in $\LNM_{A^n_L[0,a'],\Sigma}$ 
which sends $1$ to $f(\vv)$ and it is injective because it is non-zero 
and the rank of $M_{\xi}$ is one. So we have shown that $F$ has a 
non-zero $\Sigma$-constant subobject on $A^n_L[0,a']$. 
So the proof of the claim in the first paragraph is done and hence 
the proof of the proposition is finished in this case. \par 
In the case where $L$ is as in (1) in Proposition \ref{3.4.3}, we 
need to work more. Let us put $H := 
H^0_{L,\xi}(A^n_L[0,a'],F)$. Then, to finish the proof of the 
the claim in the first paragraph (and the proof of the proposition), 
it suffices to prove the following claims: \\
\quad \\
{\bf claim 1.} \,\,\, 
$H$ is a finitely generated $L$-module and the canonical map 
$\pi_{a'}^*H \lra F$ (where $\pi_{a'}: A^n_L[0,a'] \lra \Spm L$ is the 
projection) is injective. \\
\quad \\
{\bf claim 2.} \,\,\, 
$\pi_{a'}^*H$ has a natural structure of a $\Sigma$-constant 
log-$\nabla$-module on $A^n_{L}[0,a']$ which is a subobject of $F$. \\
\quad \\
We give a proof of these claims, following the proof of 
the corresponding claims in Proposition \ref{3.4.3}. 
Remember that we have $b \in (0,a''] \cap \Gamma^*$ such that 
$F|_{A^n_L[0,b]}$ has the form $\cU_{[0,b]}(F')$ for some 
$F' \in \LNM_{A^n_L[0,0],\Sigma}$ (and so it is $\Sigma$-unipotent). 
Let $\pi_b: A^n_L[0,b] \lra \Spm L$ be the projection and for a 
finitely generated $L$-submodule $M$ of $H$, consider the following 
diagram: 
\begin{equation*}
\begin{CD}
\pi_{a'}^* M @>>> F \\ 
@VVV @VVV \\ 
\pi_b^*H^0_{L,\xi}(A^n_L[0,b],F) @>>> 
F |_{A^n_L[0,b]}. 
\end{CD}
\end{equation*}
Then the vertical arrows are injective and we can prove the 
injectivity of the lower horizontal arrow by using the 
expression $F|_{A^n_L[0,b]} = \cU_{[0,b]}(F')$, as in the proof of 
Proposition \ref{3.4.3}. So the map $\pi_{a'}^*M \lra F$
is injective and by using the Noetherianness of $F$, 
we can prove that $H$ is finitely generated and that the map 
$\pi_{a'}^*H \lra F$ is injective in the same way as the proof of 
claim 1 in Proposition \ref{3.4.3}. So we have proved the claim 1. \par 
Then, since $H$ is stable by the action of continuous derivations of 
$\Spm L$ over $\Spm K$, we see that $H$ is locally free 
$\cO_{\Spm L}$-module, and adding the actions of $t_i(\pa/\pa t_i)$
($=$ multiplication by $\xi_{i1}$), we see that 
$H$ defines a $\Sigma$-constant 
object in $\LNM_{A^n_L[0,0],\Sigma}$. Hence 
$\pi_{a'}^*H = \cU_{[0,a']}(H)$ 
has a natural structure of a $\Sigma$-constant 
log-$\nabla$-module on $A^n_{L}[0,a']$ which is a subobject of $F$, 
as desired. So the proof of the claim 2 is also 
finished and we are done. 
\end{proof}

\section{Monodromy of isocrystals and logarithmic extension}

Let us assume given an open immersion $X \hra \ol{X}$ of smooth 
$k$-varieties with 
$Z:=\ol{X}-X = \bigcup_{i=1}^r Z_i$ a simple normal crossing divisor and 
let us denote the log structure on $\ol{X}$ induced by $Z$ by $M$. 
In this section, 
we introduce the notion of `having 
$\Sigma$-unipotent monodromy' for an overconvergent isocrystal 
on $(X,\ol{X})/K$ and prove that, under the assumption that 
$\Sigma = \prod_{i=1}^r \Sigma_i \subseteq \Z_p^r$ is 
$\NID$ and $\NLD$, any overconvergent isocrystal on 
$(X,\ol{X})/K$ having $\Sigma$-unipotent monodromy uniquely extends 
to an isocrystal on log convergent site $((\ol{X},M)/O_K)_{\conv}$ 
with exponents in $\Sigma$. \par 
Also, we can give the following expression of the above result. 
For an overconvergent isocrystal on $(X,\ol{X})/K$, we define 
the notion of the Robba condition (see Definition \ref{defrobba} 
for detail). 
Then, by combining the result stated in the previous paragraph and a little 
argument, we prove the following: If we fix a section 
$\tau: (\Z_p/\Z)^r \lra \Z_p^r$ of the canonical projection, 
any overconvergent isocrystal satisfying 
the Robba condition and the condition $\NLD$ (which stands for 
`non-Liouville difference' of exponents and will be defined below) 
uniquely extends to an isocrystal on log convergent site 
$((\ol{X},M)/O_K)_{\conv}$ with exponents in $\tau(\ol{\Sigma}_{\cE})$, 
where $\ol{\Sigma}_{\cE} \subseteq (\Z_p/\Z)^r$ denotes the set of 
Christol-Mebkhout exponents of $\cE$ which will also be defined below. 
Since the Robba condition is regarded as a $p$-adic analogue of 
the regular singularity of integrable connections, this result can be 
regarded as a $p$-adic analogue of the canonical extension of 
regular singular integrable connections in 
\cite[II 5.4]{deligne}. \par 
{\it Throughout this section, we assume that $K$ is discretely valued.} 
First we recall terminologies on frames in \cite{kedlayaI}. 

\begin{defn}[{\cite[2.2.4]{kedlayaI}}] 
A frame $($or affine frame$)$ is a tuple $(X,\ol{X},P,i,j)$, 
where $X,\ol{X}$ are $k$-varieties, $P$ is a $p$-adic affine formal scheme 
topologically of finite type over $O_K$, $i:\ol{X} \hra P$ is a 
closed immersion over $O_K$, $j: X \hra \ol{X}$ is an open immersion 
over $k$ such that $P$ is formally smooth over $O_K$ on a neighborhood 
of $X$. We say that the frame encloses a pair $(X,\ol{X})$. 
\end{defn} 

\begin{defn}[{\cite[4.2.1]{kedlayaI}}]
A small frame is a frame $(X,\ol{X},P,i,j)$ such that 
$\ol{X}$ is isomorphic to $P_k$ via $i$ and that there exists 
an element $f \in \Gamma(\ol{X},\cO_{\ol{X}})$ with $X=\{f\not=0\}$. 
\end{defn}

Before we introduce some more terminologies on frames, we introduce 
one terminology which is not standard: For a scheme $Z$, 
{\it a decomposition of $Z$ into irreducible components or empty schemes} 
is a decomposition $Z = \bigcup_{i\in I}Z_i$ such that 
$Z = \bigcup_{\scriptstyle i \in I \atop \scriptstyle Z_i \not= \emptyset} 
Z_i$ gives the decomposition of $Z$ into irreducible components. 

\begin{defn} 
Let $X \hra \ol{X}$ be an open immersion of smooth $k$-varieties such that 
$Z:=\ol{X}-X$ is a simple normal crossing divisor. Let 
$Z = \bigcup_{i=1}^rZ_i$ be a decomposition of $Z$ into irreducible 
components or empty schemes. 
Then a standard small frame enclosing $(X,\ol{X})$ 
is a small frame $\cP := (X,\ol{X},P,i,j)$ enclosing $(X,\ol{X})$ 
which satisfies the following condition$:$ 
There exist $t_1, ..., t_r \in \Gamma(P,\cO_P)$ such that, 
if we denote the zero locus of $t_i$ in $P$ by $Q_i$, 
each $Q_i$ is irreducible $($possibly empty$)$ and that 
$Q = \bigcup_{i=1}^rQ_i$ is a relative simple normal crossing 
divisor of $P$ 
satisfying $Z_i = Q_i \times_P \ol{X}$. 
We call a pair $(\cP, (t_1,...,t_r))$ a charted standard small frame. 
When $r=1$, we call $\cP$ a smooth standard small frame and 
the pair $(\cP,t_1)$ a charted smooth standard small frame. 
\end{defn} 

Next we recall the relation between isocrystals on log convergent site 
and log-$\nabla$-modules. This is explained partly in an abstract way 
in \cite{shiho1}, \cite{shiho2}, \cite{shiho3} and 
explained clearly and explicitly in \cite{kedlayaI}. \par 
Let $X \hra \ol{X}$ be an open immersion of smooth $k$-varieties 
such that $Z:=\ol{X}-X$ is a simple normal crossing divisor and let 
$Z = \bigcup_{i=1}^rZ_i$ be the decomposition of $Z$ into irreducible 
components. Let us denote the log structure on $\ol{X}$ 
induced by $Z$ by $M$. 
Let us take a charted standard small frame 
$((X,\ol{X},P,i,j), (t_1,...,t_r))$ enclosing $(X,\ol{X})$ and 
let $Q_i$ $(1 \leq i \leq r)$ be the zero loci of $t_i$ in $P$. 
Then, by \cite[1.2.7]{shiho2}, 
a locally free isocrystal $\cE$ on the log convergent site 
$((\ol{X},M)/O_K)_{\conv}$ induces in natural way 
a log-$\nabla$-module $E_{\cE}$ on $P_K$ with respect to 
$t_1,...,t_r$. Also, an overconvergent isocrystal $\cE$ on 
$(X,\ol{X})/K$ induces in natural way a $\nabla$-module 
on some strict neighborhood of $]X[_P$ in $P_K$. 
Following \cite[6.3.1]{kedlayaI}, we make the following definition: 

\begin{defn} 
Let $X \hra \ol{X}$ and $(X,\ol{X},P,i,j)$ be as above. 
Then a log-$\nabla$-module $E$ on $P_K$ with respect to 
$t_1, ..., t_r$ is called convergent if the restriction of $E$ 
to some strict neighborhood of $]X[_P$ in $P_K$ 
comes from an overconvergent isocrystal on $(X,\ol{X})/K$. 
\end{defn} 

Then we have the following proposition: 

\begin{prop}[{\cite[6.4.1]{kedlayaI}}] 
Let the notations be as above. Then the functor $\cE \mapsto E_{\cE}$ 
induces an equivalence of categories 
$$ 
\left( \begin{aligned} 
& \text{{\rm locally free isocrystals}} \\ 
& \text{{\rm on $((\ol{X},M)/O_K)_{\conv}$}} 
\end{aligned} 
\right) \os{=}{\lra} 
\left( \begin{aligned} 
& \text{{\rm convergent log-$\nabla$-modules}} \\ 
& \text{{\rm on $P_K$ w.r.t. $t_1,...,t_r$}} 
\end{aligned}
\right). $$
\end{prop} 

Keep the above notation. (In particular, we take an isocrystal 
$\cE$ on $((\ol{X},M)/O_K)\allowbreak {}_{\conv}$ and denote the associated 
log-$\nabla$-module by $E_{\cE}$.) Then, by \cite[6.3.4]{kedlayaI}, 
the multisequence 
$$ \left\{ \dfrac{1}{i_1! \cdots i_r!} \left( 
\prod_{j=1}^r \prod_{l=1}^{i_j-1} \left( 
t_j \dfrac{\pa}{\pa t_j} -l \right)\right)(\vv) \right\}_{i_1,...,i_r} $$ 
is $\eta$-null for any $\vv \in \Gamma(P_K,E_{\cE})$. 
For a subset $I$ of $\{1,...,r\}$, let us put 
$Z_I := \bigcap_{i \in I} Z_i, Q_I := \bigcap_{i \in I}Q_i$. 
Then, if $Z_I$ is non-empty, 
the sections $t_i \,(i \in I)$ induce the isomorphism 
$]Z_I[_P \,\cong\, ]Z_I[_{Q_I} \times A^{|I|}_K[0,1)$ 
(where the coordinate of $A^{|I|}_K[0,1)$ is given by $t_i \,(i \in I)$) 
and for any $\vv \in \Gamma(P_K,E_{\cE})$ and for any 
$\eta \in (0,1), a \in (0,1)\cap\Gamma^*$, the multisequence 
$$ \left\{ \dfrac{1}{\prod_{j \in I}i_j!} \left( 
\prod_{j\in I} \prod_{l=1}^{i_j-1} \left( 
t_j \dfrac{\pa}{\pa t_j} -l \right)\right)(\vv) \right\}_{i_j \,(j \in I)} $$ 
is $\eta$-null on $]Z_I[_{Q_I} \times A^{|I|}_K[0,a]$. 
Hence we have the following: 

\begin{prop}
Let the notations be as above. Then, for any $I \subseteq \{1,...,r\}$, 
the restriction of $E_{\cE}$ to $]Z_I[_P \,\cong\, 
]Z_I[_{Q_I} \times A^{|I|}_K[0,1)$ is log-convergent. 
\end{prop}

Next we define the notion of `having exponents in $\Sigma$' for 
an isocrystal on log convergent site. \par 

\begin{defn}\label{defsigexpo}
Let $X \hra \ol{X}$ be an open immersion of smooth $k$-varieties 
such that $Z:=\ol{X}-X$ is a simple normal crossing divisor and let 
$Z = \bigcup_{i=1}^r Z_i$ be the decomposition of $Z$ by irreducible 
components. Let us denote the log structure on $\ol{X}$ 
induced by $Z$ by $M$. Let $\Sigma = \prod_{i=1}^r\Sigma_i$ be a 
subset of $\Z_p^r$. 
Then we say that a locally free isocrystal $\cE$ on $((\ol{X},M)/O_K)_{\conv}$ 
has exponents in $\Sigma$ if there exist an affine open covering 
$\ol{X} = \bigcup_{\alpha\in \Delta} \ol{U}_{\alpha}$ and 
charted standard small frames 
$(
(U_{\alpha},\ol{U}_{\alpha},P_{\alpha},i_{\alpha},j_{\alpha}), 
(t_{\alpha,1}, ..., t_{\alpha,r}))$ enclosing 
$(U_{\alpha},\ol{U}_{\alpha})$ $(\alpha \in \Delta$, where we put 
$U_{\alpha} := X \cap \ol{U}_{\alpha})$ such that, for 
any $\alpha \in \Delta$ and any $i$ $(1 \leq i \leq r)$, 
all the exponents of the log-$\nabla$-module $E_{\cE, \alpha}$ 
on $P_{\alpha,K}$ induced by $\cE$ along the locus $\{t_{\alpha,i}=0\}$ 
are contained in $\Sigma_i$. 
\end{defn}

Then we have the following: 

\begin{lem} 
Let $(X,\ol{X}), Z,\Sigma$ be as above. Then 
a locally free isocrystal $\cE$ on $((\ol{X},M)/O_K)_{\conv}$ 
has exponents in $\Sigma$ if and only if the following condition 
is satisfied$:$ For any 
affine open subscheme $\ol{U} \hra \ol{X}$ and 
any charted standard small frames 
$((U,\ol{U},P,i,j), 
(t_{1}, ..., t_{r}))$ enclosing 
$(U,\ol{U})$ $($where we put 
$U:= X \cap \ol{U})$ such that, for any $i$ $(1 \leq i \leq r)$, 
all the exponents of the log-$\nabla$-module $E_{\cE}$ 
on $P_{K}$ induced by $\cE$ along the locus $\{t_{i}=0\}$ 
are contained in $\Sigma_i$. 
\end{lem} 

\begin{proof} 
Let $\cE$ be a locally free isocrystal $\cE$ on 
$((\ol{X},M)/O_K)_{\conv}$ with exponents in $\Sigma$ and take 
$((U,\ol{U},P,i,j), (t_{1}, ..., t_{r}))$, $E_{\cE}$ as in the 
statement of the lemma. It suffices to prove that the exponents of 
$E_{\cE}$ along the locus $\{t_{i}=0\}$ are contained in $\Sigma_i$. 
Note that we may work Zariski locally on $P$. \par 
Let $\ol{X} = \bigcup_{\alpha \in \Delta} \ol{U}_{\alpha}$ and take 
$(
(U_{\alpha},\ol{U}_{\alpha},P_{\alpha},i_{\alpha},j_{\alpha}), 
(t_{\alpha,1}, ..., t_{\alpha,r}))$ and $E_{\cE,\alpha}$ 
as in Definition \ref{defsigexpo}. Then, by shrinking $P$, we may assume 
that $\ol{U}$ is contained in $\ol{U}_{\alpha}$ for some $\alpha$. 
By shrinking $P_{\alpha}$, we see that we may assume 
$\ol{U}=\ol{U}_{\alpha}$ to prove the lemma. Moreover, by omitting 
the indices $i$ with $Z_i \cap \ol{U} = \emptyset$, we may assume that 
$Z_i \cap \ol{U} \not= \emptyset$ for any $1 \leq i \leq r$. 
Under this assumption, we
 prove that the exponents of $E_{\cE}$ and those of $E_{\cE,\alpha}$ 
coincide. \par 
Let us consider $P$ (resp. $P_{\alpha}$) as log formal scheme 
with log structure induced by $\{t_1\cdots t_r=0\}$ (resp. 
$\{t_{\alpha,1}\cdots t_{\alpha,r}=0\}$) and consider 
$\ol{U}=\ol{U}_{\alpha}$ as log scheme with log structure induced by 
$\ol{U}\cap Z$. Then we have the following diagram 
\begin{equation}\label{wdeq1}
P_{\alpha,K}=]\ol{U}[_{P_{\alpha}} \os{\pi_1}{\lla}\,
]\ol{U}[^{\log}_{P_{\alpha} \times P} \,\cong\, 
]\ol{U}[_{\hat{{\Bbb A}}^d \times P} \os{\pi_2}{\lra}\, ]\ol{U}[_{P} = 
P_K, 
\end{equation}
where $\pi_i$ $(i=1,2)$ are the maps induced by the projections, 
$d=\dim \ol{U}$ and 
the middle isomorphism is locally defined as follows: 
Let us choose local parameter of the form 
$t_1,...,t_d$ (resp. $t_{\alpha,1}, \cdots t_{\alpha,d}$) 
of $P$ (resp. $P_{\alpha}$), where $t_i, t_{\alpha,i}$ $(1 \leq i \leq r)$ 
are the ones in the data of the charted standard small frames 
taken above. Let 
us put $$ (P_{\alpha} \times P)' := (P_{\alpha} \times P) 
\times_{\Spf O_K\{t_1, ..., t_r, t_{\alpha,1},...,t_{\alpha,r}\}} 
\Spf O_K\{t_1,..., t_r, u_1, ..., u_r\}, $$ 
where the left hand side is 
induced by the map 
$\Spf O_K\{t_1,..., t_r, u_1, ..., u_r\} \lra 
\Spf O_K\{t_1, ..., t_r, t_{\alpha,1},...,t_{\alpha,r}\}$ defined by 
$t_{\alpha,i} \mapsto t_i(1+u_i)$. Then we have 
$]\ol{U}[^{\log}_{P_{\alpha} \times P} = 
]\ol{U}[_{(P_{\alpha} \times P)'}$ by definition and 
$(P_{\alpha} \times P)'$ is formally smooth over $P$ with 
relative coordinate $u_1,...,u_r, t_{r+1}-t_{\alpha,r+1},\cdots, 
t_d-t_{\alpha,d}$. So we have the isomorphism 
$]\ol{U}[^{\log}_{P_{\alpha} \times P} \cong 
]\ol{U}[_{\hat{{\Bbb A}}^d \times P}$ induced by this coordinate, 
which is the middle isomorphism in \eqref{wdeq1}. 
Let us take the 
section $s: ]\ol{U}[_{P} \lra ]\ol{U}[_{\hat{{\Bbb A}}^r \times P}$ 
of $\pi_2$ defined by `the inclusion into the origin'. 
Then, by definition, we see that the log-$\nabla$-module $E_{\cE,\alpha}$ 
restricts to $E_{\cE}$ by $\pi_1\circ s$. Moreover, 
by $\pi_1\circ s$, $\dlog t_{\alpha,i}$ is restricted to 
$s^*(\dlog t_i + \dlog (1+u_i)) = \dlog t_i$. Hence we see that the 
residue of $E_{\cE,\alpha}$ along $\{t_{\alpha,i}=0\}$ restricts to 
the residue of $E_{\cE,\alpha}$ along $\{t_i=0\}$ via 
$\pi_1\circ s$. Hence their exponents coincide and so we are done. 
\end{proof} 

Let $X \hra \ol{X}$ be an open immersion of smooth $k$-varieties 
such that $Z:=\ol{X}-X$ is a smooth divisor. 
Suppose we are given a charted smooth standard small frame 
$((X,\ol{X},P,i,j), t)$ enclosing $(X,\ol{X})$ and 
let $Q$ be the zero loci of $t$ in $P$. 
Then, as we saw above, an overconvergent isocrystal $\cE$ on 
$(X,\ol{X})/K$ induces in natural way a $\nabla$-module $E_{\cE}$ 
on some strict neighborhood of $]X[_P$ in $P_K$. In particular, 
it is defined on $\{x \in P_K \,\vert\, |t(x)| \geq \lam\}$ for 
some $\lam \in (0,1)\cap\Gamma^*$. Hence we can restrict 
$E_{\cE}$ to 
$\{x \in P_K \,\vert\, |t(x)| \geq \lam\} \cap ]Z[_P \cong 
]Z[_Q \times A^1_K[\lam,1) = Q_K \times A^1_K[\lam,1)$ if 
$Z$ is non-empty. 
Using this observation, 
we define the notion of `having $\Sigma$-unipotent monodromy' 
for an overconvergent isocrystal as follows: \par

\begin{defn}\label{defsigmon}
Let $X \hra \ol{X}$ be an open immersion of smooth $k$-varieties 
such that $Z:=\ol{X}-X$ is a simple normal crossing divisor.
Let $Z = \bigcup_{i=1}^rZ_i$ be the decomposition of $Z$ into 
irreducible components and let 
$Z_{\sing}$ be the set of singular points of $Z$. 
Let $\Sigma = \prod_{i=1}^r\Sigma_i$ be a 
subset of $\Z_p^r$. Then we say that an overconvergent isocrystal 
$\cE$ on $(X,\ol{X})/K$ has $\Sigma$-unipotent monodromy 
if there exist an affine open covering 
$\ol{X}-Z_{\sing} = \bigcup_{\alpha\in \Delta} \ol{U}_{\alpha}$ and 
charted smooth standard small frames 
$(
(U_{\alpha},\ol{U}_{\alpha},P_{\alpha},i_{\alpha},j_{\alpha}), 
t_{\alpha})$ enclosing 
$(U_{\alpha},\ol{U}_{\alpha})$ $(\alpha \in \Delta$, where we put 
$U_{\alpha} := X \cap \ol{U}_{\alpha})$ such that, for 
any $\alpha \in \Delta$, there exists some $\lam \in (0,1)\cap\Gamma^*$ 
such that the $\nabla$-module $E_{\cE, \alpha}$ associated to 
$\cE$ is defined on $\{x \in P_{\alpha,K} \,\vert\, 
|t_{\alpha}(x)| \geq \lam\}$ and that the restriction of $E_{\cE,\alpha}$ to 
$Q_{\alpha,K} \times A^1_K[\lam,1)$ is $\Sigma_i$-unipotent $($where 
$Q_{\alpha}$ is the zero locus of $t_{\alpha}$ in $P_{\alpha}$ and $i$ 
is any index with $\ol{U}_{\alpha}\cap Z \subseteq Z_i$, which is unique 
if $\ol{U}_{\alpha}\cap Z$ is non-empty. When $\ol{U}_{\alpha} \cap Z$ 
is empty, we regard this last condition as vacuous one.$)$ 
\end{defn} 

Note that the notion of having $\Sigma$-unipotent monodromy depends 
only on the image $\ol{\Sigma}$ of $\Sigma$ in $(\Z_p/\Z)^r$ in the 
sense that $\cE$ has $\Sigma$-unipotent monodromy if and only if 
$\cE$ has $\tau(\ol{\Sigma})$-unipotent monodromy for some (any) 
section $\tau:(\Z_p/\Z)^r \lra \Z_p$ of the canonical projection 
$\Z_p^r \lra (\Z_p/\Z)^r$. (See Remark \ref{bar}.) Hence we will say 
also that $\cE$ has $\ol{\Sigma}$-unipotent monodromy, by abuse of 
terminology. \par 
The notion of having $\Sigma$-unipotent monodromy does not depend 
on the choice of the data chosen in Definition \ref{defsigmon}, under 
some assumption on $\Sigma$. In fact, 
we have the following lemma: 

\begin{lem}\label{wdsigmon}
Let $(X,\ol{X}), Z=\bigcup_{i=1}^rZ_i, \Sigma$ be as above and assume 
that $\Sigma$ is $\NID$ and $\NLD$. 
Then 
an overconvergent isocrystal $\cE$ on $(X,\ol{X})/K$ has 
$\Sigma$-unipotent monodromy 
if and only if the following condition 
is satisfied$:$ 
For any 
affine open subscheme $\ol{U} \hra \ol{X}-Z_{\sing}$ and 
any charted smooth standard small frame 
$((U,\ol{U},P,i,j), t)$ enclosing 
$(U,\ol{U})$ $($where we put 
$U:= X \cap \ol{U})$, there exists some $\lam \in (0,1)\cap\Gamma^*$ 
such that the $\nabla$-module $E_{\cE}$ associated to 
$\cE$ is defined on $\{x \in P_{K} \,\vert\, 
|t(x)| \geq \lam\}$ and that the restriction of $E_{\cE}$ to 
$Q_{K} \times A^1_K[\lam,1)$ is $\Sigma_i$-unipotent $($where 
$Q$ is the zero locus of $t$ in $P$ and $i$ is any index with 
$\ol{U}\cap Z \subseteq Z_i$, which is unique if $\ol{U}\cap Z$ is 
non-empty. When $\ol{U} \cap Z$ 
is empty, we regard this last condition as vacuous one.$)$ 
\end{lem} 

\begin{proof} 
Let $\cE$ be an overconvegent isocrystal on $(X,\ol{X})/K$ having 
$\Sigma$-unipotent monodromy, and let 
$((U,\ol{U},P,i,j), t),Q$ and $E_{\cE}$ (resp. 
$\ol{X}-Z_{\sing} = \bigcup_{\alpha\in \Delta} \ol{U}_{\alpha}$, 
$((U_{\alpha},\ol{U}_{\alpha}, \allowbreak P_{\alpha},i_{\alpha},j_{\alpha}), 
t_{\alpha}), Q_{\alpha}$ and $E_{\cE,\alpha}$) 
as in the statement of the lemma 
(resp. Definition \ref{defsigmon}). It suffices to prove that, 
under the assumption $\ol{U}=\ol{U}_{\alpha}$ for some $\alpha$ and 
the assumption $\ol{U} \cap Z \not= \emptyset$, 
the restriction to $E_{\cE}$ to 
$Q_K \times A^1_K[\lam,1)$ is $\Sigma_i$-unipotent for some $\lam$. 
(Here $i$ is the index with $\ol{U}\cap Z \subset Z_i$.) 
\par 
Let $M$ be the log structure on $\ol{X}$ induced by $Z$. 
By assumption, $E_{\cE,\alpha}$ is defined on 
$\{x \in P_{\alpha,K} \,\vert\, 
|t_{\alpha}(x)| \geq \lam\}$ for some $\lam$ and $\Sigma_i$-unipotent on 
$\{x \in P_{\alpha,K} \,\vert\, 
|t_{\alpha}(x)| \geq \lam\} \cap 
]Z\cap\ol{U}_{\alpha}[_{P_{\alpha}} \cong Q_{\alpha,K} \times A^1_K[\lam,1)$. 
Then, by Corollary 
\ref{3.3.4}, we may suppose that the restriction of 
$E_{\cE,\alpha}$ to $Q_{\alpha,K} \times A^1_K[\lam,1)$ extends to 
a log-$\nabla$-module on $Q_{\alpha,K} \times A^1_K[0,1)$ with exponents 
in $\Sigma_i$. Therefore, $E_{\cE,\alpha}$ extends to a log-$\nabla$-module 
on $P_{\alpha,K}$ with respect to $t_{\alpha}$ with exponents in 
$\Sigma_i$, that is, $\cE$ extends to an isocrystal $\wt{\cE}$ on 
$((\ol{U},M|_{\ol{U}})/O_K)_{\conv}$ with exponents in $\Sigma_i$. 
Then $\wt{\cE}$ induces a 
log-$\nabla$-module $E_{\wt{\cE}}$ on $P_K$ with respect to $t$ 
with exponents in $\Sigma_i$ which extends $E_{\cE}$. 
Then the restriction of $E_{\wt{\cE}}$ to 
$Q_K \times A^1_K[0,1)$ is log-convergent and has exponents in $\Sigma_i$. 
Hence it is $\Sigma_i$-unipotent by Propsition \ref{3.6.2}. 
Hence the restriction of $E_{\wt{\cE}}$ to $Q_K \times A^1_K[\lam,1)$, 
which is nothing but the restriction of $E_{\cE}$ to 
$Q_K \times A^1_K[\lam,1)$ if $\lam$ is sufficiently close to $1$, 
is also $\Sigma_i$-unipotent. So we are done. 
\end{proof} 

We make the following auxiliary definition to prove that 
the property of `having $\Sigma$-unipotent monodromy' is 
generic in some sense:

\begin{defn}\label{defsiggenmono}
Let $X \hra \ol{X}$ be an open immersion of smooth $k$-varieties 
such that $Z:=\ol{X}-X$ is a simple normal crossing divisor.
Let $Z = \bigcup_{i=1}^rZ_i$ be the decomposition of $Z$ by 
irreducible components and let 
$Z_{\sing}$ be the set of singular points of $Z$. 
Let $\Sigma = \prod_{i=1}^r\Sigma_i$ be a 
subset of $\Z_p^r$. Then we say that an overconvergent isocrystal 
$\cE$ on $(X,\ol{X})/K$ has $\Sigma$-unipotent generic monodromy 
if there exist affine open subschemes 
$\ol{U}_{\alpha} \subseteq \ol{X}-Z_{\sing}$ containing 
the generic point of $Z_{\alpha}$  $(1 \leq \alpha \leq r)$, 
charted smooth standard small frames 
$(
(U_{\alpha},\ol{U}_{\alpha},P_{\alpha},i_{\alpha},j_{\alpha}), 
t_{\alpha})$ enclosing 
$(U_{\alpha},\ol{U}_{\alpha})$ $($where we put 
$U_{\alpha} := X \cap \ol{U}_{\alpha})$, 
fields $L_{\alpha}$ containing $\Gamma(Q_{\alpha,K},\cO)$ 
$($where $Q_{\alpha}$ is the zero locus of $t_{\alpha}$ in 
$P_{\alpha})$ which is complete with respect to a norm which 
restricts to the supremum norm on $Q_{\alpha,K}$ satisfying 
the following condition$:$ For 
any $1 \leq \alpha \leq r$, there exists some $\lam \in (0,1)\cap\Gamma^*$ 
such that the $\nabla$-module $E_{\cE, \alpha}$ associated to 
$\cE$ is defined on $\{x \in P_{\alpha,K} \,\vert\, 
|t_{\alpha}(x)| \geq \lam\}$ and that the restriction of $E_{\cE,\alpha}$ 
to $A^1_{L_{\alpha}}[\lam,1)$ is $\Sigma_{\alpha}$-unipotent. 
\end{defn}

Note that the notion of having $\Sigma$-unipotent monodromy depends 
only on the image $\ol{\Sigma}$ of $\Sigma$ in $(\Z_p/\Z)^r$ in the 
sense that $\cE$ has $\Sigma$-unipotent monodromy if and only if 
$\cE$ has $\tau(\ol{\Sigma})$-unipotent monodromy for some (or any) 
section $\tau:(\Z_p/\Z)^r \lra \Z_p$ of the canonical projection 
$\Z_p^r \lra (\Z_p/\Z)^r$. (See Remark \ref{bar}.) Hence, 
as in the case of $\Sigma$-unipotent monodromy, it is allowed to say 
also that $\cE$ has $\ol{\Sigma}$-unipotent generic 
monodromy by abuse of terminology, 
where $\ol{\Sigma}$ is the image of $\Sigma$ in $(\Z_p/\Z)^r$. \par 
As for the relation between monodromy and generic monodromy, 
we have the following: 

\begin{prop}\label{sigmon-siggenmon}
Let $(X,\ol{X}), Z=\bigcup_{i=1}^rZ_i, \Sigma$ be as above and assume 
that $\Sigma$ is $\NID$ and $\NLD$. 
Then an overconvergent isocrystal $\cE$ on $(X,\ol{X})/K$ has 
$\Sigma$-unipotent monodromy if and only if it has 
$\Sigma$-unipotent generic monodromy. 
\end{prop} 

\begin{proof} 
It suffices to prove that `having $\Sigma$-unipotent generic monodromy' 
implies `having $\Sigma$-unipotent monodromy'. Let $\cE$ be 
an overconvergent isocrystal on $(X,\ol{X}) \allowbreak /K$ which has 
$\Sigma$-unipotent generic monodromy, and let us take an 
affine open subscheme $\ol{U} \hra \ol{X}-Z_{\sing}$ and 
a charted smooth standard small frame 
$(\cP:=(U,\ol{U},P,i,j), t)$ enclosing 
$(U,\ol{U})$ $($where $U:= X \cap \ol{U})$ with $\ol{U} \cap Z \not= 
\emptyset$. 
Let $Q$ be the zero locus of $t$ in $P$ and let $i$ be the index with 
$\ol{U}\cap Z \subseteq Z_i$.
Let $E_{\cE}$ be a 
$\nabla$-module on a strict neighborhood of $]U[_P$ in $P_K$ 
induced by $\cE$. It suffices to prove that 
there exists some $\lam \in (0,1)\cap\Gamma^*$ 
such that $E_{\cE}$ 
is defined on $\{x \in P_{K} \,\vert\, 
|t(x)| \geq \lam\}$ and that the restriction of $E_{\cE}$ to 
$Q_{K} \times A^1_K[\lam,1)$ is $\Sigma_i$-unipotent. 
First let us note that, in the case where 
$(\cP, t)$ is equal to 
$(\cP_{\alpha}:=
(U_{\alpha},\ol{U}_{\alpha},P_{\alpha},i_{\alpha},j_{\alpha}), 
t_{\alpha})$ in Definition \ref{defsiggenmono}, this claim is true 
(with $i=\alpha$) by Proposition \ref{3.4.3}. \par 
Now we prove the claim in the previous paragraph in general case. 
Since $\ol{U}\cap Z$ is non-empty, $\ol{U}$ has non-empty 
intersection with some $\ol{U}_{\alpha}$. Then, if we replace $\ol{U}$ by a 
smaller affine open subscheme which is contained also in $\ol{U}_{\alpha}$, 
the claim is true for $(\cP, t)$ by the fact the claim 
is true for the restriction of $(\cP_{\alpha}, t_{\alpha})$ to $\ol{U}$ and 
by Lemma \ref{wdsigmon}. Then we see that the claim is true for general 
$\ol{U}$ and $(\cP, t)$, by Proposition \ref{3.4.3}. 
\end{proof}

Next we introduce the notion of the Robba condition for 
an overconvergent isocrystal. 

\begin{defn}\label{defrobba}
Let $X \hra \ol{X}$ be an open immersion of smooth $k$-varieties 
such that $Z:=\ol{X}-X$ is a simple normal crossing divisor.
Let $Z = \bigcup_{i=1}^rZ_i$ be the decomposition of $Z$ by 
irreducible components and let 
$Z_{\sing}$ be the set of singular points of $Z$. 
Then we say that an overconvergent isocrystal 
$\cE$ on $(X,\ol{X})/K$ satisfies the Robba condition 
if it satisfies the following condition$:$ 
There exist affine open subschemes 
$\ol{U}_{\alpha} \subseteq \ol{X}-Z_{\sing}$ containing 
the generic point of $Z_{\alpha}$  $(1 \leq \alpha \leq r)$, 
charted smooth standard small frames 
$(\cP_{\alpha}:=
(U_{\alpha},\ol{U}_{\alpha},P_{\alpha},i_{\alpha},j_{\alpha}), 
t_{\alpha})$ enclosing 
$(U_{\alpha},\ol{U}_{\alpha})$ $($where we put 
$U_{\alpha} := X \cap \ol{U}_{\alpha})$, 
fields $L_{\alpha}$ containing $\Gamma(Q_{\alpha,K},\cO)$ 
$($where $Q_{\alpha}$ is the zero locus of $t_{\alpha}$ in 
$P_{\alpha})$ which are complete with respect to a norm which 
restricts to the spectral norm on $Q_{\alpha,K}$ 
satisfying the following condition$:$ For 
any $1 \leq \alpha \leq r$, there exists some $\lam \in (0,1)\cap\Gamma^*$ 
such that the $\nabla$-module $E_{\cE, \alpha}$ associated to 
$\cE$ is defined on $\{x \in P_{\alpha,K} \,\vert\, 
|t_{\alpha}(x)| \geq \lam\}$ and that the restriction 
$F_{\cE,\alpha}$ 
of $E_{\cE,\alpha}$ to 
$A^1_{L_{\alpha}}[\lam,1)$ satisfies the Robba condition 
in the sense of {\rm \cite[11.1]{cmsurvey}}. 
\end{defn} 

Let the notation be as in Definition \ref{defrobba} and let $\mu$ be the 
rank of $\cE$. Then, for each $1 \leq \alpha \leq r$, 
we have the notion of an exponent $\Exp(F_{\cE,\alpha})$ 
of $F_{\cE,\alpha}$ 
in the sense of \cite[11.4]{cmsurvey}, which is an element in 
$\Z_p^{\mu}/\os{e}{\sim}$, where $\os{e}{\sim}$ is a certain equivalent 
relation (defined in \cite[10.4]{cmsurvey}). We call an overconvergent 
isocrystal $\cE$ on $(X,\ol{X})/K$ satisfying the Robba condition is 
$\NLD$ if, for each $1 \leq \alpha \leq r$, the difference of any 
two components of any lift of $\Exp(F_{\cE,\alpha})$ to $\Z_p^{\mu}$ is 
$p$-adically non-Liouville. \par 
Let the situation be as above and assume that $\cE$ is an 
overconvergent isocrystal $(X,\ol{X})/K$ which satisfies the 
Robba condition and which is $\NLD$. 
In this case, it is known (\cite[10.5]{cmsurvey}) that 
$\Exp(F_{\cE,\alpha})$ is well-defined as an element of 
$(\Z_p/\Z)^{\mu}$ up to order. So, if we take a representative 
$(a_1,...,a_{\mu}) \in (\Z_p/\Z)^{\mu}$ of $\Exp(F_{\cE,\alpha})$, 
the set $\ol{\Sigma}_{\cE,\alpha}:=\{a_1,...,a_{\mu}\} \subseteq 
\Z_p/\Z$ does not depend on the representative. Let us put 
$\ol{\Sigma}_{\cE} := \prod_{\alpha=1}^r\ol{\Sigma}_{\cE,\alpha} 
\subseteq (\Z_p/\Z)^{r}$. 
Then, by \cite[6-2.6]{cmII} and \cite[12.1]{cmsurvey},  
we see that 
$\cE$ has $\ol{\Sigma}_{\cE}$-unipotent generic monodromy. 
Noting the fact that $\ol{\Sigma}_{\cE}$ is $\NLD$ 
(hence $\tau(\ol{\Sigma}_{\cE})$ is $\NLD$ and $\NLD$ for any 
section $\tau:(\Z_p/\Z)^r \lra \Z_p^r$ of the canonical projection)
and using Proposition \ref{sigmon-siggenmon}, we have the following: 

\begin{prop}\label{robba-sigmon}
Let $X,\ol{X}$ be as above and let $\cE$ be an 
overconvergent isocrystal on $(X,\ol{X})/K$ which satisfies the 
Robba condition and which is $\NLD$. Then, if we define 
$\ol{\Sigma}_{\cE} \subseteq (\Z_p/\Z)^r$ as above, 
$\cE$ has $\ol{\Sigma}_{\cE}$-unipotent monodromy. 
\end{prop} 

\begin{rem} 
On the other hand, we can prove directly that an overconvergent 
isocrystal $\cE$ having $\Sigma$-unipotent generic monodromy 
for some $\Sigma$ which is $\NID$ and $\NLD$ satisfies the Robba 
condition and in this case, $\ol{\Sigma}_{\cE}$ is contained in the 
image of $\Sigma$ in $(\Z_p/\Z)^r$. 
\end{rem} 

Now we state and prove the main result in this paper, 
which says that, under certain condition, 
overconvergent isocrystals having $\Sigma$-unipotent 
monodromy extends uniquely to an isocrystal on log convergent site 
with exponents in $\Sigma$: 

\begin{thm}\label{main}
Let $X \hra \ol{X}$ be an open immersion of smooth $k$-varieties 
such that $Z :=\ol{X}-X$ is a simple normal crossing divisor and let 
$Z=\bigcup_{i=1}^rZ_i$ be the decomposition of $Z$ into irreducible 
components. Let us denote the log structure on $X$ induced by $Z$ by 
$M$. Let $\Sigma := \prod_{i=1}^r\Sigma_i$ be a subset of 
$\Z_p^r$ which is $\NID$ and $\NLD$. Then we have the canonical 
equivalence of categories 
\begin{equation}\label{maineq1}
j^{\d}: \left( 
\begin{aligned}
& \text{{\rm isocrystals on}} \\ 
& \text{$((\ol{X},M)/O_K)_{\conv}$} \\ 
& \text{{\rm with exponents in $\Sigma$}}
\end{aligned}
\right) 
\os{=}{\lra} 
\left( 
\begin{aligned}
& \text{{\rm overcongent isocrystals}} \\ 
& \text{{\rm on $(X,\ol{X})/K$ having}} \\
& \text{{\rm $\Sigma$-unipotent monodromy}} 
\end{aligned}
\right), 
\end{equation}
which is defined by the restriction. 
\end{thm} 

\begin{rem} 
By using the image of $\ol{\Sigma}$ of $\Sigma$ in $(\Z_p/\Z)^r$ 
instead of $\Sigma$, we can express Theorem \ref{main} in the 
following way: Let $X \hra \ol{X}$, $Z=\bigcup_{i=1}^rZ_i$ and $M$ 
be as above. Let 
Let $\ol{\Sigma} := \prod_{i=1}^r \ol{\Sigma}_i$ be a subset of 
$(\Z_p/\Z)^r$ which is $\NLD$. Then, for each section 
$\tau: (\Z_p/\Z)^r \lra \Z_p$ of the canonical projection 
$\Z_p^r \lra (\Z_p/\Z)^r$, we have the canonical 
equivalence of categories 
\begin{equation*}
j^{\d}: \left( 
\begin{aligned}
& \text{{\rm isocrystals on}} \\ 
& \text{$((\ol{X},M)/O_K)_{\conv}$} \\ 
& \text{{\rm with exponents in $\tau(\ol{\Sigma})$}}
\end{aligned}
\right) 
\os{=}{\lra} 
\left( 
\begin{aligned}
& \text{{\rm overcongent isocrystals}} \\ 
& \text{{\rm on $(X,\ol{X})/K$ having}} \\
& \text{{\rm $\ol{\Sigma}$-unipotent monodromy}} 
\end{aligned}
\right), 
\end{equation*}
which is defined by the restriction. 
\end{rem} 

Before the proof of Theorem \ref{main}, we state an immediate 
corollary, which is regarded as a $p$-adic 
analogue of \cite[II 5.4]{deligne}: 

\begin{cor}
Let $X \subseteq \ol{X}$, $Z=\bigcup_{i=1}^rZ_i$ and $M$ 
be as above and let $\cE$ be an overconvergent isocrystal on 
$(X,\ol{X})/K$ which satisfies the Robba condition and which is 
$\NLD$. Then, for each $\tau: (\Z_p/\Z)^r \lra \Z_p$ of the canonical 
projection $\Z_p^r \lra (\Z_p/\Z)^r$, $\cE$ extends uniquely and 
canonically to an isocrystal on $((\ol{X},M)/O_K)_{\conv}$ with exponents in 
$\tau(\ol{\Sigma}_{\cE})$, where $\ol{\Sigma}_{\cE}$ is the subset of 
$(\Z_p/\Z)^r$ defined in the paragraph before Proposition 
\ref{robba-sigmon}. 
\end{cor} 

\begin{proof} 
Immediate from Theorem \ref{main} and Proposition \ref{robba-sigmon}. 
\end{proof} 

Now we give a proof of Theorem \ref{main}: 

\begin{proof}[Proof of Theorem \ref{main}]
First let us see tht the functor is well-defined, that is, 
for an isocrystal $\cE$ on $((\ol{X},M)/O_K)_{\conv}$ with 
exponents in $\Sigma$, the restriction $j^{\d}\cE$ 
of $\cE$ to the category of 
overconvergent isocrystal on $(X,\ol{X})/K$ has $\Sigma$-unipotent 
monodromy. To see this, let us take 
an affine open subscheme $\ol{U} \hra \ol{X}-Z_{\sing}$ and 
a charted smooth standard small frame 
$((U,\ol{U},P,i,j), t)$ enclosing 
$(U,\ol{U})$ $($where 
$U:= X \cap \ol{U})$ with $\ol{U} \cap Z \not= \emptyset$. 
Let $Q$ be the zero locus of $t$ in $P$. 
Then $\cE$ induces a log-$\nabla$-module $E_{\cE}$ on $P_K$ 
with respect to $t$ whose exponents along $Q_K$ is in $\Sigma_i$, 
where $i$ is the index with $\ol{U}\cap Z \subseteq Z_i$. 
Then the restriction of $E_{\cE}$ to 
$]\ol{U}\cap Z [_P = Q_{K} \times 
A^1_K[0,1)$ has exponents in $\Sigma_i$ and it 
is log-convergent. Hence it is $\Sigma$-unipotent by Proposition 
\ref{3.6.2} and so is the restriction of it to 
$Q_{K} \times A^1_K[\lam,1)$ for any $\lam \in (0,1)\cap\Gamma^*$. 
So $j^{\d}\cE$ has $\Sigma$-unipotent monodromy. \par 
Next we see that, if we are given a morphism $f: \cE \lra \cF$ of 
overconvergent isocrystals on $(X,\ol{X})/K$ having 
$\Sigma$-unipotent monodromy and if we have extensions 
$\wt{\cE}, \wt{\cF}$ of $\cE,\cF$ to isocrystals on 
$((\ol{X},M)/O_K)_{\conv}$ with exponents in $\Sigma$, 
$f$ extends uniquely to the morphism $\wt{f}: \wt{\cE} \lra 
\wt{\cF}$. To see this, we may work Zariski locally on $\ol{X}$. 
Hence we may assume that there exists 
a charted standard small frame 
$((X,\ol{X},P,i,j), 
(t_{1}, ..., t_{r}))$ enclosing 
$(X,\ol{X})$. Let $\varphi: E_{\cE} \lra E_{\cF}$ be a morphism of 
$\nabla$-modules on a strict neighborhood of $]X[_P$ in $P_K$ 
induced by $f$ and let $E_{\wt{\cE}}, E_{\wt{\cF}}$ be the 
log-$\nabla$-module on $P_K$ with respect to $t_1,...,t_r$ 
induced by $\wt{\cE},\wt{\cF}$. Then there exists 
some $\lam \in (0,1)\cap\Gamma^*$ such that both 
$E_{\cE},E_{\cF}$ are defined on 
$\fY := \{x \in P_K \,\vert\, \forall i, |t_i(x)| \geq \lam\}.$ 
Let us consider the admissible covering 
$P_K = \bigcup_{I \subseteq \{1,...,r\}} \fX_I$, where 
$\fX_I$ is defined by 
$$ 
\fX_I := \{x \in P_K\,\vert\, |t_i(x)|<1 \,(i \in I), \, |t_i(x)|\geq \lam 
\,(i \notin I)\}. 
$$
This covering induces the admissible covering 
$\fY = \bigcup_{I \subseteq \{1,...,r\}} \fY_I$, where 
$$ 
\fY_I := 
\{x \in P_K\,\vert\, \lam \leq |t_i(x)|<1 \,(i \in I), \, |t_i(x)|\geq \lam 
\,(i \notin I)\}. 
$$ 
By assumption, $E_{\wt{\cE}}, E_{\wt{\cF}}$ have exponents in 
$\Sigma$ and log-convergent on 
$
\{x \in P_K\,\vert\,t_i(x)=0 \,(i \in I)\} \times A^{|I|}[0,1)$. 
Hence they are $\Sigma$-unipotent on it. Hence they are 
$\Sigma$-unipotent on 
$$ \fX_I = 
\{x \in P_K\,\vert\, t_i(x)=0 \,(i \in I), |t_i(x)|\geq \lam 
\,(i \notin I)\} \times A^{|I|}[0,1), $$
which extends the restriction of $E_{\cE}, E_{\cF}$ to 
$$ \fY_I = 
\{x \in P_K\,\vert\, t_i(x)=0 \,(i \in I), |t_i(x)|\geq \lam 
\,(i \notin I)\} \times A^{|I|}[\lam,1). $$
Hence, by Corollary \ref{3.3.6}, there exists a unique 
morphism $\wt{\varphi}: E_{\wt{\cE}} \lra E_{\wt{\cF}}$ on 
$\fX_I$ which extends the map $\varphi$ on $\fY_I$. On 
\begin{align*}
\fX_I \cap \fX_J = 
\{x \in P_K\,\vert\, \lam 
\leq |t_i(x)| < 1 \,& (i \in (I\cup J)-(I\cap J)), \\ 
& \lam \leq |t_i(x)| \,(i \notin I \cup J)\} \times 
A^{|I\cap J|}_K[0,1),
\end{align*} 
the maps $\wt{\varphi}$ 
we constructed on $\fX_I$ and on $\fX_J$ coincide because 
they both extends the map $\varphi$ on 
\begin{align*}
\fY_I \cap \fY_J = 
\{x \in P_K\,\vert\, \lam \leq |t_i(x)| < 1 \,& 
(i \in (I\cup J)-(I\cap J)), \\ 
& \lam \leq |t_i(x)| \,(i \notin I \cup J)\} \times 
A^{|I\cap J|}_K[\lam,1). 
\end{align*}  
Therefore we can glue them and so we have a unique extension 
$\wt{\varphi}: E_{\wt{\cE}} \lra E_{\wt{\cF}}$ of the map $\varphi$ which 
gives the desired map 
$\wt{f}: \wt{\cE} \lra 
\wt{\cF}$. \par 
Finally, we prove the essential surjectivity of the functor 
$j^{\d}$, that is, any overconvergent isocrystal $\cE$ 
on $(X,\ol{X})/K$ having $\Sigma$-unipotent monodromy necessarily 
extends to an isocrystal on $((\ol{X},M)/O_K)_{\conv}$ with 
exponents in $\Sigma$. 
Since we have already shown the uniqueness of the extension, we may work 
locally on $\ol{X}$, that is, we may assume that there exists 
a charted standard small frame 
$((X,\ol{X},P,i,j), 
(t_{1}, ..., \allowbreak t_{r}))$ enclosing 
$(X,\ol{X})$. Let $Q_i$ be the zero locus of $t_i$ in $P$ and 
put $Z^0_i := Z_i-\bigcup_{j\not= i} (Z_i \cap \allowbreak Z_j)$. 
By definition, $\cE$ induces a $\nabla$-module $E:=E_{\cE}$ on 
$\{x \in P_K \,\vert\, \forall j, |t_j(x)| \geq \lam\}$ 
which is $\Sigma$-unipotent on 
$$\{x \in P_K \,\vert\, \forall j, |t_j(x)| \geq \lam\}\,\cap\, ]Z^0_i[_P 
= \{ x \in Q_{i,K} \,\vert\, \forall j\not=i, |t_j(x)|=1\} 
\times A^1_K[\lam,1)$$ 
for some $\lam \in (0,1)\cap \Gamma^*$. \par 
It suffices to prove that $E$ extends to a log-$\nabla$-module 
on $P_K$ with respect to $t_1,...,t_r$ with exponents in $\Sigma$. 
Hence we are reduced to proving that, 
for any $0 \leq a \leq r$, there exists some $\lam$ such that 
$E$ extends to a log-$\nabla$-module 
on $\{x \in P_K\,\vert\, \forall j>a, |t_j(x)| \geq \lam\}$ 
with respect to $t_1,...,t_r$ with exponents in $\Sigma$. 
We prove this claim by induction on $a$. 
(The case $a=0$ is true by assumption.) \par 
Assume that the claim is true for $a-1$ (for some $\lam$). 
For any $\lam' \in [\lam,1)\cap\Gamma^*$, we have an admissible covering 
$P_K = \{x \in P_K\,\vert\, |t_a(x)| \geq \lam'\} \cup 
\{x \in P_K\,\vert\, |t_a(x)| < 1\}$. Hence we have the induced 
admissible coverings 
\begin{align}
& \{x \in P_K \,\vert\, \forall j>a-1, |t_j(x)| \geq \lam'\} 
\label{cov1} \\ 
= & 
\{x \in P_K \,\vert\, \forall j>a-1, |t_j(x)| \geq \lam'\} \cup 
\{x \in P_K \,\vert\, \forall j>a, |t_j(x)| \geq \lam', \lam'\leq |t_a(x)|<1\} 
 \nonumber \\ 
= & 
\{x \in P_K \,\vert\, \forall j>a-1, |t_j(x)| \geq \lam'\} \cup 
(\{x \in Q_{a,K} \,\vert\, \forall j>a, |t_j(x)| \geq \lam'\} 
\times A^1_K[\lam',1)), \nonumber 
\end{align}
\begin{align}
& \{x \in P_K \,\vert\, \forall j>a, |t_j(x)| \geq \lam'\} \label{cov2} \\ 
= & 
\{x \in P_K \,\vert\, \forall j>a-1, |t_j(x)| \geq \lam'\} \cup 
\{x \in P_K \,\vert\, \forall j>a, |t_j(x)| \geq \lam', |t_a(x)|<1\} 
\nonumber \\ 
= & 
\{x \in P_K \,\vert\, \forall j>a-1, |t_j(x)| \geq \lam'\} \cup 
(\{x \in Q_{a,K} \,\vert\, \forall j>a, |t_j(x)| \geq \lam'\} 
\times A^1_K[0,1)). \nonumber  
\end{align}
By induction hypothesis, the log-$\nabla$-module $E$ with exponents in 
$\Sigma$ is defined on 
$\{x \in P_K \,\vert\, \forall j>a-1, |t_j(x)| \geq \lam\}$ and it is 
$\Sigma$-unipotent on 
$\{x \in Q_{a,K} \,\vert\, \forall j\not=a, |t_j(x)|=1\} 
\times A^1_K[\lam,1)$. 
By \eqref{cov1} and \eqref{cov2}, we see that it suffices to 
extend the restriction of $E$ to 
$\{x \in Q_{a,K} \,\vert\, \forall j>a, |t_j(x)| \geq \lam'\} 
\times A^1_K[\lam',1)$ to some log-$\nabla$-module on  
$\{x \in Q_{a,K} \,\vert\, \forall j>a, |t_j(x)| \geq \lam'\} 
\times A^1_K[0,1)$ with exponents in $\Sigma$, for some $\lam' \in 
(\lam,1)\cap\Gamma^*$. \par 
Take any closed aligned subinterval $[b,c] \subseteq (\lam,1)$ of 
positive length. Then, by the $\Sigma$-unipotence of $E$ on 
$\{x \in Q_{a,K} \,\vert\, \forall j\not=a, |t_j(x)|=1\} 
\times A^1_K[\lam,1)$ and 
Proposition \ref{3.5.3}, we see that there exists 
some $\lam' \in (c,1)\cap\Gamma^*$ such that 
$E$ is $\Sigma$-unipotent on 
$$ 
\{x \in Q_{a,K} \,\vert\, \forall j>a, |t_j(x)|\geq \lam', 
\forall j<a, |t_j(x)|=1\} \times A^1_K[b,c]. $$
Then, by Proposition \ref{3.4.3}, we see that 
$E$ is $\Sigma$-unipotent on 
$$ 
\{x \in Q_{a,K} \,\vert\, \forall j>a, |t_j(x)|\geq \lam'\} 
\times A^1_K(b,c). $$
So it can be extended to a log-$\nabla$-module with exponents in $\Sigma$ 
on 
$$ 
\{x \in Q_{a,K} \,\vert\, \forall j>a, |t_j(x)|\geq \lam'\} 
\times A^1_K[0,c). $$ 
Hence we can glue it with $E$ on 
$\{x \in Q_{a,K} \,\vert\, \forall j>a, |t_j(x)| \geq \lam'\} 
\times A^1_K[\lam,1)$ to give a log-$\nabla$-module with exponents 
in $\Sigma$ on 
$$ 
\{x \in Q_{a,K} \,\vert\, \forall j>a, |t_j(x)|\geq \lam'\} 
\times A^1_K[0,1), $$ 
as desired. So the proof of the theorem is finished. 
\end{proof}


\begin{thebibliography}{[11]}

\bibitem{bc0}
   F. ~Baldassarri and B. ~Chiarellotto, 
   {\it On Christol's theorem. A generalization to systems of PDE's 
   with logarithmic singularities depending upon parameters}, 
   Contemporary Math. {\bf 133}, 1--24. 

\bibitem{bc}
   F. ~Baldassarri and B. ~Chiarellotto, 
   {\it Formal and $p$-adic theory of differential 
   systems with logarithmic singularities depending 
   upon parameters}, 
   Duke Math. J. {\bf 72}(1993), 241--300.

\bibitem{bc2} 
   F. ~Baldassarri and B. ~Chiarellotto, 
   {\it Algebraic versus rigid Cohomology with logarithmic coefficients},
   in Barsotti Symposium in Algebraic Geometry, Academic Press.

\bibitem{berthelotrig} 
  P. ~Berthelot, 
  {\it Cohomologie rigide et cohomologie rigide \`{a} supports propres
   \,\,\, premi\`{e}re partie}, pr\'{e}publication de 
   l'IRMAR 96-03. Available at http://perso.univ-rennes1.fr/pierre.berthelot/

\bibitem{chr}
 G. ~Christol, 
 {\it Un th\'eor\`eme de transfert pour les disques singuliers r\'eguliers}, 
Ast\'erisque {\bf 119-120}(1984), 151--168. 

\bibitem{cmII}
   G. ~Christol and Z. ~Mebkhout, 
   {\it Sur le th\'eor\`eme de l'induce des \'equations diff\'erentielles 
    $p$-adiques II}, Ann. of Math. {\bf 146}(1997), 345--410. 

\bibitem{cmsurvey}
    G. ~Christol and Z. ~Mebkhout, 
   {\it \'Equations diff\'erentielles $p$-adiques et coefficients 
    $p$-adiques sur les courbes}, Ast\'erisque {\bf 279}(2002), 
   125--183. 

\bibitem{deligne}
   P. ~Deligne, 
  {\it Equations diff\'erentielles \`a points singuliers r\'eguliers}, 
  Lecture Note in Math. {\bf 163}, Springer-Verlag, 1970. 

\bibitem{dgs}
   B. ~Dwork, G. ~Gerotto and F. ~J. ~Sullivan, 
   {\it An introduction to $G$-functions}, Annals of Math. Studies 
   {\bf 133}, Princeton University Press (1994). 

\bibitem{kedlayaI}
   K. ~S. ~Kedlaya, 
   {\it Semistable reduction for overconvergent $F$-isocrystals, I$:$ 
   Unipotence and logarithmic extensions}, 
   Compositio Math.. {\bf 143}(2007), 1164--1212. 

\bibitem{kedlayaIV}
   K. ~S. ~Kedlaya, 
   {\it Semistable reduction for overconvergent $F$-isocrystals, IV$:$ 
   Local semistable reduction at nonmonomial valuations}, 
   arXiv:0712.3400v3(2009).

\bibitem{kiehl} 
   R. ~Kiehl, 
   {\it Die de Rham Kohomologie algebraischer Manningfaltigkeiten 
   \"{u}ber einem bewerteten K\"{o}rper}, Pub. Math. IHES 
   {\bf 33}(1967), 5--20. 

\bibitem{shiho1}
   A. ~Shiho, 
   {\it Crystalline Fundamental Groups I ---
    Isocrystals on Log Crystalline Site and Log Convergent Site}, 
    J. ~Math. ~Sci. ~Univ. 
    ~Tokyo, {\bf 7}(2000), 509--656.

\bibitem{shiho2}
   A. ~Shiho, 
  {\it Crystalline fundamental groups II --- 
  Log convergent cohomology and rigid cohomology}.
  J. ~Math. ~Sci. ~Univ. ~Tokyo, {\bf 9}(2002),  1--163.

\bibitem{shiho3} 
  A. ~Shiho, 
  {\it Relative log convergent cohomology and relative 
  rigid cohomology I}, arXiv:0707.1742v2(2008). 

\end{thebibliography}
\end{document}